\documentclass[ejsv2,preprint,noshowframe]{imsart}

\RequirePackage[numbers]{natbib}
\RequirePackage[colorlinks,citecolor=blue,urlcolor=blue]{hyperref}
\RequirePackage{graphicx}

\usepackage{multirow}
\usepackage{mathtools}
\usepackage{booktabs}
\usepackage{tikz}
\usepackage{adjustbox}
\usepackage{enumitem}

\newcommand{\norm}[1]{\left\Vert#1\right\Vert}

\newcommand{\abs}[1]{\left\vert#1\right\vert}

\def\argmax{\mathop{\rm argmax}}

\def\E{\mathbb{E}}
\newcommand{\tr}{\operatorname{tr}}

\newcommand\proj{\mathrm{Proj}}
\def\diag{\mbox{diag}}

\def\diag{\mbox{Diag}}

\def\onev{\mathbf 1}

\def\av{\mathbf a}

\def\cv{\mathbf c}

\def\ev{\mathbf e}

\def\qv{\mathbf q}
\def\rv{\mathbf r}

\def\uv{\mathbf u}
\def\vv{\mathbf v}
\def\wv{\mathbf w}
\def\xv{\mathbf x}
\def\yv{\mathbf y}
\def\zv{\mathbf z}

\def\Av{\mathbf A}
\def\Bv{\mathbf B}

\def\Dv{\mathbf D}

\def\Gv{\mathbf G}
\def\Hv{\mathbf H}
\def\Iv{\mathbf I}

\def\Mv{\mathbf M}

\def\Pv{\mathbf P}
\def\Qv{\mathbf Q}
\def\Rv{\mathbf R}
\def\Sv{\mathbf S}
\def\Tv{\mathbf T}
\def\Uv{\mathbf U}
\def\Vv{\mathbf V}

\def\Xv{\mathbf X}
\def\Yv{\mathbf Y}
\def\Zv{\mathbf Z}

\newcommand{\zetav}{\mbox{\boldmath$\zeta$}}

\newcommand{\xiv}{\mbox{\boldmath{$\xi$}}}

\newcommand{\Sigmav}{\mbox{\boldmath{$\Sigma$}}}
\newcommand{\Lambdav}{\mbox{\boldmath{$\Lambda$}}}

\newcommand{\Ac}{\mathcal{A}}

\newcommand{\Cc}{\mathcal{C}}

\newcommand{\Nc}{\mathcal{N}}

\newcommand{\Sc}{\mathcal{S}}

\newcommand{\Uc}{\mathcal{U}}

\newcommand{\Yc}{\mathcal{Y}}

\def\1v{\mathbf 1}
\def\0v{\mathbf 0}

\newcommand{\plim}{\operatorname*{plim}}
\newcommand{\uh}{\hat{\mathbf{u}}}
\newcommand{\vh}{\hat{\mathbf{v}}}

\newcommand{\lambdah}{\hat{\lambda}}
\newcommand{\uhoneJS}{\hat{\mathbf{u}}^{\rm JS}_1}
\newcommand{\uhJS}{\hat{\mathbf{u}}^{\rm JS}}
\newcommand{\utilJS}{\tilde{\mathbf{u}}^{\rm JS}}
\newcommand{\qhJS}{\hat{\mathbf{q}}^{\rm JS}}

\newcommand{\uhiJS}{\hat{\mathbf{u}}^{\rm JS}_i}
\newcommand{\uhiorc}{\hat{\mathbf{u}}^{\rm oracle}_i}
\newcommand{\uhorc}{\hat{\mathbf{u}}^{\rm oracle}}
\newcommand{\qhorc}{\hat{\qv}^\mathrm{oracle}}
\newcommand{\Uchorc}{\hat{\Uc}^{\rm oracle}}
\newcommand{\vspan}{\operatorname{span}}
\newcommand{\Rb}{\mathbb{R}}

\newcommand{\convp}{\xrightarrow{P}}
\newcommand{\convas}{\xrightarrow{{\rm a.s.}}}
\newcommand{\projC}{\proj_\Cc}
\newcommand{\projCiaug}{\proj_{\Cc_{i}^\mathrm{aug}}}

\newcommand{\projS}{\proj_\Sc}

\newcommand{\Uvh}{\hat{\Uv}}

\newcommand{\Uch}{\hat{\Uc}}
\newcommand{\UchJS}{\hat{\Uc}^{\mathrm{JS}}}

\newcommand{\UchHL}{\hat{\Uc}^{\mathrm{HL}}}
\newcommand{\UvhJS}{\hat{\Uv}^{\mathrm{JS}}}
\newcommand{\Uvhorc}{\hat{\Uv}^{\mathrm{oracle}}}

\newcommand{\aiiaug}{a_{ii}^\mathrm{aug}}

\newcommand{\Ciaug}{\Cc_{i}^\mathrm{aug}}

\newcommand\rhoh{\hat{\rho}}

\newcommand{\paren}[1]{\left( #1 \right)}
\newcommand{\curly}[1]{\left\{ #1 \right\}}
\newcommand{\bracket}[1]{\left[ #1 \right]}
\newcommand{\inner}[1]{\left\langle #1 \right\rangle}

\newcommand{\Ebr}[1]{\mathbb{E}\left[ #1 \right]}
\newcommand{\Pbr}[1]{\mathbb{P}\left[ #1 \right]}

\newcommand{\mmi}{[m]\setminus\{i\}}
\newcommand{\Rvmi}{\Rv_{-i}}
\newcommand{\Dvmi}{\Dv_{-i}}
\newcommand{\avmii}{\av_{-i,i}}

\newcommand{\Avmi}{\Av_{-i}}

\startlocaldefs
\theoremstyle{plain}

\newtheorem{theorem}{Theorem}[section]
\newtheorem{lemma}[theorem]{Lemma}
\newtheorem{proposition}[theorem]{Proposition}
\theoremstyle{definition}
\newtheorem{assumption}{Assumption}

\theoremstyle{remark}

\endlocaldefs

\begin{document}
\begin{frontmatter}
\title{Augmented James--Stein estimation for leading eigenvectors and eigenspaces in high dimensions}
\runtitle{Augmented James--Stein estimation in high dimensions}

\begin{aug}
\author[A]{\fnms{Giheon}~\snm{Seong}\ead[label=e1]{heon1998@snu.ac.kr}},
\author[A]{\fnms{Seungki}~\snm{Hong}\ead[label=e2]{skgaboja@snu.ac.kr}}
\and
\author[A]{\fnms{Sungkyu}~\snm{Jung}\ead[label=e3]{sungkyu@snu.ac.kr}\orcid{0000-0002-6023-8956}}
\address[A]{Department of Statistics,
Seoul National University\printead[presep={,\ }]{e1,e2,e3}}

\runauthor{G. Seong, S. Hong and S. Jung}
\end{aug}

\begin{abstract}
Building on the James–Stein approach to leading eigenvector estimation (Goldberg and Kercheval, \emph{Proc. Natl. Acad. Sci. USA} 120, e2207046120, 2023), we develop a data-adaptive augmented James–Stein shrinkage framework for estimating leading eigenvectors and eigenspaces under a generalized spiked population model, in the high-dimensional regime where the dimension $p$ and sample size $n$ grow proportionally. For each spiked eigenvector, we construct an augmented target subspace that combines auxiliary information, either from domain knowledge or prior information, with the remaining sample spiked eigenvectors. This augmentation allows information shared across the sample spiked components to be exploited while retaining a fully data-driven shrinkage rule. We show that the resulting eigenvector estimator strictly improves upon standard PCA whenever the target subspace contains nonvanishing information about the population eigenvector, while asymptotically reverting to PCA when the target is uninformative. The individual estimators further yield a nested sequence of estimators for all leading spiked eigenspaces, with analogous dominance properties. The proposed estimator also strictly improves upon the existing HDLSS-motivated shrinkage estimator in the proportional high-dimensional regime. Simulation studies demonstrate substantial finite-sample gains and robustness to misspecification of the number of spikes.
\end{abstract}



\end{frontmatter}

\section{Introduction}\label{sec:intro}

Estimating leading eigenvectors is a central task in principal component analysis (PCA). Although sample eigenvectors are consistent in classical large-sample settings \cite{anderson1963asymptotic,davis1977asymptotic}, they can retain nonvanishing estimation error when the dimension $p$ and sample size $n$ grow proportionally. Spiked population models provide a framework for studying this phenomenon \cite{johnstone2001,paul2007asymptotics,yao2015sample}, including generalized models with heterogeneous non-spiked eigenvalues \cite{bai2012sample,dey2019asymptotic,ding2021spiked}.

Consistent estimation of spiked eigenvectors is generally impossible in high-dimensional settings \cite{birnbaum2013minimax}. Additional structural information can improve eigenvector estimation, as illustrated by sparse PCA \cite{zou2018selective}. Another approach uses reference directions whose span defines a \textit{target subspace} $\Cc$ as a form of prior information. Examples include the equal-weight direction and sector-based references in financial applications \cite{Goldberg2020BetterBetas,gurdogan2022multiple}, and pathway indicator vectors as a reference in genomics \cite{kanehisa2000kegg}. Motivated by James--Stein estimation of multivariate means \cite{james1961estimation}, shrinkage toward such references has been developed for the leading eigenvector \cite{shkolnik2022james,goldberg2023james,goldberg2025portfolio}. Multispike extensions estimate the full spiked eigenspace, with asymptotic improvements over PCA established under the high-dimension, low-sample-size (HDLSS) regime \cite{shkolnik2025portfolio,yoon2025adaptive}.

These results concern regimes where $p\to\infty$, $n$ remains fixed, and the spikes are of order $p$. We instead consider the proportional regime $p/n\to\gamma\in(0,\infty)$ with a fixed number of spikes of constant magnitude, hereafter called the RMT regime. Our aim is to determine how target information should be incorporated in this regime and to estimate individual spiked eigenvectors together with all their leading eigenspaces. This requires accounting for both the appropriate shrinkage amount and information shared across the spiked components.

We develop an augmented James--Stein estimator based on a deterministic target subspace. The key observation is that distinct population eigenvectors can have nonorthogonal projections onto the target. Augmenting the target with the remaining spiked sample eigenvectors exploits this shared information. Using consistently estimable eigenvector alignments \cite{dey2019asymptotic}, we obtain a fully data-driven shrinkage rule that asymptotically attains the oracle direction within the signal subspace spanned by the original target and all spiked sample eigenvectors. Each estimator strictly improves upon PCA in limiting squared alignment when the target is informative and asymptotically reduces to PCA otherwise. Augmentation provides a further improvement over unaugmented shrinkage when cross-eigenvector target information is present. The successive spans yield nested estimators of all leading spiked eigenspaces, and ordered orthogonalization preserves improvements over PCA. For the full spiked eigenspace, the proposed estimator also strictly improves upon the HDLSS estimator under the RMT regime whenever the target is informative.

We further examine spike-number misspecification and increasing target dimension, and establish asymptotic equivalence with the HDLSS estimator for the full spiked eigenspace under HDLSS and ultra-high-dimensional asymptotics. Simulations illustrate the resulting finite-sample gains. Sections~\ref{sec:PCA under GSPM}--\ref{sec:JSE} present the model and methodology. Section~\ref{sec dimension and spike misspecification} examines spike-number misspecification and increasing target dimension. Section~\ref{sec:comparison} compares the proposed estimator with existing multispike estimators under various high-dimensional asymptotic regimes including HDLSS. Simulation studies are presented in Section~\ref{sec:simulation}. Appendix~\ref{apdx:supp_simulation} contains additional simulation results, while Appendix~\ref{apdx:proofs} provides proofs and technical details.

\section{Generalized spiked model and PCA asymptotics} \label{sec:PCA under GSPM}

This section introduces the generalized spiked population model and collects existing results on the asymptotic behavior of sample eigenvalues and eigenvectors that will be used throughout the paper.
 
For each pair $(n,p)$, let $\xv_1,\ldots,\xv_n\in\Rb^p$ be independent and identically distributed random vectors satisfying 
\begin{equation*} \mathbb{E}(\xv_\ell)=\0v, \qquad \operatorname{Cov}(\xv_\ell)=\Sigmav_p, \qquad \ell\in[n],\end{equation*} 
and let $\Xv_n=[\xv_1,\ldots,\xv_n]\in\Rb^{p\times n}$ denote the observed data matrix. We consider the proportional high-dimensional regime, hereafter referred to as the random matrix theory (RMT) regime, in which the dimension $p$ and sample size $n$ diverge at the same order. 

Write the spectral decomposition of the population covariance matrix as
$$ 
\Sigmav_p
=
\Uv_p\Lambdav_p\Uv_p^\top
=
\sum_{j=1}^p\lambda_{j,p}\uv_{j,p}\uv_{j,p}^\top, \qquad \lambda_{1,p}\ge\cdots\ge\lambda_{p,p}\ge0. 
$$  
We work under the generalized spiked population model in \cite{bai2012sample}. Under this model, the population spectrum consists of a fixed number $m$ of leading eigenvalues, called \textit{spikes}, and the remaining  \textit{non-spiked} eigenvalues. The spikes are separated from the limiting non-spiked spectrum, as formalized in Assumption~\ref{ass:generalized spiked population model} below. The non-spiked eigenvalues are described by their empirical spectral distribution
$ {H}_p\coloneq\frac{1}{p-{m}}\sum_{j={m}+1}^p\delta_{\lambda_{j,p}}$,
where $\delta_a$ denotes the point mass at $a$. Johnstone's spiked population model \cite{johnstone2001} is a special case of the generalized model in which all non-spiked eigenvalues are equal:
\begin{equation}
\lambda_{{m}+1,p}=\cdots=\lambda_{p,p}=\tau^2\quad\text{for some}\quad \tau^2>0.
\label{eq:Johnstone spike model}
\end{equation} 

We impose the following conditions on the population spectrum, the distribution of the observations, and the asymptotic regime.

\begin{assumption}
\label{ass:generalized spiked population model}
The number $m$ of spikes is fixed, and the following conditions hold.
\begin{enumerate}
    \item[(a)] $n,p\to\infty$ with $p/n\to\gamma\in(0,\infty)$.

    \item[(b)] The empirical spectral distribution ${H}_p$ of the non-spiked eigenvalues converges weakly to a probability distribution $H$ that is not concentrated entirely at zero and has compact support $\Gamma_H$, and $\lambda_{{m}+1,p}\to\sup\Gamma_H$ as $p\to\infty$.

    \item[(c)] For each $i\in[{m}]$, $\lambda_{i,p}=\lambda_i$ does not depend on $p$, and
    $\lambda_1>\cdots>\lambda_{{m}}> \lambda_+$,       
    where $\lambda_+
    \coloneq
    \inf\left\{
        \lambda>\sup\Gamma_H:\psi'(\lambda)>0
    \right\}$, for $\psi'(\lambda) = 1-\gamma\int\frac{t^2}{(\lambda-t)^2}dH(t)$.

    \item[(d)] The observed data matrix $\Xv_n\in\Rb^{p\times n}$ satisfies
    \begin{equation*}
    \Xv_n=\Uv_p\Lambdav_p^{1/2}\Zv_n, \quad \Zv_n=[\zv_1,\ldots,\zv_n], 
    \end{equation*}
    where $\zv_1,\ldots,\zv_n\in\Rb^p$ are independent and identically distributed random vectors. For each $\ell \in [n]$, the entries of $\zv_\ell$ are independent with mean zero, variance one, and uniformly bounded $(4+\varepsilon)$-th moments for some $\varepsilon>0$.

\end{enumerate}
\end{assumption}

The critical spike threshold $\lambda_+$ also appears in \cite{bai2012sample}, where the asymptotic behavior of the spiked sample eigenpairs is described by the  function $\psi:(\sup\Gamma_H,\infty)\to\Rb$,
$\psi(\lambda) = \lambda+\gamma\lambda\int\frac{t}{\lambda-t}dH(t)$ and its derivative, both depending on  the aspect ratio $\gamma$ and the limiting spectral distribution $H$. Under Johnstone's spiked population model in \eqref{eq:Johnstone spike model}, the critical threshold becomes
$\lambda_+ =\tau^2(1+\sqrt{\gamma})$ \cite{baik2006eigenvalues,paul2007asymptotics}.

Hereafter, we suppress subscripts indicating dependence on $n$ or $p$ when no ambiguity arises. Let $\Sv$ be the sample covariance matrix, with spectral decomposition
\begin{equation*}
\Sv
=
\frac{1}{n}\Xv\Xv^\top
=
\sum_{j=1}^p\lambdah_{j}\uh_{j}\uh_{j}^\top,
\qquad
\lambdah_{1}\ge\cdots\ge\lambdah_{p}\ge0.
\end{equation*}
When $p > n$, the last $p-n$ sample eigenvalues are zero. For each $i\in[m]$, we choose the sign of $\uh_i$ so that
$\inner{\uh_i,\uv_i}\ge0$ and refer to $(\lambdah_i,\uh_i)$ as a \textit{spiked sample eigenpair}.

The following lemma collects standard asymptotic results for spiked sample eigenpairs established in \cite{bai2012sample,yao2015sample,ding2021spiked}, which will be repeatedly used throughout the paper.

\begin{lemma}\label{lem:eigenpair asymptotics RMT}
Suppose Assumption~\ref{ass:generalized spiked population model} holds. For each $i\in[m]$,
\begin{equation*}
\lambdah_i\convp\psi(\lambda_i)>\lambda_i,
\qquad
\inner{\uh_i,\uv_i}\convp\rho_i,
\end{equation*}
where 
\begin{equation*}
\rho_i\coloneq\sqrt{\frac{\lambda_i\psi'(\lambda_i)}{\psi(\lambda_i)}}\in(0,1).
\end{equation*}
Moreover, $\inner{\uh_i,\uv_j}\convp0$ for every fixed $j\ne i$.
\end{lemma}

Lemma~\ref{lem:eigenpair asymptotics RMT}
highlights a fundamental difference between sample eigenvalues and eigenvectors in the RMT regime.  The sample eigenvalue $\lambdah_i$ converges to the deterministic transform $\psi(\lambda_i)>\lambda_i$, and this asymptotic bias can be
consistently corrected; see below.
In contrast, the sample eigenvector remains
inconsistent, with a \textit{nonvanishing asymptotic alignment}
$\rho_i<1$ with its population counterpart. Importantly, however, $\rho_i$ itself can be consistently estimated from the sample spectrum.

For later use, we introduce consistent estimators of both the population spike $\lambda_i$ and the asymptotic eigenvector alignment $\rho_i$. Let $\gamma_n=(p-m)/(n-m)$. For each $i\in[m]$, define the debiased eigenvalue estimator proposed in \cite{bai2012estimation} by
\begin{equation*}
\tilde\lambda_i
:=
\paren{
\frac{1-\gamma_n}{\lambdah_i}
+
\frac{\gamma_n}{p-m}
\sum_{j=m+1}^p
\frac{1}{\lambdah_i-\lambdah_j}
}^{-1},
\end{equation*}
and the corresponding estimator of the asymptotic eigenvector alignment proposed in \cite{dey2019asymptotic} by
\begin{equation}
\rhoh_i
:=
\paren{
1+
\frac{\tilde\lambda_i\gamma_n}{p-m}
\sum_{j=m+1}^p
\frac{\lambdah_j}{(\lambdah_i-\lambdah_j)^2}
}^{-1/2},
\label{eq:rhoh_i def}
\end{equation}
which, under Assumption~\ref{ass:generalized spiked population model}, consistently estimate $\lambda_i$ and $\rho_i$, respectively; see Lemma~\ref{lem:lambda rho estimation} for a verification. 

Therefore, although the sample eigenvector  itself remains inconsistent, the magnitude of its asymptotic alignment $\rho_i$ is consistently estimable. 
This observation is a key ingredient in the shrinkage estimators developed in Section~\ref{sec:JSE}.

\section{James--Stein eigenvector estimation under the RMT regime}\label{sec:JSE}
 In this section, we combine the estimable sampling geometry, discussed in Section~\ref{sec:PCA under GSPM}, with auxiliary
structural information about the population eigenvectors to construct
James--Stein shrinkage estimators.

\subsection{Target subspace}\label{subsec:target subspace}

Suppose auxiliary information about the population eigenvectors is available in the form of a small collection of prespecified reference directions. Their span defines a \textit{target subspace} $\Cc=\Cc_p\subset\Rb^p$, toward which the sample eigenvectors will be shrunk \cite{shkolnik2025portfolio,yoon2025adaptive}. 
The target subspace may encode domain knowledge or other prior information.
For example, in financial applications, the equal-weight direction $\onev=p^{-1/2}(1,1,\ldots,1)^\top$ provides a natural reference direction for a leading eigenvector representing market-wide co-movement \cite{Goldberg2020BetterBetas}, while multiple reference directions can encode sector structure \cite{gurdogan2022multiple}.  More generally, in a transfer-learning setting, leading eigenvectors learned from a related source population or a previous study may provide reference directions for estimating the corresponding eigenvectors in a target population \cite{franks2019shared,mcgrath2026learner}. In the present analysis, such prior information is represented by a deterministic target subspace.

We do not require the target subspace to contain any population eigenvector exactly. Rather, its informativeness is determined by its asymptotic geometry relative to the population spiked eigenvectors. 

\begin{assumption}\label{ass:target}
    The target subspace $\Cc=\Cc_p$ of $\Rb^p$ is deterministic and has fixed
dimension $r$. Furthermore, there exists an $m\times m$ matrix $\Av=[a_{ij}]$ such that
    \begin{equation*}
 \lim_{p\to\infty}\inner{\projC\uv_i,\projC\uv_j}=a_{ij},\quad\forall i,j\in[m].
    \end{equation*}
\end{assumption}  

The matrix $\Av$ is the limiting Gram matrix of the projected population eigenvectors $\{\proj_{\Cc}\uv_i:i\in[m]\}$ and hence summarizes their asymptotic geometry relative to the target subspace. Since orthogonal projection is contractive and $\uv_1,\ldots,\uv_m$ are orthonormal, we have
    $\0v\preceq \Av\preceq \Iv_m$.
Its diagonal entries satisfy
$
a_{ii} = \lim_{p\to\infty}  \|\projC \uv_i\|^2 \in [0,1],
$
so that 
$a_{ii}$ quantifies the amount of target information available for estimating $\uv_i$, and we call the target subspace \textit{informative} for $\uv_i$ if $a_{ii}>0$. 
At one extreme, $\Av=\0v$ means that all population spiked eigenvectors
are asymptotically orthogonal to $\Cc$, so that the target is uninformative
for every spiked eigenvector. At the other extreme, $\Av=\Iv_m$ means that
all population spiked eigenvectors are asymptotically contained in $\Cc$.

For $i\ne j$, the off-diagonal entry $a_{ij}$ measures the limiting overlap between the projections of $\uv_i$ and $\uv_j$ onto $\Cc$. Although the  population eigenvectors themselves are orthogonal, their projections onto $\Cc$ need not be. Thus, a nonzero $a_{ij}$ indicates that $\uv_i$ and $\uv_j$ share directional information within the target subspace. As shown in Section~\ref{subsec:multi spike case}, this shared information can be exploited by using the sample eigenvector $\uh_j$ to improve estimation of $\uv_i$, motivating the augmented target subspaces used in the multi-spiked construction.

The next lemma shows how the population geometry relative to the target subspace is inherited by the sample eigenvectors. Recall that $\rho_i$ is the asymptotic eigenvector alignment factor, satisfying $\inner{\uh_i,\uv_i} \convp \rho_i$.

\begin{lemma}\label{lem:projC uhi RMT}
    Suppose Assumptions~\ref{ass:generalized spiked population model}--\ref{ass:target} hold. For any value of $\Av$,
    \begin{align}
    \norm{\projC\uh_i-\rho_i\projC\uv_i}&\convp 0,\quad\forall i\in[m]\label{eq:projC uhi projC uvi parallel}.
    \end{align}
    Consequently, we have for $i,j\in [m]$,
    \begin{align}
        \norm{\projC\uh_i}^2&=\inner{\projC\uh_i,\uh_i}\convp \rho_i^2a_{ii},\label{eq:norm projC uhi sq}\\
        \inner{\projC\uh_i,\projC\uv_j}&=\inner{\projC\uh_i,\uv_j}\convp \rho_i a_{ij}.\label{eq:inner projC uhi projC uvj}
    \end{align}
\end{lemma}

Equation~\eqref{eq:projC uhi projC uvi parallel} is the key geometric relation. Although $\uh_i$ itself remains inconsistent for $\uv_i$ in the RMT regime, its projection $\projC \hat\uv_i$ on the finite-dimensional target subspace $\Cc$ removes the high-dimensional diffuse component of the sample eigenvector, and is asymptotically equivalent to  $\rho_i\projC \uv_i$.

The two subsequent limits quantify the resulting geometry.
Equation~\eqref{eq:norm projC uhi sq} shows that the observable quantity $\|\projC\uh_i\|^2$ consistently estimates $\rho_i^2a_{ii}$, the target alignment $a_{ii}$ attenuated by the factor $\rho_i^2$. This limit will be used in Section~\ref{subsec:single spike case} to construct a fully data-driven shrinkage parameter.

Equation~\eqref{eq:inner projC uhi projC uvj} describes how the
off-diagonal target geometry appears at the sample level. At first glance, $\uh_j$ carries no asymptotic information about $\uv_i$ for $j\ne i$, since Lemma~\ref{lem:eigenpair asymptotics RMT} gives $\inner{\uh_j,\uv_i} \convp 0$. Nevertheless, Equation~\eqref{eq:inner projC uhi projC uvj} shows that its component within $\Cc$ has the limiting interaction 
$\inner{\projC\uh_j,\uv_i}\convp\rho_j a_{ij}$. Consequently, 
$$
\inner{\uh_j-\projC\uh_j,\uv_i}  \convp -\rho_j a_{ij}.
$$
Thus, when $a_{ij}\ne0$, the target and residual components of $\uh_j$ contain nonvanishing, mutually offsetting information about $\uv_i$. This structure is exploited in Section~\ref{subsec:multi spike case}, where the sample eigenvectors $\uh_j$ for $j\in\mmi$ are used to augment the target subspace for estimating $\uv_i$.

\subsection{Single-spiked case}\label{subsec:single spike case}
We first consider the single-spiked case, $m=1$, which isolates the basic  shrinkage mechanism before extending it to multiple spikes. In this case, the informativeness of the target subspace for $\uv_1$ is completely described by $a_{11}$. 

Following the target-subspace shrinkage approach in \cite{goldberg2025portfolio}, consider, for $t\in[0,1]$,
\begin{equation}\label{eq:uhJSone def}
    \uhJS_1(t) =\frac{\utilJS_1(t)}{\norm{\utilJS_1(t)}},\qquad \utilJS_1(t) = t\projC\uh_1+(1-t)\uh_1.
\end{equation}
Equivalently, 
$\utilJS_1(t) = \projC\uh_1 + (1-t)(\Iv_p-\proj_{\Cc})\uh_1$.
Thus, $t$ controls the amount of shrinkage applied to the component of $\uh_1$ orthogonal to $\Cc$. Setting $t=0$ gives the PCA estimator $\uh_1$, while $t=1$ retains only its component in the target subspace.

We seek a data-driven choice of $t$ that asymptotically maximizes the alignment of $\uhoneJS(t)$ with the population eigenvector $\uv_1$. By 
Lemmas~\ref{lem:eigenpair asymptotics RMT} and \ref{lem:projC uhi RMT},
\begin{equation*} 
\inner{\uh_1,\uv_1}\convp\rho_1, \quad \|\projC\uh_1\|^2\convp\rho_1^2a_{11}, \quad \inner{\projC\uh_1,\uv_1}\convp\rho_1a_{11}.
\end{equation*}
Consequently, when $a_{11}>0$, 
\begin{align} 
\inner{\uh_1^{\rm JS}(t),\uv_1}^2 & =\frac{\curly{t\inner{\projC\uh_1,\uv_1}+(1-t)\inner{\uh_1,\uv_1}}^2}{\norm{\projC\uh_1}^2+(1-t)^2\paren{1-\norm{\projC\uh_1}^2}}\nonumber\\
&\convp \frac{ \rho_1^2\{1-t(1-a_{11})\}^2 }{ \rho_1^2a_{11} + (1-t)^2(1-\rho_1^2a_{11}) }, 
\label{eq:single spike limiting alignment} 
\end{align}
for every fixed $t\in[0,1]$, which is uniquely maximized at 
\begin{equation*}
t^*=\frac{1-\rho_1^2}{1-\rho_1^2a_{11}}\in(0,1].
\end{equation*}
When $a_{11} = 0$, the limiting squared alignment reduces to $\rho_1^2$  for every fixed $t \in [0,1)$, so that shrinkage yields no asymptotic gain over PCA.
Figure~\ref{fig:single_case} illustrates the limiting squared alignment of the shrinkage estimator. The larger $a_{11}$ is, the greater the potential gain from shrinkage.

\begin{figure}[!tp]
    \centering
    \includegraphics[width=\linewidth]{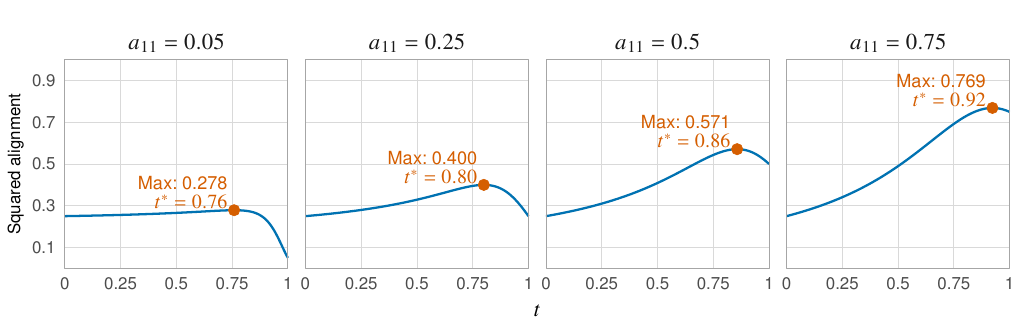}
    \caption{The limiting squared alignment (\ref{eq:single spike limiting alignment}) as a function of $t$, with $\rho_1^2=0.25$ and various values of $a_{11}$. The marked point in each panel is the maximum attained at $t^*$.}
    \label{fig:single_case}
\end{figure}

Substituting $t=t^*$ into \eqref{eq:single spike limiting alignment} yields
$\inner{\uh_1^{\rm JS}(t^*),\uv_1}^2 \convp G_{\rho_{1}}(a_{11})$,
where, for $\rho \in (0,1)$ and $a \in [0,1]$,
\begin{equation*}
    G_\rho(a) = a+\frac{\rho^2(1-a)^2}{1-\rho^2a}.
\end{equation*}
For each fixed $\rho\in(0,1)$, $G_\rho(a)$ is strictly increasing in $a$. Thus, the asymptotic alignment of $\uh_1^{\rm JS}(t^*)$ with $\uv_1$ improves as the target subspace becomes more informative for $\uv_1$. Moreover,
whenever the target subspace is informative, i.e., $a_{11}>0$, the optimally shrunk estimator has strictly larger asymptotic squared alignment with $\uv_1$ than the PCA estimator $\uh_1$.

Although $t^*$ depends on the unknown quantities $\rho_1$ and $a_{11}$, it depends on them only through $\rho_1^2$ and $\rho_1^2a_{11}$. Both quantities are consistently estimable: 
The estimator $\rhoh_1$ defined in \eqref{eq:rhoh_i def} satisfies $\rhoh_1^2\convp\rho_1^2$, while \eqref{eq:norm projC uhi sq} gives $\|\projC\uh_1\|^2 \convp \rho_1^2a_{11}$.  
This suggests the following fully data-driven choice of the shrinkage parameter: 
\begin{equation*}
\hat{t}\coloneq\frac{1-\rhoh_1^2}{1-\norm{\projC\uh_1}^2}.
\end{equation*} 
In the finite sample, the quantity $\hat{t}$ above may not lie in $[0,1]$. We therefore truncate it to the admissible range and, with a slight abuse of notation, use $\hat{t} \leftarrow \Pi_{[0,1]}(\hat{t}) = \min\{1,\max\{0,\hat t\}\}$.

\begin{proposition}\label{prop:oracle single spike}
Suppose Assumptions~\ref{ass:generalized spiked population model}--\ref{ass:target} hold and $m=1$.

\begin{enumerate}[label=(\roman*)]
    \item If $a_{11}>0$, then $
        \hat t\convp t^* = (1-\rho_1^2)/(1-\rho_1^2a_{11}).$ Consequently, the data-driven estimator
    $\uhoneJS(\hat t)$ attains the optimal limiting squared alignment,
\begin{equation*}
    \inner{\uhoneJS(\hat t),\uv_1}^2
    \convp G_{\rho_1}(a_{11}) > \rho_1^2,
\end{equation*}
   whereas
    $\inner{\uh_1,\uv_1}^2\convp\rho_1^2$.
    Thus, $\uhoneJS(\hat t)$ strictly improves upon the PCA estimator $\uh_1$ in asymptotic squared alignment. 
    
    \item If $a_{11}=0$, then
       $\hat t \convp 1-\rho_1^2$, and 
        $\inner{\uhoneJS(\hat t),\uh_1} \convp 1$. Hence, when the target subspace is uninformative, 
    $\uhoneJS(\hat t)$ asymptotically reduces to the standard PCA estimator. 
\end{enumerate}
\end{proposition}

\subsection{Augmented shrinkage in the multi-spiked case}\label{subsec:multi spike case}

We now extend the single-spiked construction to the general $m$-spiked setting. The presence of multiple spikes provides an additional source of information: although $\uh_j$ is asymptotically orthogonal to $\uv_i$ for $j\ne i$, its components relative to the target subspace can still carry information about $\uv_i$, as seen in Equation (\ref{eq:inner projC uhi projC uvj}). We exploit this cross-eigenvector information by augmenting the target subspace for each $\uv_i$ with the remaining spiked sample eigenvectors.

To see how such information arises, fix $i\in[m]$ and consider another spiked sample eigenvector $\uh_j$, $j\ne i$. By Lemma~\ref{lem:eigenpair asymptotics RMT}, $\inner{\uh_j,\uv_i} \convp 0$, whereas Lemma~\ref{lem:projC uhi RMT} gives $\inner{\projC\uh_j,\uv_i} \convp \rho_j a_{ij}$. Consequently, writing for the residual 
$$\rv_j := (\Iv_p-\proj_{\Cc})\uh_j = \uh_j-\projC\uh_j,$$
we have, from Section~\ref{subsec:target subspace}, 
\begin{equation*} 
\inner{\rv_j,\uv_i} \convp -\rho_j a_{ij},\qquad \|\rv_j\|^2
=
1-\|\projC\uh_j\|^2
\convp
1-\rho_j^2a_{jj}.
\end{equation*}
Thus, even though $\uh_j$ as a whole is asymptotically orthogonal to $\uv_i$, its component orthogonal to $\Cc$ may have a nonvanishing alignment with $\uv_i$ whenever $a_{ij}\ne0$.

This observation already shows the benefit of augmenting $\Cc$ with a single sample eigenvector. Indeed, $\Cc+\operatorname{span}\{\uh_j\}=\Cc\oplus\operatorname{span}\{\rv_j\}$, where $\oplus$ denotes the orthogonal direct sum. Lemmas~\ref{lem:eigenpair asymptotics RMT} and \ref{lem:projC uhi RMT} therefore yield
\begin{align} 
\|\proj_{\Cc+ \vspan\{\uh_j\}}\uv_i \|^2 
&= \|\projC\uv_i\|^2 + \| \proj_{\vspan\{\rv_j\}}\uv_i \|^2 \nonumber\\ 
&\convp a_{ii} + \frac{\rho_j^2a_{ij}^2} {1-\rho_j^2a_{jj}}. 
\label{eq:norm of proj uv_i onto C oplus uh_j} 
\end{align}
The second term $\| \proj_{\vspan\{\rv_j\}}\uv_i \|^2 = \langle\rv_j / \| \rv_j\|, \uv_i\rangle^2 = \langle\rv_j,\uv_i\rangle^2 / \|\rv_j\|^2$ encodes the squared alignment of the residual direction with $\uv_i$. It is nonnegative and is strictly positive if and only if $a_{ij}\ne0$. Hence, whenever the projections of $\uv_i$ and $\uv_j$
onto $\Cc$ have nonzero limiting overlap, augmenting $\Cc$ with
$\uh_j$ provides additional asymptotic information about $\uv_i$.

We note that when the dimension $r$ of the target subspace $\Cc$ is smaller than the number $m$ of spikes, and if $a_{jj}> 0$ for more than $r$ indices, their nonzero projections onto the $r$-dimensional subspace $\Cc$ cannot all be asymptotically orthogonal. Consequently, $a_{ij}\ne0$ for at least one pair $i\ne j$.

Motivated by the above observation, for each $i\in[m]$ define the \textit{augmented target subspace} 
 \begin{equation*}
\begin{aligned}
    \Ciaug \coloneq{} & \Cc + \vspan\{\uh_j:j\in \mmi\} \\
    ={} &\Cc \oplus\vspan\{\rv_j:j\in \mmi\}.
\end{aligned}
\end{equation*} 
 
We next characterize how much additional information this full augmentation provides. Let
$$\Rv_{-i} := [\,\rv_j: j\in \mmi\,] = [ \rv_1, \ldots, \rv_{i-1}, \rv_{i+1},\ldots, \rv_m]$$
 which is the collection of residuals relevant for estimating $\uv_i$. As used in \eqref{eq:norm of proj uv_i onto C oplus uh_j}, the squared alignment of the residual directions with $\uv_i$ is $$\| \proj_{\operatorname{col}(\Rv_{-i})} \uv_i\|^2 = \| \Rv_{-i}(\Rv_{-i}^\top \Rv_{-i})^{-1} \Rv_{-i}^\top \uv_i\|^2
= \uv_i^\top \Rv_{-i} (\Rv_{-i}^\top \Rv_{-i})^{-1} \Rv_{-i}^\top \uv_i.$$
It turns out that  $\Rv_{-i}$ is of full rank with probability tending to one, and 
$$\Rv_{-i}^\top\uv_i \convp -\Dv_{-i}\av_{-i,i}, \quad 
\Rv_{-i}^\top \Rv_{-i} \convp \Iv-\Dv_{-i}\Av_{-i}\Dv_{-i},$$
where $\Av_{-i}=(a_{jk})_{j,k\in \mmi}$ is the principal submatrix of $\Av$ (in Assumption~\ref{ass:target}) indexed by $\mmi$,  
$\Dv_{-i} := \diag(\rho_j: j\in \mmi)$ and 
$$\av_{-i,i} := (a_{ji}:j\in \mmi) = (a_{1i}, \ldots, a_{i-1,i}, a_{i+1,i}, \ldots, a_{m,i})^\top \in\Rb^{m-1}.
$$  
Define
\begin{equation}
\aiiaug=a_{ii}+\av_{-i,i}^\top\Dv_{-i}(\Iv-\Dv_{-i}\Av_{-i}\Dv_{-i})^{-1}\Dv_{-i}\av_{-i,i}.
\label{eq:aii aug def} \end{equation}
The second term has the familiar form of the squared norm of a projection: it measures the additional component of $\uv_i$ captured by the residual directions of the remaining sample eigenvectors, after accounting for their mutual nonorthogonality.

The following lemma summarizes the resulting augmented-target geometry.

\begin{lemma}\label{lem:augmented target geometry}
Suppose Assumptions~\ref{ass:generalized spiked population model}--\ref{ass:target} hold. Then,
$\|\proj_{\Ciaug} \uv_i\|^2 \convp \aiiaug$ for each $i\in[m]$, 
where $a_{ii}^{\rm aug}$ satisfies
$a_{ii} \le \aiiaug \le 1$,
and the first inequality is strict 
if and only if $a_{ij} \neq 0$ for some $j \neq i$. Despite the data dependence of $\Ciaug$,
\begin{equation*}
\left\| \proj_{\Ciaug}\uh_i  - \rho_i\proj_{\Ciaug}\uv_i \right\| \convp 0.
\end{equation*}
    Consequently,     
$\|\proj_{\Ciaug}\uh_i\|^2 \convp \rho_i^2 \aiiaug$ and $\langle{\proj_{\Ciaug}\uh_i, \uv_{i}}\rangle \convp \rho_i \aiiaug$.
\end{lemma}

 Thus, the multi-spiked problem has the same asymptotic shrinkage geometry as the single-spiked problem, with the original target $\Cc$ and its alignment $a_{ii}$ replaced by the data-dependent augmented target $\Cc_i^{\rm aug}$ and the enhanced alignment $a_{ii}^{\rm aug}$.

The preceding lemma allows us to apply the single-spiked construction of Section~\ref{subsec:single spike case} with the original target $\Cc$ replaced by $\Ciaug$. In particular, when $\aiiaug>0$, the limit of the squared alignment $\inner{\uh_i^{\rm JS}(t), \uv_i}^2$, where $\uh_i^{\rm JS}(t)$ is defined analogously to \eqref{eq:uhJSone def} with $\Cc$ replaced by $\Cc_i^{\rm aug}$, is uniquely maximized at
\begin{equation} 
t_i^* = \frac{1-\rho_i^2} {1-\rho_i^2\aiiaug}. \label{eq:ti star def}
\end{equation}
In Figure~\ref{fig:augmentation}, we demonstrate that although the value of $t_i^*$ changes little as the cross information $a_{ij}$ increases, the additional information greatly improves the limiting alignment of $\uh_i^{\rm JS}(t_i^*)$ with $\uv_i$. 

\begin{figure}[!tp]
    \centering
    \includegraphics[width=\linewidth]{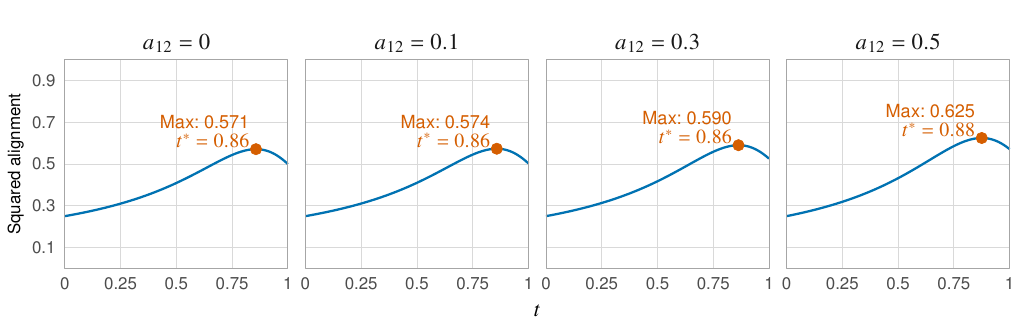}
    \caption{The limiting squared alignment $\inner{\uh_1^{\rm JS}(t), \uv_1}^2$ as a function of $t$, for various values of $a_{12}$, while $m =2$ and $\rho^2 = 0.25$, $a_{11} = a_{22} = 0.5$ are fixed. The marked point in each panel is the maximum attained at $t_1^*$.}
    \label{fig:augmentation}
\end{figure}

The optimal shrinkage parameter $t_i^*$ in \eqref{eq:ti star def} is consistently estimated by
\begin{equation} 
\hat t_i := \Pi_{[0,1]}\frac{1-\rhoh_i^2} {1-\|\proj_{\Ciaug}\uh_i\|^2}. 
\label{eq:ti hat def} \end{equation} 
The proposed James--Stein estimator of $\uv_i$ is 
\begin{equation} \label{eq:uhiJS def} 
\uh_i^{\rm JS} = \uh_i^{\rm JS}(\hat{t}_i) =\frac{\hat t_i\proj_{\Ciaug}\uh_i + (1-\hat t_i)\uh_i} {\norm{\hat t_i\proj_{\Ciaug}\uh_i + (1-\hat t_i)\uh_i}}.
\end{equation}

The construction also admits a useful oracle interpretation. Following the terminology in \cite{yoon2025adaptive}, define the \textit{signal subspace}
\begin{equation}\label{eq:signal subspace def}
    \Sc\coloneq\Cc+\vspan\{\uh_1,\ldots,\uh_m\},
\end{equation}
collecting all information from the original target and the spiked sample eigenvectors. 
Extending the single-spiked oracle construction in \cite{gurdogan2022multiple}, 
define the \textit{oracle estimator} of $\uv_i$ by the best approximation to $\uv_i$ available within the signal subspace: 
\begin{equation*}
    \uhiorc\coloneq  \frac{\projS\uv_i}{\norm{\projS\uv_i}}.
\end{equation*}
The proposed estimator \eqref{eq:uhiJS def}, although fully data-driven, asymptotically attains this oracle direction. 

For comparison, let $\uh_{i,\Cc}^{\rm JS}$ denote the estimator obtained
by applying the same shrinkage procedure directly to the original target
subspace $\Cc$, without augmentation. That is,
\begin{equation*}
\uh_{i,\Cc}^{\rm JS} = 
\frac{
\hat t_{i,\Cc}\projC\uh_i+(1-\hat t_{i,\Cc})\uh_i}{\|\hat t_{i,\Cc}\projC\uh_i+(1-\hat t_{i,\Cc})\uh_i\|},
\qquad
\hat t_{i,\Cc}= \Pi_{[0,1]} \frac{1-\rhoh_i^2}{1-\|\projC\uh_i\|^2}.
\end{equation*}

\begin{theorem}\label{thm:uhJS limit}
Suppose Assumptions~\ref{ass:generalized spiked population model}--\ref{ass:target} hold. Then, for each $i\in[m]$, the following statements hold.
\begin{enumerate}[label=(\roman*)]
    \item $\uhiJS$ asymptotically aligns with the oracle estimator: $\inner{\uhiJS,\uhiorc}\convp 1$.
    \item The asymptotic squared alignments of the PCA, unaugmented
    James--Stein, and augmented James--Stein estimators satisfy
 \begin{equation*}
    \begin{aligned}
        \rho_i^2  =  \plim_{n,p\to\infty} \inner{\uh_i,\uv_i}^2
        &\le \plim_{n,p\to\infty}\inner{\uh_{i,\Cc}^{\rm JS},\uv_i}^2  = G_{\rho_i}(a_{ii})
        \\
        &\le \plim_{n,p\to\infty} \inner{\uh_i^{\rm JS},\uv_i}^2
        =  G_{\rho_i}(\aiiaug).
    \end{aligned}
    \end{equation*}
        The first inequality is strict if and only if $a_{ii}>0$, whereas
    the second inequality is strict if and only if
    $a_{ij}\ne0$ for some $j\ne i$.

    \item If $a_{ii}=0$, then $\inner{\uhiJS,\uh_i}\convp 1$, so that $\uhiJS$ asymptotically reduces to the corresponding sample eigenvector. 
\end{enumerate}
\end{theorem}
Note that, since $\Av\succeq\0v$, $a_{ii}=0$ implies
$a_{ij}=0$ for every $j\in[m]$. Thus, augmentation cannot create information about $\uv_i$ when the original target is asymptotically orthogonal to $\uv_i$. The diagonal entries of $\Av$ determine the gain from target-subspace shrinkage itself, while its off-diagonal entries determine the additional gain made possible by cross-eigenvector augmentation.

\subsection{Spiked eigenspace estimation}\label{subsec:spiked eigenspace estimation}

We now use the individual eigenvector estimators developed in Section~\ref{subsec:multi spike case} to estimate all leading spiked eigenspaces. For each $k\in[m]$, let $\Uc_k \coloneq \operatorname{span}\{\uv_1,\ldots,\uv_k\}$ denote the $k$-dimensional population spiked eigenspace. These eigenspaces are naturally nested, \[ \Uc_1\subset\Uc_2\subset\cdots\subset\Uc_m, \] and our goal is to estimate this entire nested sequence.

To compare $k$-dimensional subspace estimators, for any $k$-dimensional subspace $\tilde\Uc_k$ of $\Rb^p$, let
$\theta_1(\tilde\Uc_k,\Uc_k),\ldots,
\theta_k(\tilde\Uc_k,\Uc_k)$ denote the canonical angles, and define the subspace similarity
\begin{equation}
\label{eq:def subspace similarity}
S(\tilde\Uc_k,\Uc_k)^2
\coloneq
\frac{1}{k}
\sum_{\ell=1}^k
\cos^2\theta_\ell(\tilde\Uc_k,\Uc_k)
= \frac{1}{k} \tr\!
\left( \Pv_{\tilde\Uc_k}\Pv_{\Uc_k}\right),
\end{equation}
where $\Pv_{\Uc}$ is the orthogonal projection matrix onto $\Uc$. The similarity takes values in $[0,1]$ and equals one if and only if the two subspaces coincide. 
By Lemma~\ref{lem:eigenpair asymptotics RMT}, the sample spiked eigenspace $\Uch_k \coloneq
\operatorname{span}\{\uh_1,\ldots,\uh_k\}$ satisfies
\begin{equation*}
S(\Uch_k,\Uc_k)^2
\convp
\frac{1}{k}\sum_{i=1}^k\rho_i^2.
\end{equation*}

For each $k\in[m]$, define the proposed James--Stein estimator of $\Uc_k$ as the span of the corresponding individual eigenvector estimators:
\begin{equation*}
    \UchJS_k\coloneq\vspan\{\uhJS_1,\ldots,\uhJS_k\}.
\end{equation*}
By construction, these estimators are nested:
   $ \UchJS_1\subset\UchJS_2\subset\cdots\subset\UchJS_m$.

The oracle interpretation in Theorem~\ref{thm:uhJS limit} extends naturally to these eigenspaces. Recall the signal subspace $\Sc$ defined in \eqref{eq:signal subspace def}, and define the oracle subspace estimator of $\Uc_k$ by
\begin{equation*}
\Uchorc_k
\coloneq 
\argmax_{\substack{\tilde\Uc_k\subset \Sc\\\dim\tilde\Uc_k=k}} S(\tilde\Uc_k,\Uc_k)
=
\vspan\{\uhorc_1,\ldots,\uhorc_k\}.
\end{equation*}
The second equality is established in Lemma~\ref{lem:similarity tilde Uc}. Since each $\uhJS_i$ asymptotically aligns with its oracle counterpart, the corresponding spans also asymptotically coincide. 

The geometry of these oracle eigenspaces can be summarized by the limiting Gram matrix of $\proj_{\Sc}\uv_i$'s.  Since the signal subspace $\Sc$ is the orthogonal direct sum of the target  $\Cc$ and the residual subspace $\operatorname{col}(\Rv)$, where $\Rv = [\rv_1,\ldots,\rv_m]$ and $\rv_i = (\Iv_p - \projC)\hat\uv_i$, we have
\begin{equation*}
    \begin{aligned}
    \bigl[\inner{\proj_{\Sc}\uv_i,\proj_{\Sc}\uv_j}\bigr]_{i,j\in[m]}  & = \Uv_m^\top\Pv_\Sc\Uv_m  \\
    & = \Uv_m^\top\Pv_\Cc\Uv_m + \Uv_m^\top\Pv_{\operatorname{col}(\Rv)}\Uv_m \convp 
\Bv,    
    \end{aligned}
\end{equation*}
where, for $\Dv:=\diag(\rho_1,\ldots,\rho_m)$, 
\begin{equation*}
    \Bv
    =
    \Av+(\Iv-\Av)\Dv
    (\Iv-\Dv\Av\Dv)^{-1}
    \Dv(\Iv-\Av).
\end{equation*}
Thus, $\Bv$ is the limiting Gram matrix of the projections of the
population spiked eigenvectors onto the signal subspace $\Sc$.
Its diagonal entries recover the individual-eigenvector limits in
Theorem~\ref{thm:uhJS limit}. Since $\uhorc_i=\proj_{\Sc}\uv_i/\|\proj_{\Sc}\uv_i\|$, we have 
$\inner{\uhorc_i,\uv_i}^2    =    \|\proj_{\Sc}\uv_i\|^2$.
Consequently,
\begin{equation*}
    B_{ii}
    =
    G_{\rho_i}(\aiiaug),
    \qquad i\in[m].
\end{equation*}
Hence, the oracle eigenspace satisfies
\begin{align*}
S(\Uchorc_k,\Uc_k)^2
&=
\frac{1}{k}
\sum_{i=1}^k
\|\proj_{\Uchorc_k}\uv_i\|^2
\nonumber\\
&=
\frac{1}{k}
\sum_{i=1}^k
\|\proj_{\Sc}\uv_i\|^2
\convp
\frac{1}{k}\tr(\Bv_{1:k,1:k})
=
\frac{1}{k}
\sum_{i=1}^k
G_{\rho_i}(\aiiaug).
\end{align*}
Here, the second equality follows because
$\proj_{\Sc}\uv_i\in\Uchorc_k$ and
$\uv_i-\proj_{\Sc}\uv_i$ is perpendicular to $\Sc$ for every $i\in[k]$.

For comparison, define the unaugmented James--Stein eigenspace estimator
by
\begin{equation*}
    \UchJS_{k,\Cc}
    :=
    \vspan\{
        \uh_{1,\Cc}^{\rm JS},\ldots,
        \uh_{k,\Cc}^{\rm JS}
    \}.
\end{equation*}
The following theorem summarizes the oracle property of the proposed
eigenspace estimators and compares their asymptotic performance with
PCA and unaugmented James--Stein shrinkage.

\begin{theorem}\label{thm:main thm-1}
Suppose Assumptions~\ref{ass:generalized spiked population model}--\ref{ass:target} hold. Then, for each $k\in[m]$, the following statements hold.
\begin{enumerate}[label=(\roman*)]
    \item $\UchJS_k$ asymptotically aligns with its oracle counterpart $\Uchorc_k$: $
        S(\UchJS_k,\Uchorc_k)\convp 1.
    $

    \item The asymptotic squared similarities of the PCA, unaugmented James--Stein, and augmented James--Stein eigenspace estimators satisfy 
    \begin{equation*}
        \begin{aligned}
   \plim_{n,p\to\infty} S(\Uch_k,\Uc_k)^2 & 
\le
\plim_{n,p\to\infty}
S(\UchJS_{k,\Cc},\Uc_k)^2
\\
&\le
\plim_{n,p\to\infty}
S(\UchJS_k,\Uc_k)^2
=
\frac{1}{k}\sum_{i=1}^k
G_{\rho_i}(\aiiaug).
        \end{aligned}
\end{equation*} 
If $a_{ii}>0$ for some $i\in[k]$, then the first inequality  is strict, and both
$\UchJS_{k,\Cc}$ and $\UchJS_k$ strictly improve upon the PCA estimator $\Uch_k$ in asymptotic squared similarity. If, in addition, $a_{ij} \neq 0$ for some pair $i\ne j \in [k]$, then the last inequality is strict. 

    \item If $a_{ii}=0$ for every $i\in[k]$, then $S(\UchJS_k,\Uch_k)\convp 1$, so that $\UchJS_k$ asymptotically reduces to the corresponding sample eigenspace.

\end{enumerate}
\end{theorem}

The individual James--Stein eigenvector estimators $\uh_1^{\rm JS},\ldots, \uh_m^{\rm JS}$ are not mutually orthogonal in general. This does not affect eigenspace estimation, but an orthonormal frame is useful when individual principal directions are required. The nested nature of $\UchJS_k$ allows us to construct such a frame without changing any of the estimated eigenspaces.
Applying the Gram--Schmidt procedure in the order of the spikes, define for $i\in[m]$,
\begin{equation}
    \qhJS_i = {\qv}_i^{\rm JS}/ \norm{{\qv}_i^{\rm JS}}, \qquad {\qv}_i^{\rm JS} = (\Iv_p - \proj_{\UchJS_{i-1}})\uhJS_i,
    \label{eq:orthoestimator}
\end{equation}
where $\UchJS_0\coloneq\{\0v\}$.  Then
$\qhJS_1,\ldots,\qhJS_m$ are mutually orthonormal and, for every $k\in[m]$, 
$ \vspan\{\qhJS_1,\ldots,\qhJS_k\}
= \UchJS_k$.

The matrix $\Bv$ introduced above also characterizes the effect of this
orthogonalization. For $k\ge2$, define
\begin{equation*}
\beta_k^2 \coloneq 
B_{kk} - \Bv_{k,1:(k-1)}
\Bv_{1:(k-1),1:(k-1)}^{-1}
\Bv_{1:(k-1),k},
\end{equation*}
with $\beta_1^2\coloneq B_{11}$. Thus, $\beta_k^2$ is the Schur complement of $\Bv_{1:(k-1),1:(k-1)}$ in $\Bv_{1:k,1:k}$. 
This leads to a direct geometric interpretation: While $B_{kk}$ measures the squared alignment available for estimating $\uv_k$ before imposing orthogonality to the preceding directions, $\beta_k^2$ measures the amount remaining after that orthogonality constraint is imposed. 

The following result shows that the orthogonalized estimator retains an asymptotic improvement over the corresponding PCA eigenvector.

\begin{proposition}
\label{prop:orthogonalized frame}
Suppose Assumptions~\ref{ass:generalized spiked population model}--\ref{ass:target} hold. Then, for each $i\in[m]$,
\begin{equation*}
\plim_{n,p\to\infty} \inner{\qhJS_i,\uv_i}^2
= \beta_i^2  \ge \rho_i^2 = \plim_{n,p\to\infty}
\inner{\uh_i,\uv_i}^2.
\end{equation*}
The inequality is strict if and only if $a_{ii}>0$.
If $a_{ii}=0$, then
$ \inner{\qhJS_i,\uh_i}  \convp 1$, so that $\qhJS_i$ asymptotically reduces to the corresponding sample eigenvector.
\end{proposition}

\section{Spike-number misspecification and increasing target-subspace dimension}
\label{sec dimension and spike misspecification}

In this section, we examine two departures from the baseline assumptions: misspecification of the number of spikes and a target subspace whose dimension may increase with $p$.

\subsection{Robustness to misspecification of the number of spikes}
\label{subsec:model misspecification}

The preceding analysis assumes that the number of spikes is known. In practice, this number is not available in general. Several consistent estimators are available \cite{passemier2012determining,ke2023estimation}, but finite-sample errors can occur in practice. We therefore examine the robustness of the proposed estimator to misspecification of the number of spikes.

Let $m_0$ denote the true number of spikes and let $m$ denote
the working value used to construct the estimator. We write $\qhJS_k(m)$ for the proposed estimator (\ref{eq:orthoestimator}) constructed using $m$, so that
$\qhJS_k(m_0)$ is its correctly specified counterpart. We focus here on the orthogonalized eigenvector estimators; the corresponding results for the nested eigenspace estimators are given in Appendix~\ref{apdx:proofs}.

For the underspecified case, i.e., $m<m_0$, some informative sample eigenvectors are excluded from the augmented target, potentially reducing the gain from augmentation. On the other hand, for the overspecified case with $m>m_0$, additional sample eigenvectors corresponding to non-spiked components are included. The following proposition summarizes the effects of these two types of misspecification.

\begin{proposition}
\label{prop:number spike misspecification}
Suppose Assumptions~\ref{ass:generalized spiked population model}--%
\ref{ass:target} hold with $m_0$ spikes.

\begin{enumerate}[label=(\roman*)]

\item
Suppose $m<m_0$. Then, for every $i\in[m]$,
\begin{equation}
\plim_{n,p\to\infty} \inner{\uh_i,\uv_i}^2
\le
\plim_{n,p\to\infty} \inner{\qhJS_i(m),\uv_i}^2 \le 
\plim_{n,p\to\infty} \inner{\qhJS_i(m_0),\uv_i}^2. 
\label{eq:underspecification PCA comparison}
\end{equation}
The first inequality is strict if and only if $a_{ii}>0$. If $a_{ii}=0$, then $\langle\qhJS_i(m),\uh_i\rangle\convp1$. 

\item
Suppose $m>m_0$. Additionally assume Gaussianity and that $\Sigmav$ follows Johnstone's spiked population model in \eqref{eq:Johnstone spike model}. Then, for every $i\in[m_0]$,
$ \langle\qhJS_i(m),\qhJS_i(m_0)\rangle
    \convp1.$

Moreover, for every $m_0< i\le m$, $\langle\qhJS_i(m),\uh_i\rangle \convp 1$. 

\end{enumerate}
\end{proposition}

The second inequality of (\ref{eq:underspecification PCA comparison}) can also be strict when the omitted spikes provide additional information for estimating \(\uv_i\). 

Proposition~\ref{prop:number spike misspecification} shows an asymmetry
between under- and overspecification. Underspecification may reduce the
gain from augmentation by omitting informative sample eigenvectors, but
the resulting estimator remains asymptotically no worse than PCA. Under the Gaussian Johnstone model, overspecification is benign: the estimators of the true spikes are asymptotically unchanged, while the additional directions reduce to their PCA counterparts. Thus, when interest lies in a fixed number of leading eigenvectors, a conservatively large working value of $m$ can be used to avoid omitting informative spiked directions.
Analogous robustness properties hold for the corresponding nested eigenspace estimators; see Proposition~\ref{prop:eigenspace spike misspecification}. In particular, underspecification remains asymptotically no worse than PCA, while overspecification preserves the gains attained over the true spiked components.

\subsection{Increasing dimension of the target subspace}
\label{subsec: target dimension}

So far, we have required the target subspace to have fixed dimension. This setting covers many applications motivating our framework, in which the target is generated by a fixed collection of reference directions. More generally, the dimension may increase with $p$ while remaining asymptotically negligible. We show that the basic projection and shrinkage geometry continues to hold when $r_p\coloneq\dim(\Cc_p)\to\infty$ but $r_p/p\to0$. A qualitatively different phenomenon arises only when the target occupies a nonvanishing proportion of the ambient space.

To isolate the effect of the target dimension, we impose a more restrictive model in this subsection. Suppose that $\xv_\ell\sim\Nc(\0v,\Sigmav)$ and that $\Sigmav$ follows Johnstone's spiked model with unit bulk variance:
$$
\lambda_1>\cdots>\lambda_m>1+\sqrt{\gamma},
\qquad
\lambda_{m+1}=\cdots=\lambda_p=1,
$$
where $m$ is fixed and $p/n\to\gamma\in(0,\infty)$. Let $\Cc=\Cc_p$ be deterministic and suppose that
\begin{equation}
\frac{r_p}{p}\to\delta\in[0,1),
\qquad
\Uv_m^\top\Pv_\Cc\Uv_m\to\Av.
\label{eq target assumptions}
\end{equation}
As before, write $\Av=[a_{ij}]$, so that $\norm{\projC\uv_i}^2\to a_{ii}$. That is, although the dimension of $\Cc_p$ may increase, the amount of population signal captured by the target remains asymptotically fixed. The Gaussianity and spherical bulk imposed above ensure rotational invariance of the diffuse component of each spiked sample eigenvector.

 Heuristically, a spiked sample eigenvector may be viewed as $\uh_i \approx \rho_i \uv_i + (1-\rho_i^2)^{1/2} \uv_i^\perp$, where $\uv_i^\perp$ behaves as a uniform random direction in the $(p-m)$-dimensional subspace $\Uc_m^\perp$ \cite{paul2007asymptotics}. A target subspace occupying an asymptotic proportion $\delta$ of the ambient space retains the same proportion of this diffuse component. This gives the following counterpart of Lemma~\ref{lem:projC uhi RMT}.

\begin{lemma}
\label{lem target projected geometry}
Under \eqref{eq target assumptions} and the Gaussian Johnstone model above, for each $i\in[m]$, 
$$\inner{\projC\uh_i,\uv_i}
\convp \rho_i a_{ii}, \quad 
\norm{\projC\uh_i}^2
\convp \rho_i^2a_{ii}+\delta(1-\rho_i^2).$$
\end{lemma}

When $\delta=0$, including the case $r_p\to\infty$ with $r_p=o(p)$, Lemma~\ref{lem:projC uhi RMT} is recovered.  
When $\delta>0$, however, projection onto $\Cc_p$ no longer removes all the diffuse eigenvector noise. In particular, the observable projected norm $\norm{\projC\uh_i}^2$ no longer consistently estimates $\rho_i^2a_{ii}$.

To see the consequence for shrinkage, consider the unaugmented James--Stein estimator 
$$
\uh_{i,\Cc}^{\mathrm{JS}}(t) = \frac{t\projC\uh_i+(1-t)\uh_i
}{\norm{t\projC\uh_i+(1-t)\uh_i}},
\qquad t\in[0,1].
$$
By Lemma~\ref{lem target projected geometry},
\begin{equation*}
\inner{\uh_{i,\Cc}^{\mathrm{JS}}(t),\uv_i}^2
\convp
F_{\rho_i,\delta}(t;a_{ii}),
\end{equation*}
where, for $\rho\in(0,1)$, $\delta\in[0,1)$, and $a\in[0,1]$, 
\begin{equation}
F_{\rho,\delta}(t;a)
\coloneq
\frac{\rho^2\{1-t(1-a)\}^2}
{\rho^2a+\delta(1-\rho^2)+(1-t)^2\{1-\rho^2a-\delta(1-\rho^2)\}},
\label{eq rho delta def}
\end{equation}
with $F_{\rho,\delta}(t;a)$ defined to be zero when $a=\delta=0$ and $t=1$. On $t \in [0,1]$, \eqref{eq rho delta def} is maximized at 
$$t_{\rho,\delta}^*(a)
=
\begin{cases}
\frac{(a-\delta)(1-\rho^2)}
{a\{1-\rho^2a-\delta(1-\rho^2)\}},
& a>\delta,\\[3mm]
0,
& a\le\delta.
\end{cases}
$$
This maximizer is unique unless $a=\delta=0$.

Thus, shrinkage is beneficial only when $a>\delta$. This condition has a natural interpretation. A generic target subspace of dimension approximately $\delta p$ captures an asymptotic proportion $\delta$ of the squared norm of a diffuse direction. Hence, $a\le\delta$ means that the target captures no more population signal than would be expected
from its dimension alone. In this case, the optimal choice is $t=0$, which leaves the PCA estimator unchanged. 

Figure~\ref{fig target alignment} illustrates the limiting squared alignment $F_{\rho,\delta}(t;a)$. 
As $\delta$ increases, both the optimal amount of shrinkage and the gain over the PCA alignment decrease. Thus, enlarging \(\Cc\) with uninformative directions does not improve estimation and may dilute the benefit of shrinkage.  When $a = \delta$ (in the last panel), the optimal choice is $t^*=0$, and the estimator reduces to PCA. 

\begin{figure}[!tbp]
    \centering
    \includegraphics[width=\linewidth]{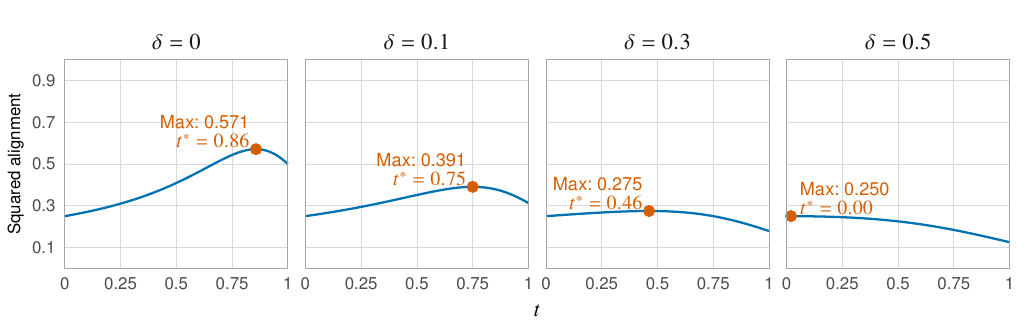}
    \caption{The limiting squared alignment $F_{\rho,\delta}(t;a)$ as a function of $t$, with $\rho^2=0.25$ and $a=0.5$. The marked point in each panel is the maximum attained at $t_{\rho,\delta}^*(a)$.}
\label{fig target alignment}
\end{figure} 

The optimal shrinkage parameter $t^*_{\rho_i,\delta}(a_{ii})$ has a natural plug-in estimator. 
Since we have $\|\projC \uh_i\|^2 \convp \rho_i^2 a_{ii}+\delta(1-\rho_i^2)$, and $\hat\rho_i^2 \convp \rho_i^2$, we set
$$
\hat t_{i,\delta} =
\begin{cases}
\Pi_{[0,1]}\frac{
(\hat a_{ii}-\hat\delta)(1-\hat\rho_i^2)}{\hat a_{ii}\{1-\norm{\projC\uh_i}^2\}}, & \hat a_{ii}>\hat\delta,\\[3mm]
0, & \hat a_{ii}\le\hat\delta.
\end{cases}
$$
Here,  $\hat\delta = r_p / p$, $\hat{a}_{ii} = \{\|\projC \uh_i\|^2 - \hat\delta (1- \hat\rho_i^2) \} / \hat\rho_i^2$.

\begin{proposition}
\label{prop adjusted JS}
Under the conditions of Lemma~\ref{lem target projected geometry}, for each $i\in[m]$, unless $a_{ii}=\delta=0$, $
\hat t_{i,\delta}\convp t_{\rho_i,\delta}^*(a_{ii}).$ Regardless of whether $a_{ii}=\delta=0$,
\begin{equation*}
\plim_{n,p\to\infty}
\inner{\uh_{i,\Cc}^{\mathrm{JS}}(\hat t_{i,\delta}),\uv_i}^2
\ge
\plim_{n,p\to\infty}
\inner{\uh_i,\uv_i}^2.
\end{equation*}
The inequality is strict if and only if $a_{ii}>\delta$.
\end{proposition}
  
The corresponding augmented construction can also be extended to this setting, but its geometry is more involved because the diffuse-noise component contributes to the Gram matrix of the residual directions. We give the resulting augmented-target geometry and shrinkage rule in Section~\ref{subsec:proof increasing target dimension}. Moreover, the dimension-adjusted eigenspace and orthogonalized eigenvector estimators, examined here only numerically, perform similarly to their fixed-dimensional counterparts for small target dimensions but remain stable as the latter deteriorate at larger dimensions; see Section~\ref{subsec:simulation_increasing_r}.

Since the proportional-dimensional target case is less typical in practice, the remainder of the paper returns to the fixed-dimensional target framework.

\section{Comparison with the HDLSS estimator across asymptotic regimes}
\label{sec:comparison}

We now compare the proposed estimator with a line of James--Stein estimators developed under the high-dimension, low-sample-size (HDLSS) framework; see, e.g., \cite{goldberg2023james,shkolnik2022james,shkolnik2025portfolio,yoon2025adaptive}.

We restrict the comparison in this section to the full spiked eigenspace $\Uc_m$, since the existing multispike HDLSS procedure of \cite{shkolnik2025portfolio,yoon2025adaptive} is formulated as an eigenspace estimator rather than as a collection of ordered orthogonal eigenvector estimators.
For the multi-spiked model with a multidimensional target subspace $\Cc$, the corresponding HDLSS estimator is
\begin{equation}\label{eq:UchHL}
    \Uch^\mathrm{HL}_m\coloneq \operatorname{col}\left\{\paren{\Uvh_m\hat{\Lambdav}_m\Uvh^\top_m-\bar{\lambda}\Iv_p}\paren{\Iv_p-\Pv_\Cc}\Uvh_m\right\}.
\end{equation}
Here, $\Uvh_m=[\uh_1,\ldots,\uh_m]$, $\hat\Lambdav_m=\diag(\lambdah_1,\ldots,\lambdah_m)$ and
$\bar\lambda = \{\tr(\Sv)-\sum_{i=1}^m\lambdah_i\}/({n-m})$ estimates the bulk-noise level. 
The estimator \eqref{eq:UchHL}, hereafter called the HL estimator, can be viewed as jointly shrinking the components of the sample spiked eigenspace orthogonal to the target subspace, with the shrinkage determined by the sample eigenvalues and the estimated bulk-noise level; see \cite{yoon2025adaptive} for further discussion.

\subsection{Comparison under the RMT regime}
\label{subsec:comparison HL RMT}

We first show that  the proposed estimator $\UchJS_m$  asymptotically dominates $\UchHL_m$ under 
the RMT regime. 

\begin{proposition}
\label{prop:compare UchHL UchJS RMT}
Suppose Assumptions~\ref{ass:generalized spiked population model}--\ref{ass:target} hold.

\begin{enumerate}[label=(\roman*)]
\item
If $a_{ii}>0$ for some $i\in[m]$, then the proposed estimator strictly
outperforms the HL estimator: 
$
\plim_{n,p\to\infty}
S(\UchJS_m,\Uc_m)
>
\plim_{n,p\to\infty}
S(\UchHL_m,\Uc_m).
$

\item
If $a_{ii}=0$ for every $i\in[m]$, $S(\UchJS_m,\UchHL_m)\convp 1$, so that the two estimators are
asymptotically identical.
\end{enumerate}
\end{proposition}

To better understand the difference between the two estimators, we consider equivalent representations of their column spaces. Recall that $\Rv=(\Iv_p-\Pv_\Cc)\Uvh_m$, and let $\hat\Dv=\diag(\rhoh_1,\ldots,\rhoh_m)$. As shown in Lemma~\ref{lem:estimator equivalent representation}, if the truncation of $\hat t_i$ in \eqref{eq:ti hat def} is omitted, the resulting James--Stein eigenspace and the HL estimator satisfy
\begin{equation}
    \begin{aligned}
          \UchJS_m
    &=
    \operatorname{col}\bracket{
    \Pv_\Cc\Uvh_m+
    \Rv\curly{\Iv_m-(\Rv^\top\Rv)^{-1}(\Iv_m-\hat\Dv^2)}
    }, \\
    \UchHL_m
    &=
    \operatorname{col}\bracket{
    \Pv_\Cc\Uvh_m+
    \Rv
    \paren{\Iv_m-(\Rv^\top\Rv)^{-1} \bar\lambda\hat\Lambdav_m^{-1}}
    }.
\end{aligned}
\label{eq:estimator_equiv_reprn}
\end{equation}
Thus, the two estimators have the same basic form and differ only in the shrinkage matrix applied to the residual component $\Rv$.

This difference becomes particularly transparent by examining the $i$th generating vector. With truncation of $\hat{t}_i$ still omitted, the $i$th generating vector of $\UchJS_m$ is the unnormalized version of the proposed James--Stein eigenvector estimator \eqref{eq:uhiJS def}:
$$\tilde{\uv}^{\rm JS}_i := \uh_i  - \hat{t}_i (\Iv_p-\proj_{\Ciaug})\uh_i,\quad \hat{t}_i = \frac{1-\rhoh_i^2} {1-\|\proj_{\Ciaug}\uh_i\|^2}.$$
Likewise, the $i$th generating vector of $\UchHL_m$ can be written as
$$\tilde{\uv}^{\rm HL}_i := \uh_i  - \hat{t}_i^{\rm HL}(\Iv_p-\proj_{\Ciaug})\uh_i, \quad \hat{t}_i^{\rm HL} = \frac{\bar\lambda / \lambdah_i} {1-\|\proj_{\Ciaug}\uh_i\|^2}. $$
Hence, the HL estimator can also be interpreted as shrinking toward the same augmented target subspace $\Ciaug$, but with a different amount of shrinkage. 
The key distinction under the RMT regime is that $\hat t_i$ consistently estimates the asymptotically optimal shrinkage parameter $t_i^*$ in \eqref{eq:ti star def}, whereas, as shown in the proof of Proposition~\ref{prop:compare UchHL UchJS RMT}, $\hat t_i^{\rm HL}$ converges to a strictly smaller value.
Thus, the HL estimator asymptotically under-shrinks relative to the RMT-optimal rule. When the target subspace is informative, this difference leads to the strict improvement in Proposition~\ref{prop:compare UchHL UchJS RMT}.

\subsection{Comparison under the HDLSS and UHD regimes}\label{sec:other asymptotic regimes}

We next consider whether the difference between the proposed and HL estimators persists in more extreme high-dimensional regimes.
Suppose, for simplicity, that the spiked eigenvalues diverge at a common rate, 
\begin{equation*}
\lambda_{i,p}=\sigma_i^2p^\alpha,\qquad i\in[m],
\end{equation*}
where $\sigma_1^2>\cdots>\sigma_m^2>0$. We consider the following two regimes:
\begin{description}
    \item[(HDLSS)] $\alpha=1$ and $p\to\infty$ with $n$ fixed.
    \item[(UHD)] $\alpha>0$ and $n,p\to\infty$ with $p^{1-\alpha}/n\to c>0$.
\end{description}

Here, UHD stands for the ultra-high-dimensional regime, following \cite{aoshima2018survey}. In both regimes, $\lambda_{i,p}\asymp p/n$, so that the population spikes and the sample noise eigenvalues are of comparable order. This scaling yields nonvanishing PCA estimation error while retaining nontrivial alignment with the population spiked eigenspace. The asymptotic behavior of sample eigenpairs in these and related settings has been studied under HDLSS asymptotics \cite{Jung2009a,Jung2012a} and UHD asymptotics \cite{Yata2012,yata2013pca,lee2014convergence,wang2017asymptotics}; broader frameworks connecting different asymptotic regimes are developed in \cite{shen2016general,shen2016statistics}.

We impose the following counterpart of
Assumption~\ref{ass:generalized spiked population model}. We retain the data representation
\begin{equation*}
\Xv=\Uv\Lambdav^{1/2}\Zv,
\qquad
\Zv=[\zv_1,\ldots,\zv_n],
\end{equation*}
but replace the spectral and moment conditions by those below.

\setcounter{assumption}{0}
{
  \renewcommand{\theassumption}{\arabic{assumption}\ensuremath{'}}
\begin{assumption}
\label{ass:spiked population model HL,UHD}
The following conditions hold.
\begin{enumerate}
    \item[(a)] $\lim_{p\to\infty}\frac{1}{p-m}\sum_{j=m+1}^{p}\lambda_{j,p}=\tau^2>0$ and $\max_{m+1\le j\le p}\lambda_{j,p}$ is bounded.

    \item[(b)] For each $\ell\in[n]$, the entries of $\zv_\ell$ are independent with mean zero, variance one, and uniformly bounded fourth moments. Under the HDLSS regime, additionally assume that $n>m$ and that the entries have continuous distributions.

\end{enumerate}
  \end{assumption}
}
\setcounter{assumption}{2}

The following result shows that, unlike under the RMT regime, the proposed estimator and the HL estimator become asymptotically equivalent under both alternative regimes.

\begin{proposition}\label{prop:multi spike other regimes}
Suppose Assumptions~\ref{ass:spiked population model HL,UHD} and \ref{ass:target} hold. Under either the HDLSS or UHD regime, the following statements hold, with all probability limits taken along the corresponding regime.

\begin{enumerate}[label=(\roman*)]
    \item The proposed estimator asymptotically agrees with both the HL estimator and the oracle eigenspace estimator: $
    S(\UchJS_m,\UchHL_m)\convp 1$, and $S(\UchJS_m,\Uchorc_m)\convp 1$.

    \item If $a_{ii} > 0$ for some $i\in [m]$, then $\UchJS_m$ strictly outperforms $\Uch_m$ in the limit: 
    $$\plim S(\UchJS_m,\Uc_m)>\plim S(\Uch_m,\Uc_m).$$

    \item If $a_{ii}=0$ for all $i\in[m]$, then $\UchJS_m$ asymptotically reduces to $\Uch_m$: $
    S(\UchJS_m,\Uch_m)\convp 1.$
\end{enumerate}
\end{proposition}

The equivalence in Proposition~\ref{prop:multi spike other regimes} can again be understood from \eqref{eq:estimator_equiv_reprn} and the discussion that follows it. Under either the HDLSS or UHD regime, the two shrinkage numerators are asymptotically equivalent, as shown in the proof of Proposition~\ref{prop:multi spike other regimes}:
$$\frac{\bar\lambda}{\lambdah_i}  - (1-\rhoh_i^2)
\convp 0.
$$
Thus, the difference in shrinkage amounts that leads to the strict improvement over the HL estimator under the RMT regime disappears in these regimes. 

Under the fixed-dimensional target framework, across all high-dimensional regimes considered here, the proposed full-eigenspace estimator retains the oracle property and improves upon PCA whenever the target is informative.

\section{Simulation studies}\label{sec:simulation}

In this section, we numerically examine the finite-sample performance of the proposed estimator under different target geometries and compare it with competing methods. 
\subsection{Simulation setup}\label{subsec:simulation_setup}
For each $p$, the spiked eigenvalues are set to 
$\lambda_i=2+(m-i)$ for $i \in [m]$, and the non-spiked eigenvalues are $\lambda_{i}=({p-i+1/2})/({p-m})$, for $i=m+1,\ldots,p$, so that the limiting non-spiked spectral distribution is $H=\operatorname{Unif}(0,1)$. 

We consider two target geometries. In Model~I, $m=2$, $(\lambda_1,\lambda_2)=(3,2)$, $r=2$, and
\begin{equation*}
\Av=
\begin{pmatrix}
a&a_{\rm off}\\
a_{\rm off}&a
\end{pmatrix}.
\end{equation*}
Here $a$ controls the target information for each population eigenvector, while $a_{\rm off}$ controls the overlap between their projections onto $\Cc$. We use
\begin{equation*}
a\in\{0,0.25,0.5\},
\qquad
a_{\rm off}\in\{-0.5,-0.25,0,0.25,0.5\},
\end{equation*}
retaining only combinations satisfying
$|a_{\rm off}|\le\min\{a,1-a\}$, which ensures
$\0v\preceq\Av\preceq\Iv_2$.

In Model~II, $m=4$, $(\lambda_1,\ldots,\lambda_4)=(5,4,3,2)$, $r=2$, and
\begin{equation*}
\Av
=
a
\begin{pmatrix}
1&1&0&0\\
1&1&0&0\\
0&0&1&0\\
0&0&0&0
\end{pmatrix},
\qquad
a\in\{0,0.25,0.5\}.
\end{equation*}
Thus, for $a>0$, the projections of $\uv_1$ and $\uv_2$ onto $\Cc$ are nonorthogonal, $\uv_3$ has a nonzero projection onto $\Cc$ but no such interaction with the other spiked directions, and $\uv_4$ is orthogonal to $\Cc$. Explicit constructions of the target subspaces for Models~I and II are given in Section~\ref{subsec:supp_target_construction}.

Unless stated otherwise, the entries of $\Zv$ are independent standard Gaussian variables. We consider $n\in\{100,500,1000,2000\}$, $p/n\in\{1,2\}$.
We compare PCA, the HL estimator in \eqref{eq:UchHL}, James--Stein shrinkage toward $\Cc$ without augmentation, and the proposed augmented James--Stein estimator. For individual-vector comparisons, the two James--Stein estimators are orthogonalized in spike order as in \eqref{eq:orthoestimator}. For HL, we use the ordered eigenvectors of $\Pv_{\UchHL_m}\Sv\Pv_{\UchHL_m}$ lying in $\UchHL_m$, ordered by decreasing eigenvalue, where $\Sv$ is the sample covariance matrix.

For an individual direction, performance is measured by squared alignment with the corresponding population eigenvector. For a $k$-dimensional eigenspace, we use the squared similarity in \eqref{eq:def subspace similarity}. Reported are the averages over  500 repetitions with error bars representing two standard errors.

\subsection{Finite-sample behavior under the RMT regime}
\label{subsec:simulation_geometry}
 
We first examine finite-sample convergence to the RMT limits. Under Model~I, we fix $a=0.5$ and $a_{\rm off}=0.25$ and consider $n\in\{100,500,1000,2000\}$ with $p/n\in\{1,2\}$. Figure~\ref{fig:simulation_convergence} reports the squared alignments of the first two estimated directions and the squared similarity of the full spiked eigenspace. As $n$ and $p$ increase with their ratio fixed, the empirical performance of PCA and the two James--Stein estimators approaches the corresponding asymptotic limits. Across all sample sizes and both aspect ratios considered, the proposed estimator empirically outperforms PCA, HL, and unaugmented James--Stein shrinkage. The conclusion is similar for other target geometries and for Model II.  

\begin{figure}[!tp]
\centering
\includegraphics[width=\linewidth]{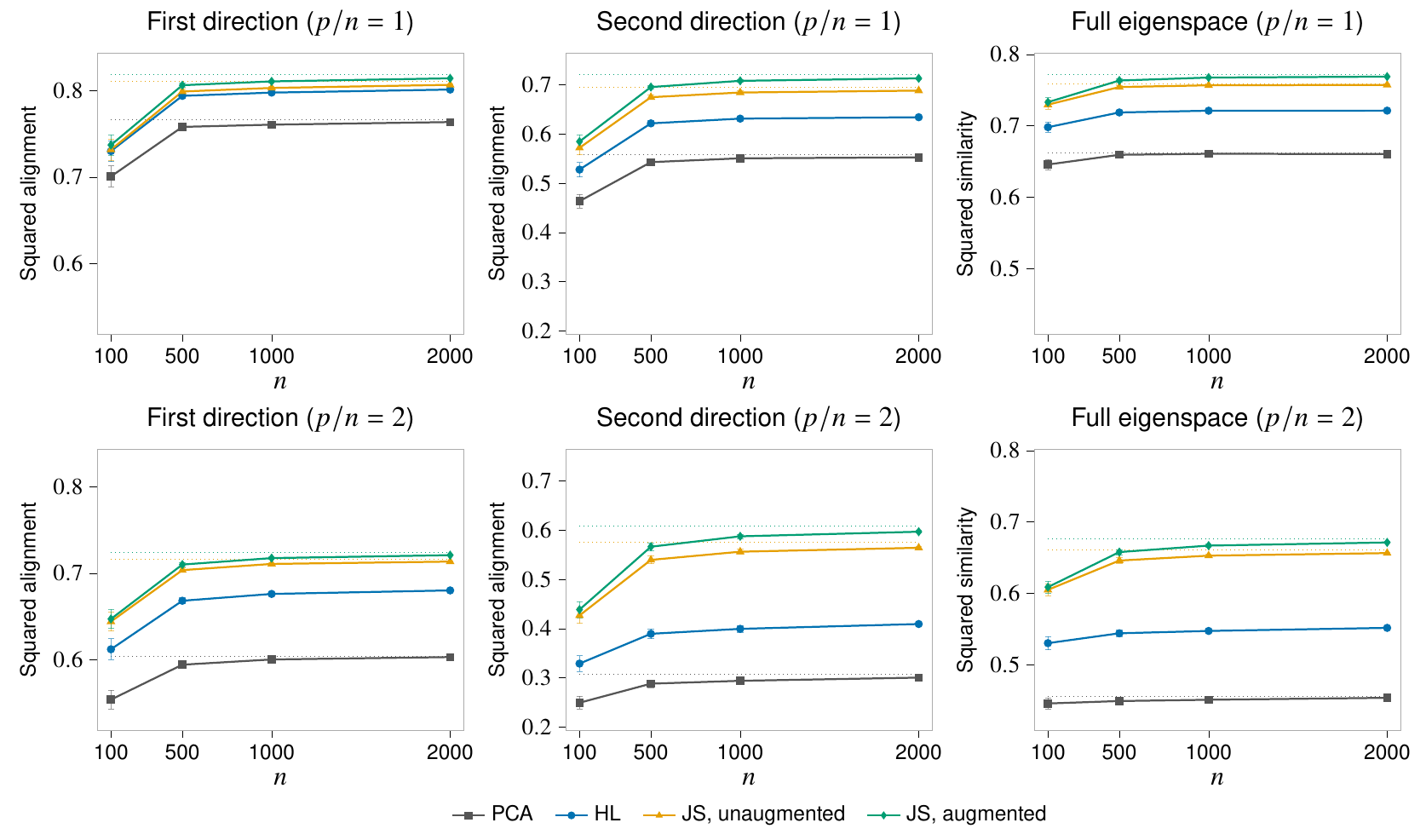}
 \caption{Finite-sample convergence under Model~I with $a=0.5$ and $a_{\rm off}=0.25$. The first and second rows correspond to $p/n=1$ and $p/n=2$, respectively. The three columns show the squared alignments of the first and second estimated directions and the squared similarity of the full spiked eigenspace. The dashed lines indicate the corresponding RMT limits.}
\label{fig:simulation_convergence}
\end{figure}

We next isolate the contribution of cross-eigenvector information by fixing $a=0.5$ and varying $a_{\rm off}$, with $(n,p)=(1000,2000)$. Since the diagonal entries of $\Av$ are unchanged, the amount of direct target information is fixed throughout this comparison. Figure~\ref{fig:simulation_offdiag} illustrates the benefit of augmentation by comparing the augmented and unaugmented James--Stein estimators at each value of $a_{\rm off}$. The two are nearly indistinguishable at $a_{\rm off}=0$, whereas the augmented estimator increasingly outperforms its unaugmented counterpart as $|a_{\rm off}|$ grows. The results are also approximately symmetric in the sign of $a_{\rm off}$, consistent with the dependence of the asymptotic augmentation gain on $a_{\rm off}^2$ in Model~I.

\begin{figure}[!tp]
\centering
\includegraphics[width=\linewidth]{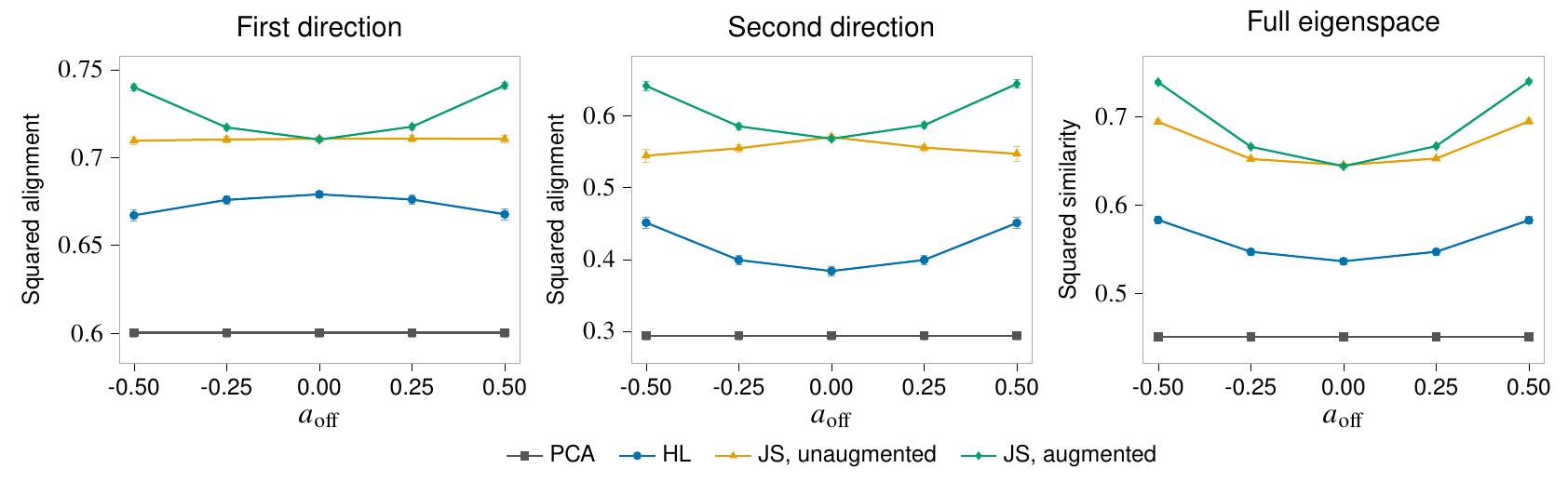}
\caption{Effect of cross-eigenvector target information in Model~I with $a=0.5$ and $(n,p)=(1000,2000)$. The diagonal target information is held fixed while $a_{\rm off}$ varies.}
\label{fig:simulation_offdiag}
\end{figure}

Model~II separates the roles of direct and cross-eigenvector information across directions. Figure~\ref{fig:simulation_modelII} reports the individual squared alignments for $(n,p)=(1000,2000)$ and $a=0.5$, $0.25$, and $0$. For $a=0.5$, augmentation gives a clear additional gain for the first two directions, while the augmented and unaugmented estimators perform similarly for the third direction and neither improves the fourth direction over PCA. The same pattern remains at $a=0.25$ with smaller gains, whereas at $a=0$ both James--Stein estimators are close to PCA.

\begin{figure}[!tp]
\centering
\includegraphics[width=\linewidth]{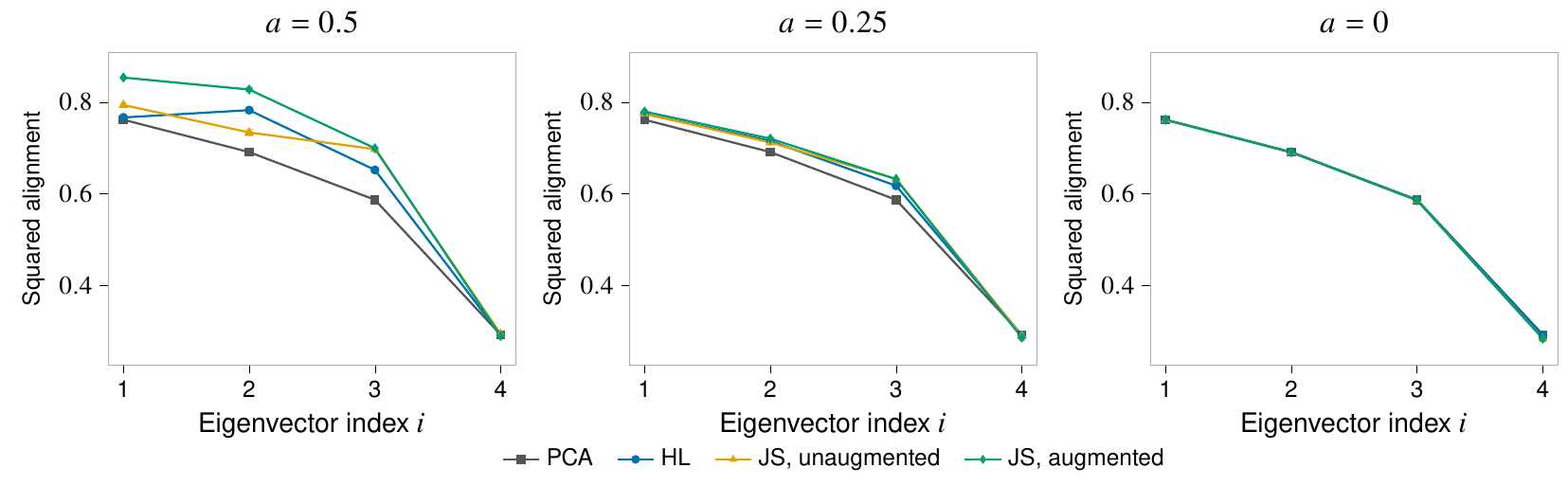}
\caption{Individual eigenvector estimation under Model~II with $(n,p)=(1000,2000)$. The first, second, and third panels correspond to $a=0.5$, $0.25$, and $0$, respectively.}
\label{fig:simulation_modelII}
\end{figure}

Table~\ref{tab:simulation_shrinkage_geometry} further describes the shrinkage geometry for $a=0.5$ by reporting the target angles before and after shrinkage together with the fitted shrinkage parameters.
For the first two directions, augmentation makes the target substantially closer to both the population and sample eigenvectors. For the third direction, the augmented and unaugmented geometries are nearly identical. For the fourth direction, the target remains nearly orthogonal to the relevant direction and no improvement over PCA is obtained despite the relatively large fitted shrinkage parameter, illustrating that the shrinkage parameter alone does not measure the amount of useful shrinkage.
\begin{table}[!tp]
\centering
\small
\begin{adjustbox}{max width=\linewidth}
\begin{tabular}{c|cccc|cccc}
\toprule
& \multicolumn{4}{c|}{Augmented shrinkage}
& \multicolumn{4}{c}{Unaugmented shrinkage}\\
$i$
& $\angle(\uv_i,\Ciaug)$
& $\angle(\uh_i,\Ciaug)$
& $\angle(\uhiJS,\Ciaug)$
& $\hat t_i$
& $\angle(\uv_i,\Cc)$
& $\angle(\uh_i,\Cc)$
& $\angle(\uh_{i,\Cc}^{\rm JS},\Cc)$
& $\hat t_{i,\Cc}$\\
\midrule
1 & $28.2^\circ$ & $39.2^\circ$ & $20.3^\circ$ & $0.562$
  & $45.0^\circ$ & $51.2^\circ$ & $38.5^\circ$ & $0.372$\\
2 & $25.5^\circ$ & $40.8^\circ$ & $16.3^\circ$ & $0.675$
  & $45.0^\circ$ & $53.7^\circ$ & $37.4^\circ$ & $0.448$\\
3 & $44.5^\circ$ & $56.4^\circ$ & $32.8^\circ$ & $0.573$
  & $45.0^\circ$ & $56.9^\circ$ & $33.5^\circ$ & $0.567$\\
4 & $84.2^\circ$ & $85.9^\circ$ & $75.6^\circ$ & $0.707$
  & $90.0^\circ$ & $87.2^\circ$ & $80.1^\circ$ & $0.704$\\
\bottomrule
\end{tabular}
\end{adjustbox}
\caption{Shrinkage geometry under Model~II with $a=0.5$ and $(n,p)=(1000,2000)$. All reported values are Monte Carlo averages.}
\label{tab:simulation_shrinkage_geometry}
\end{table}

Additional simulation results in Appendix~\ref{apdx:supp_simulation} show that the main performance patterns remain similar under non-Gaussian score distributions and under extreme high-dimensional settings. When the number of spikes is underspecified and the potential gain from augmentation is excluded, the performance of the proposed method is on par with the unaugmented James--Stein estimator. On the other hand, overspecification has comparatively little effect on the leading estimates; see Section~\ref{subsec:simulation_misspecification} for detailed numerical experiments.

\subsection{Increasing dimension of the target subspace}
\label{subsec:simulation_increasing_r}

We next examine the effect of increasing the dimension of the target subspace. We use Model~II with $a=0.5$ and $(n,p)=(1000,2000)$. Starting from the original two-dimensional target $\Cc_2$, we sequentially add independent Gaussian directions, orthogonalized against the directions already included, to obtain nested target subspaces $\Cc_r$ for
\begin{equation*}
r=2,\ldots,p/4=500.
\end{equation*}
The added directions are generated independently of the data and the population eigenvectors, and the resulting nested target sequence is held fixed across Monte Carlo replications.

For both unaugmented and augmented James--Stein shrinkage, we compare the original rules for fixed-dimensional targets with their dimension-adjusted versions, using PCA and HL as benchmarks. The original rules have shrinkage parameters $\hat t_{i,\Cc}$ and $\hat t_i$, respectively, while the adjusted rules have shrinkage parameters $\hat t_{i,\delta}$ and $\hat t_{i,\delta}^{\rm aug}$, defined in Sections~\ref{subsec: target dimension} and \ref{subsec:proof increasing target dimension}. As before, we form eigenspace estimates by successive spans and obtain individual directions by Gram--Schmidt orthogonalization in spike order.

For each $r$, let $\Av_r=\Uv_m^\top\Pv_{\Cc_r}\Uv_m$. We summarize the information in $\Cc_r$ by
\begin{equation*}
I_{\rm total}(r)
=
\frac{1}{m}\tr(\Av_r)
=
\frac{1}{m}\sum_{i=1}^m
\|\Pv_{\Cc_r}\uv_i\|^2,
\qquad
I_{\rm excess}(r)
=
I_{\rm total}(r)-\frac{r}{p}.
\end{equation*}
Thus, $I_{\rm total}(r)$ is the average proportion of the population spiked directions captured by the target. As the target dimension grows, however, part of this increase occurs simply because a larger subspace captures a larger fraction of any generic direction. The term $r/p$ represents this dimensional baseline, and $I_{\rm excess}(r)$ measures the target information remaining after accounting for it. Although $I_{\rm total}(r)$ increases steadily with $r$, $I_{\rm excess}(r)$ decreases as uninformative directions are added (Figure~\ref{fig:simulation_increasing_r}, upper-left panel). This difference is reflected in the behavior of the shrinkage rules. 

For the first direction, which is already strongly informed by the original target, the fixed-dimensional and dimension-adjusted procedures behave almost identically over most of the range of $r$. Their squared alignments and shrinkage parameters begin to separate visibly only when $r$ becomes very large, near $r=p/4$. Thus, for this informative direction, the fixed-dimensional approximation remains accurate over a substantial range of increasing target dimensions, with the dimensional correction becoming important only when $r/p$ is no longer small.

The contrast is much sharper for the fourth direction, for which the original target is uninformative. The ordinary fixed-dimensional rules apply increasingly strong shrinkage as $r$ grows, and their squared alignments deteriorate substantially. In contrast, the dimension-adjusted shrinkage parameters decrease toward zero, and the corresponding estimators remain close to PCA throughout the range of $r$. The full-eigenspace results exhibit the same qualitative distinction: the fixed-dimensional and adjusted procedures are similar when $r$ is small, but separate increasingly as $r$ becomes a nonnegligible fraction of $p$.

\begin{figure}[!tp]
\centering
\includegraphics[width=\linewidth]{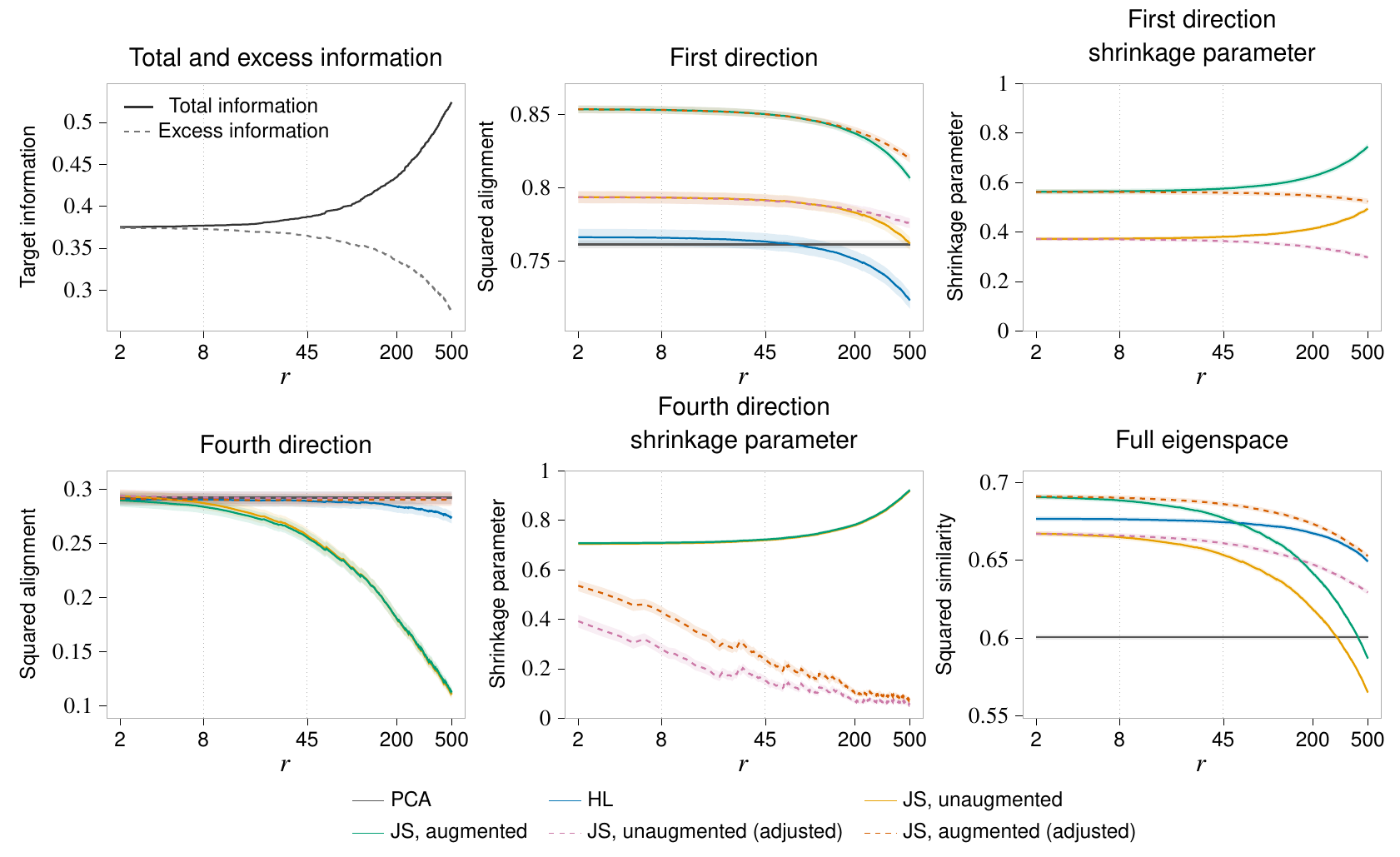}
\caption{Effect of increasing the target dimension under Model~II with $a=0.5$ and $(n,p)=(1000,2000)$. The top row shows, from left to right, the total and excess target information, the squared alignment of the first direction, and its shrinkage parameter. The bottom row shows the squared alignment of the fourth direction, its shrinkage parameter, and the squared similarity of the full spiked eigenspace. Vertical reference lines indicate $r\approx\log p$ and $r\approx\sqrt p$. The shaded bands represent $\pm 2$ standard errors.}
\label{fig:simulation_increasing_r}
\end{figure}

\section{Discussion}
\label{sec:discussion}

We developed an augmented James--Stein shrinkage framework for estimating leading eigenvectors in high dimensions. We evaluated the benefits of incorporating target information and information shared across the spiked sample eigenvectors under the RMT regime. In particular, the fully data-driven estimator improves upon standard PCA whenever the target is informative.

Several issues merit further discussion. First, one could specify a separate target subspace $\Cc_i$ for each population eigenvector $\uv_i$. In our setting, however, these targets can be pooled into the common subspace $\Cc_{\rm all}=\Cc_1+\cdots+\Cc_m$. Since $\Cc_{\rm all}$ contains every $\Cc_i$, pooling enlarges the signal subspace available for estimating each eigenvector and therefore cannot reduce the best attainable alignment at the oracle level. This motivates our use of a common target subspace. Second, extending the framework to an increasing number of spikes is a promising direction for future work, particularly in regimes where the spike magnitudes also diverge, as considered in \cite{cai2020limiting}. Finally, our analysis does not cover random target subspaces $\Cc_p$, including those estimated from another data source. Our framework provides a starting point for an analysis conditional on the random target, but a full treatment of multiple, potentially dependent data sources requires further investigation. We leave this problem for future work.
%
\begin{appendix}
\section{Additional Simulation Details and Results}\label{apdx:supp_simulation}

This appendix provides additional simulation details and results complementing Section~\ref{sec:simulation} of the main text. We first describe the explicit target-subspace constructions used in Models~I and II. We then examine robustness to non-Gaussian score distributions and to misspecification of the number of spikes. Results for other asymptotic regimes are collected in the final subsection.

\subsection{Construction of the target subspaces}
\label{subsec:supp_target_construction}

We give explicit constructions of the target subspaces used in Models~I and II. For Model~I, let $\uv_1=\ev_1$ and $\uv_2=\ev_2$, and define
\begin{equation*}
\vv_+=\frac{\ev_1+\ev_2}{\sqrt{2}},
\qquad
\vv_-=\frac{\ev_1-\ev_2}{\sqrt{2}}.
\end{equation*}
For each pair $(a,a_{\rm off})$ satisfying $\abs{a_{\rm off}}\le\min\{a,1-a\}$, set
\begin{equation*}
\cv_+
=
\sqrt{a+a_{\rm off}}\,\vv_+
+
\sqrt{1-a-a_{\rm off}}\,\ev_3,
\qquad
\cv_-
=
\sqrt{a-a_{\rm off}}\,\vv_-
+
\sqrt{1-a+a_{\rm off}}\,\ev_4,
\end{equation*}
and let
\begin{equation*}
\Cc=\vspan\{\cv_+,\cv_-\}.
\end{equation*}
Then, $\cv_+$ and $\cv_-$ are orthonormal, and
\begin{equation*}
\Uv_2^\top\Pv_{\Cc}\Uv_2
=
\begin{pmatrix}
a&a_{\rm off}\\
a_{\rm off}&a
\end{pmatrix}.
\end{equation*}

For Model~II, let $\uv_i=\ev_i$, $i=1,\ldots,4$, and define
\begin{equation*}
\cv_1
=
\sqrt{a}\,\ev_1+\sqrt{a}\,\ev_2+\sqrt{1-2a}\,\ev_5,
\qquad
\cv_2
=
\sqrt{a}\,\ev_3+\sqrt{1-a}\,\ev_6.
\end{equation*}
For $0\le a\le1/2$, the vectors $\cv_1$ and $\cv_2$ are orthonormal. Taking
\begin{equation*}
\Cc=\vspan\{\cv_1,\cv_2\}
\end{equation*}
gives
\begin{equation*}
\Uv_4^\top\Pv_{\Cc}\Uv_4
=
a
\begin{pmatrix}
1&1&0&0\\
1&1&0&0\\
0&0&1&0\\
0&0&0&0
\end{pmatrix}.
\end{equation*}

\subsection{Robustness to non-Gaussian scores}
\label{subsec:supp_nongaussian}

We examine the sensitivity of the finite-sample results to the distribution of the score entries. We use Model~I with $a=0.5$ and $a_{\rm off}=0.25$, and fix the aspect ratio at $p/n=2$. The sample size varies over
\begin{equation*}
n\in\{100,500,1000,2000\},
\qquad
p/n=2.
\end{equation*}
The population covariance matrix and the target subspace are held fixed across all distributions.

In addition to standard Gaussian scores, we consider standardized $t_7$ scores,
\begin{equation*}
z_{j\ell}=\frac{T_{j\ell}}{\sqrt{7/5}},
\qquad
T_{j\ell}\sim t_7,
\end{equation*}
and a mixed distribution,
\begin{equation*}
z_{j\ell}
=
\begin{cases}
T_{j\ell}/\sqrt{7/5}, & j\text{ odd},\\
G_{j\ell}-2, & j\text{ even},
\end{cases}
\qquad
T_{j\ell}\sim t_7,
\qquad
G_{j\ell}\sim\operatorname{Gamma}(4,1/2),
\end{equation*}
where the Gamma distribution is parameterized by shape and scale. Both distributions have mean zero, variance one, and finite fourth moments. We also include standardized $t_4$ scores,
\begin{equation*}
z_{j\ell}=\frac{T_{j\ell}}{\sqrt{2}},
\qquad
T_{j\ell}\sim t_4,
\end{equation*}
as a heavier-tailed setting outside the $(4+\varepsilon)$-th moment assumption underlying the RMT theory.

Figure~\ref{fig:supp_nongaussian} reports the first-direction squared alignment, the second-direction squared alignment, and the squared similarity of the full spiked eigenspace. The columns correspond to the four score distributions, and the rows correspond to the three performance measures. At $n=100$, the Gaussian setting yields higher alignment than the $t_7$ and mixed settings, but the three settings perform similarly at the larger sample sizes considered. Across all sample sizes considered, the relative ordering of the four estimators is unchanged, and the additional gain from augmentation remains visible for the second direction and for the full spiked eigenspace. The heavier-tailed $t_4$ setting yields lower absolute accuracy and greater Monte Carlo variability for all methods, but the two James--Stein procedures continue to improve substantially upon PCA and HL. Since the $t_4$ distribution does not satisfy the $(4+\varepsilon)$-th moment condition, this last experiment should be interpreted as a finite-sample robustness check rather than as a comparison with the RMT limits.

\begin{figure}[!tp]
\centering
\includegraphics[width=\linewidth]{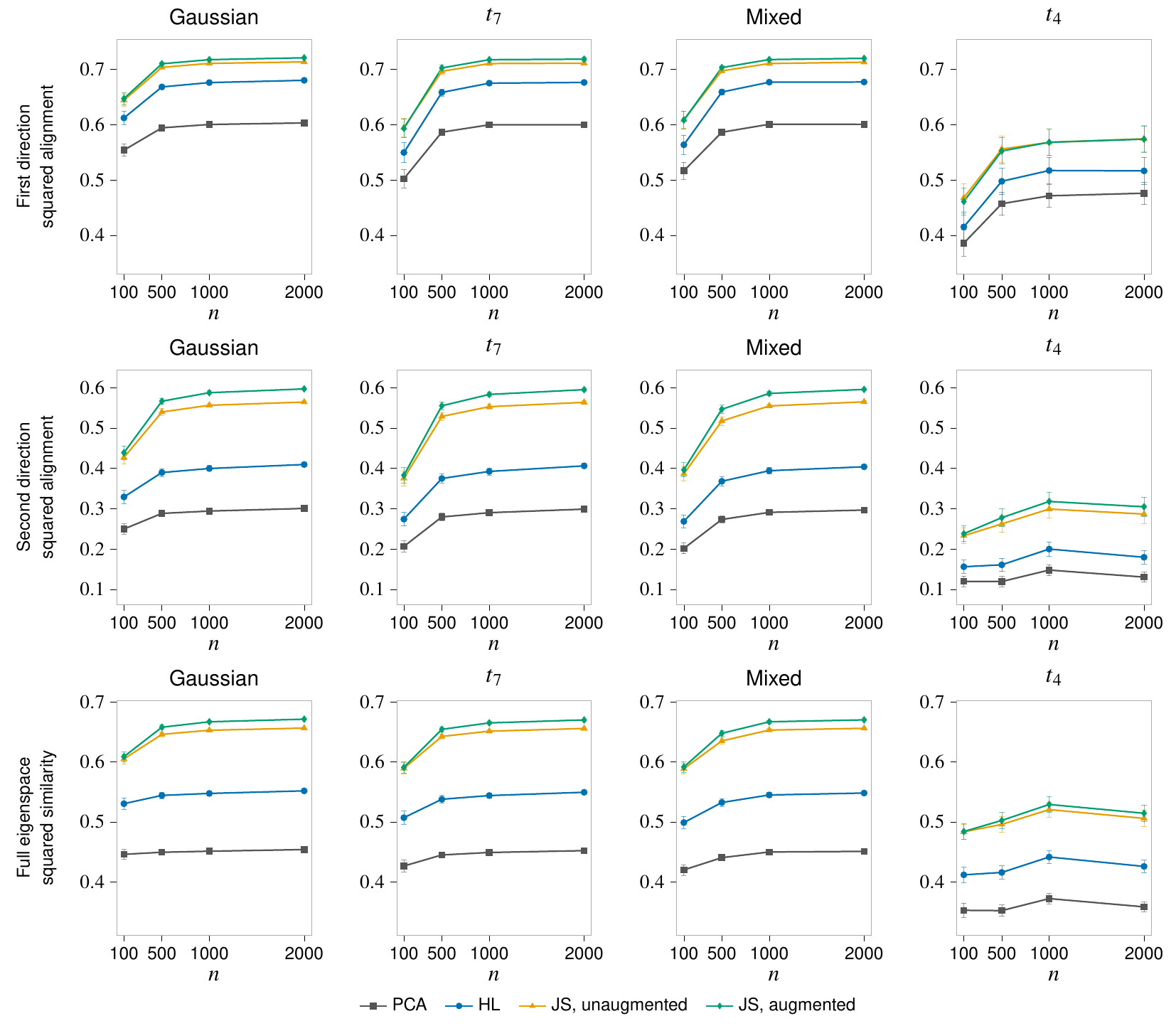}
\caption{Robustness to non-Gaussian score distributions under Model~I with $a=0.5$, $a_{\rm off}=0.25$, and $p/n=2$. The columns correspond to Gaussian, standardized $t_7$, mixed, and standardized $t_4$ scores, and the rows correspond to the first-direction squared alignment, the second-direction squared alignment, and the squared similarity of the full spiked eigenspace.}
\label{fig:supp_nongaussian}
\end{figure}

\subsection{Misspecification of the number of spikes}
\label{subsec:simulation_misspecification}

We next examine the robustness result of Section~\ref{subsec:model misspecification}. We use Model~II with true spike number $m_0=4$, $a=0.5$, and $(n,p)=(1000,2000)$, and construct the estimators using
\begin{equation*}
m\in\{1,2,3,4,5,6,8,12,20\}.
\end{equation*}

The effect of misspecification on estimation accuracy is clearly shown in Figure~\ref{fig:simulation_misspecification}. For the first direction, the main change occurs between $m=1$ and $m=2$. When $m=1$, $\uh_2$ is excluded from the augmented target, even though it carries cross-eigenvector information useful for estimating $\uv_1$. Once $\uh_2$ is included, the squared alignment of the proposed estimator increases sharply, while the inclusion of additional directions has little further effect. This is consistent with the target geometry of Model~II, in which the third and fourth directions contain no off-diagonal target information with $\uv_1$. The fourth direction provides the opposite case: because $\uv_4$ has no target information, its estimated alignment remains close to PCA for all $m\ge4$. The full-eigenspace similarity is likewise stable under overspecification, showing little deterioration even for working values of $m$ substantially larger than $m_0$.

\begin{figure}[!tp]
\centering
\includegraphics[width=\linewidth]{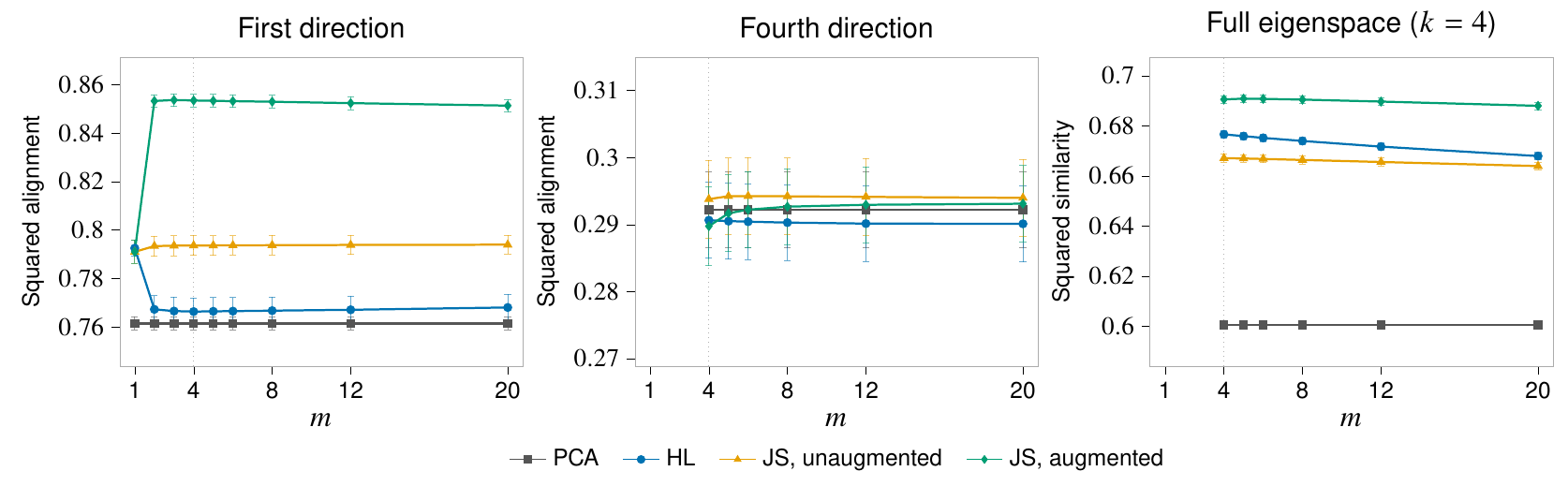}
\caption{Effect of misspecifying the number of spikes under Model~II with $a=0.5$ and $(n,p)=(1000,2000)$. The panels show the squared alignments of the first and fourth directions and the squared similarity of the true four-dimensional spiked eigenspace, respectively. The vertical line indicates the true spike number $m_0=4$.}
\label{fig:simulation_misspecification}
\end{figure}

\subsection{Simulation results under alternative asymptotic regimes}
\label{subsec:supp_other_regimes}

Lastly, we examine the estimators under the more extreme high-dimensional regimes discussed in Section~\ref{sec:other asymptotic regimes}. We use Model~I with $a=0.5$ and $a_{\rm off}=0.25$, and take
\begin{equation*}
(n,p)\in\{(100,1000),(200,4000),(400,10000)\}.
\end{equation*}
In this experiment, the non-spiked eigenvalues are all equal to one, while the two spiked eigenvalues increase with dimension according to
\begin{equation*}
\lambda_1=\frac{2}{4}\sqrt{p},
\qquad
\lambda_2=\frac{1}{4}\sqrt{p}.
\end{equation*}
The score entries are independent standard Gaussian variables.

Figure~\ref{fig:supp_other_regimes} shows the squared alignments of the first two directions and the squared similarity of the full spiked eigenspace. For both individual directions, the three target-based estimators outperform PCA throughout the range considered. The augmented James--Stein estimator generally gives the highest alignment, with the benefit of augmentation persisting across all dimensions considered. The full-eigenspace comparison is particularly relevant to the theory in Section~\ref{sec:other asymptotic regimes}: the proposed and HL estimators become increasingly close as $p$ grows, while both remain clearly separated from PCA. This behavior is consistent with their asymptotic equivalence under the UHD regime. The individual-eigenvector results are included as finite-sample comparisons, since the corresponding theoretical equivalence result is stated for the full spiked eigenspace.

\begin{figure}[!tp]
\centering
    \includegraphics[width=\linewidth]{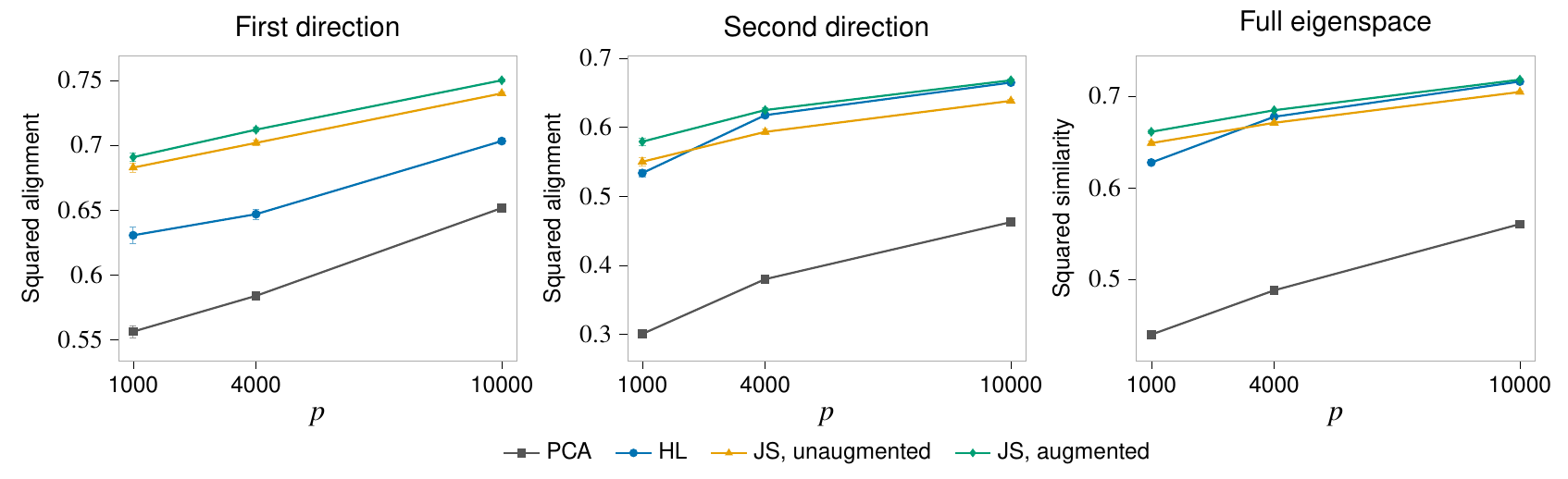}
\caption{Performance in increasingly high-dimensional settings under Model~I with $a=0.5$ and $a_{\rm off}=0.25$. The left and middle panels show the squared alignments of the first and second estimated eigenvectors, respectively, and the right panel shows the squared similarity of the full spiked eigenspace.}
\label{fig:supp_other_regimes}
\end{figure}

\section{Proofs and technical details}\label{apdx:proofs}
This appendix provides proofs and technical details for the theoretical results in the main text. We first collect notation and several preliminary results used repeatedly throughout the proofs, and then proceed in the order of the corresponding sections of the main text.

For $k\in[m]$, write
\begin{align*}
    \Uv_k&=[\uv_1,\ldots,\uv_k],
    &\Uvh_k&=[\uh_1,\ldots,\uh_k],\\
    \UvhJS_k&=[\uhJS_1,\ldots,\uhJS_k],
    &\Uvhorc_k&=[\uhorc_1,\ldots,\uhorc_k].
\end{align*}
For matrices, $\norm{\cdot}$ and $\norm{\cdot}_F$ denote the spectral and Frobenius norms, respectively. Throughout the proofs, whenever a random denominator or the inverse of a fixed-dimensional matrix appears, we verify that the denominator converges in probability to a positive limit or that the matrix converges in probability to a positive-definite limit. On the events where they are not well defined, they may be assigned arbitrary values; since the probabilities of these events tend to zero, the particular choice of those values does not affect the probability limits under consideration.

\subsection{Preliminary lemmas}

This section collects two elementary perturbation results that will be used repeatedly throughout the proofs. The following lemma gives two basic perturbation results for random vectors, which will be used without explicit reference throughout the proofs.

\begin{lemma}
Let $\xiv_p^1,\xiv_p^2\in\Rb^p$ be sequences of random vectors. Then the following statements hold as $p\to\infty$.

\begin{enumerate}[label=(\roman*)]
    \item If $\norm{\xiv_p^1-\xiv_p^2}\convp0$ and $\norm{\xiv_p^1}\convp a>0$, then
    $
        \norm{
            \xiv_p^1/\norm{\xiv_p^1}
            -
            \xiv_p^2/\norm{\xiv_p^2}}
        \convp0.
    $

    \item If $\norm{\xiv_p^1-\xiv_p^2}\convp0$, then, for every sequence $\zetav_p\in\Rb^p$ of random vectors satisfying $\norm{\zetav_p}=O_p(1)$,
    $
        \inner{\xiv_p^1,\zetav_p}
        -
        \inner{\xiv_p^2,\zetav_p}
        \convp0.
    $

\end{enumerate}
\end{lemma}

\begin{proof}
For part~(i), the reverse triangle inequality gives
$
\left|\norm{\xiv_p^2}-\norm{\xiv_p^1}\right|
\le
\norm{\xiv_p^2-\xiv_p^1}
\convp0.
$
Thus $\norm{\xiv_p^2}\convp a>0$. It follows that
\begin{equation*}
\norm{
\frac{\xiv_p^1}{\norm{\xiv_p^1}}
-
\frac{\xiv_p^2}{\norm{\xiv_p^2}}
}
\le
\frac{
\norm{\xiv_p^1-\xiv_p^2}
+
\left|\norm{\xiv_p^2}-\norm{\xiv_p^1}\right|
}{
\norm{\xiv_p^1}
}
\le
\frac{2\norm{\xiv_p^1-\xiv_p^2}}{\norm{\xiv_p^1}}
\convp0.
\end{equation*}

For part~(ii), the Cauchy--Schwarz inequality gives
\begin{equation*}
\left|
\inner{\xiv_p^1,\zetav_p}
-
\inner{\xiv_p^2,\zetav_p}
\right|
\le
\norm{\xiv_p^1-\xiv_p^2}\norm{\zetav_p}
\convp0.
\end{equation*}

\end{proof}

The next lemma gives the corresponding perturbation result for fixed-dimensional column spaces.

\begin{lemma}
\label{lem:subspace perturbation}
Fix $k\in\mathbb{N}$. For each $p$, let $\Yv_p=[\yv_{1,p},\ldots,\yv_{k,p}]$, $\Yv_p'=[\yv_{1,p}',\ldots,\yv_{k,p}']\in\Rb^{p\times k}$, and write $\Yc_p=\operatorname{col}(\Yv_p)$ and $\Yc_p'=\operatorname{col}(\Yv_p')$. Suppose
\begin{equation*}
    \norm{\Yv_p-\Yv_p'}_F\convp0,
    \qquad
    \Yv_p^\top\Yv_p\convp\Mv\succ\0v.
\end{equation*}
Then the following statements hold as $p\to\infty$.

\begin{enumerate}[label=(\roman*)]
    \item $S(\Yc_p,\Yc_p')\convp1$.

    \item For any sequence of random $k$-dimensional subspaces $\Ac_p\subset\Rb^p$,
    \begin{equation*}
        S(\Yc_p,\Ac_p)-S(\Yc_p',\Ac_p)\convp0.
    \end{equation*}
\end{enumerate}
\end{lemma}

\begin{proof}
Since $\Yv_p^\top\Yv_p\convp\Mv$, we have $\norm{\Yv_p}_F=O_p(1)$. By the triangle inequality, $\norm{\Yv_p'}_F\le\norm{\Yv_p}_F+\norm{\Yv_p'-\Yv_p}_F=O_p(1)$. Therefore,
\begin{align*}
    \norm{(\Yv_p')^\top\Yv_p'-\Yv_p^\top\Yv_p}_F
    &=
    \norm{(\Yv_p'-\Yv_p)^\top\Yv_p'+\Yv_p^\top(\Yv_p'-\Yv_p)}_F\\
    &\le 
    \paren{\norm{\Yv_p'}_F+\norm{\Yv_p}_F}
    \norm{\Yv_p'-\Yv_p}_F
    \convp0.
\end{align*}
Hence $(\Yv_p')^\top\Yv_p'\convp\Mv$. Similarly, $\Yv_p^\top\Yv_p'\convp\Mv$. Therefore,
\begin{align*}
    S(\Yc_p,\Yc_p')^2
    &=
    \frac{1}{k}
    \tr\left\{
        (\Yv_p^\top\Yv_p)^{-1}
        \Yv_p^\top\Yv_p'
        \paren{(\Yv_p')^\top\Yv_p'}^{-1}
        (\Yv_p')^\top\Yv_p
    \right\}\\
    &\convp
    \frac{1}{k}\tr(\Mv^{-1}\Mv\Mv^{-1}\Mv)
    =1.
\end{align*}
This proves part~(i).

For part~(ii), part~(i) gives $\norm{\Pv_{\Yc_p}-\Pv_{\Yc_p'}}_F^2
    =
    2k\curly{1-S(\Yc_p,\Yc_p')^2}
    \convp0$. Thus, by the Cauchy--Schwarz inequality for the Frobenius inner product,
\begin{align*}
\left|
        S(\Yc_p,\Ac_p)-S(\Yc_p',\Ac_p)
    \right|^2
    &\le
    \left|
        S(\Yc_p,\Ac_p)^2-S(\Yc_p',\Ac_p)^2
    \right|
    =
    \frac{1}{k}
    \left|
        \tr\paren{(\Pv_{\Yc_p}-\Pv_{\Yc_p'})\Pv_{\Ac_p}}
    \right|\\
    &\le
    \frac{1}{\sqrt{k}}
    \norm{\Pv_{\Yc_p}-\Pv_{\Yc_p'}}_F
    \convp0.
\end{align*}
\end{proof}

\subsection{Auxiliary results under the RMT regime}

This section verifies Lemmas~\ref{lem:eigenpair asymptotics RMT} and \ref{lem:projC uhi RMT} and the consistency of $\rhoh_i$. We may assume throughout this section that $\Sigmav=\Lambdav$ is diagonal. Indeed, replacing $\Xv$ by $\Uv^\top\Xv=\Lambdav^{1/2}\Zv$, whose covariance matrix is $\Lambdav$, leaves the population and sample eigenvalues unchanged. Transforming the target subspace $\Cc$ by $\Uv^\top$ also preserves all inner products and projection norms appearing in the assertions.

We also record the limit of the empirical distribution of the non-spiked sample eigenvalues. By Theorems~4.1 and~4.3 of \cite{bai2010spectral}, there exists a probability distribution $F_{\gamma,H}$, depending only on $\gamma$ and $H$, such that
\begin{equation}
    \frac{1}{p-m}\sum_{j=m+1}^p\delta_{\lambdah_j}
    \xrightarrow{w}F_{\gamma,H}
    \qquad\text{almost surely}.
    \label{eq:pf: sample ESD weak convergence}
\end{equation}
Furthermore, by Lemma~3.1 of \cite{bai2012sample}, $F_{\gamma,H}$ has a compact support $\Gamma_{F_{\gamma,H}}$.

\begin{proof}[Proof of Lemmas~\ref{lem:eigenpair asymptotics RMT} and \ref{lem:projC uhi RMT}]
Fix $i\in[m]$. We first verify, for every deterministic $\vv=\vv_p\in\Rb^p$ with uniformly bounded norm,
\begin{equation}
    \lambdah_i\convp\psi(\lambda_i),
    \qquad
    \abs{\inner{\uh_i,\vv}^2-\rho_i^2\inner{\uv_i,\vv}^2}\convp0.
    \label{eq:pf:eigenpair stronger limits}
\end{equation}
For later use, we also establish
\begin{equation}
    \lambdah_{m+1}\convp\sup\Gamma_{F_{\gamma,H}}.
    \label{eq:pf:sample bulk separation}
\end{equation}
Write $b=\sup\Gamma_H>0$. When $\lambda_+>b$, \eqref{eq:pf:eigenpair stronger limits} and \eqref{eq:pf:sample bulk separation} follow from Theorems~3.6 and~3.10 and Corollary~3.19 of \cite{ding2021spiked}. It remains to consider the case $\lambda_+=b$.

For $h\in(0,b/2)$, define the population covariance matrix with the non-spiked eigenvalues truncated at $b-h$:
\begin{equation*}
\Sigmav_p^{(h)}
=
\diag\{\lambda_1,\ldots,\lambda_m,
\lambda_{m+1,p}^{(h)},\ldots,\lambda_{p,p}^{(h)}\},
\qquad
\lambda_{j,p}^{(h)}=\min\{\lambda_{j,p},b-h\}.
\end{equation*}
Define $\Xv^{(h)}=(\Sigmav^{(h)})^{1/2}\Zv$ and $\Sv^{(h)}=n^{-1}\Xv^{(h)}(\Xv^{(h)})^\top$, with eigenpairs $(\lambdah_j^{(h)},\uh_j^{(h)})$ ordered by decreasing eigenvalue. If $T\sim H$, let $H_h$ be the distribution of $T_h=\min\{T,b-h\}$. Then, the counterparts of $\psi$ and $\lambda_+$ for $\Sigmav^{(h)}$ are $\psi_h(\lambda)=\lambda+\gamma\lambda\int t/(\lambda-t)\,dH_h(t)$ and $\lambda_{+,h}=\inf\{\lambda>b-h:\psi_h'(\lambda)>0\}$, where
\begin{equation*}
\psi_h'(\lambda)
=
1-\gamma\int\frac{t^2}{(\lambda-t)^2}\,dH_h(t).
\end{equation*}

Since $H_h(\{b-h\})=H([b-h,b])>0$, $\psi_h'(\lambda)\to-\infty$ as $\lambda\downarrow b-h$, while $T_h\le T$ gives $\psi_h'(\lambda)\ge\psi'(\lambda)>0$ for $\lambda>b$. Since $\psi_h'$ is continuous and strictly increasing on $(b-h,\infty)$, we obtain $b-h<\lambda_{+,h}\le b$. Therefore, we can apply \eqref{eq:pf:eigenpair stronger limits} to $\Sv^{(h)}$, and for any fixed $0<h<b/2$ and every $j\in[m]$,
\begin{equation}
\lambdah_j^{(h)}\convp\psi_h(\lambda_j),
\qquad
\abs{\inner{\uh_j^{(h)},\vv}^2
-
\frac{\lambda_j\psi_h'(\lambda_j)}{\psi_h(\lambda_j)}
\inner{\uv_j,\vv}^2}\convp 0,
\qquad
\lambdah_{m+1}^{(h)}\convp\psi_h(\lambda_{+,h}).
\label{eq:pf:ding and yang for h}
\end{equation}

We next compare $\Sv^{(h)}$ and $\Sv$. For fixed $h>0$ and all sufficiently large $p$, $\lambda_{m+1,p}>b-h$, and each row of $\Xv^{(h)}$ is obtained from the corresponding row of $\Xv$ by multiplying it by a factor in the interval $[\sqrt{(b-h)/\lambda_{m+1,p}},1]$. Hence
\begin{equation*}
    \norm{\Xv-\Xv^{(h)}}
    \le
    \paren{1-\sqrt{\frac{b-h}{\lambda_{m+1,p}}}}\norm{\Xv},
    \qquad
    \norm{\Xv^{(h)}}\le\norm{\Xv}.
\end{equation*}
Therefore,
\begin{align*}
    \norm{\Sv-\Sv^{(h)}}
    &\le
    \frac{1}{n}\norm{\Xv-\Xv^{(h)}}
    \paren{\norm{\Xv}+\norm{\Xv^{(h)}}}\\
    &\le
    \frac{2}{n}\paren{1-\sqrt{\frac{b-h}{\lambda_{m+1,p}}}}\norm{\Xv}^2
    =
    2\paren{1-\sqrt{\frac{b-h}{\lambda_{m+1,p}}}}\lambdah_1.
\end{align*}
Since $\lambdah_1=O_p(1)$ and $\lambda_{m+1,p}\to b$, we conclude that, for every $\eta>0$,
\begin{equation}
    \lim_{h\downarrow0}\liminf_{n,p\to\infty}
    \Pbr{\norm{\Sv-\Sv^{(h)}}<\eta}=1.
    \label{eq:pf:clipped covariance approximation}
\end{equation}

We now establish the eigenvalue limits. Since $T_h\to T$ and $0\le T_h\le b<\lambda_j$, dominated convergence gives $\psi_h(\lambda_j)\to\psi(\lambda_j)$ for every $j\in[m]$ as $h\downarrow0$. Moreover, since $b-h<\lambda_{+,h}\le b$ and $0\le\psi_h'(\lambda)\le1$ on $[\lambda_{+,h},b]$, the mean value theorem gives $0\le\psi_h(b)-\psi_h(\lambda_{+,h})\le h$. Applying monotone convergence gives
\begin{equation*}
    \psi_h(b)
    =
    b+\gamma b\,\Ebr{\frac{T_h}{b-T_h}}
    \xrightarrow{h\downarrow0}
    b+\gamma b\Ebr{\frac{T}{b-T}}=
    \lim_{\lambda\downarrow b}\psi(\lambda).
\end{equation*}
By $\lambda_+=b$ and Lemma~3.1 of \cite{bai2012sample}, we also have $\lim_{\lambda\downarrow b}\psi(\lambda)=\sup\Gamma_{F_{\gamma,H}}$. Hence $\psi_h(\lambda_{+,h})\to\sup\Gamma_{F_{\gamma,H}}$ as $h\downarrow0$. Weyl's inequality gives $\max_{j\in[m+1]}
    |\lambdah_j-\lambdah_j^{(h)}|
    \le\|\Sv-\Sv^{(h)}\|$. Combining this bound with \eqref{eq:pf:ding and yang for h} and \eqref{eq:pf:clipped covariance approximation}, first letting
$n,p\to\infty$ for fixed $h$ and then $h\downarrow0$, yields $\lambdah_{i}\convp\psi(\lambda_i)$ and $\lambdah_{m+1}\convp \sup\Gamma_{F_{\gamma,H}}$.

We now consider the eigenvectors. Since
$\psi(\lambda_1)>\cdots>\psi(\lambda_m)>\sup\Gamma_{F_{\gamma,H}}$,
the eigengap $\min_{j\ne i}\abs{\lambdah_i-\lambdah_j}$
converges in probability to a positive constant.
Applying the Davis-Kahan Theorem (Corollary~3 of \cite{yu2015useful}), we have, for every deterministic $\vv$ with uniformly bounded norm,
\begin{equation*}
    \abs{\inner{\uh_i,\vv}^2-\inner{\uh_i^{(h)},\vv}^2}
    \le
    \norm{\vv}^2
    \norm{\uh_i\uh_i^\top-\uh_i^{(h)}(\uh_i^{(h)})^\top}
    \le
    \frac{2\norm{\vv}^2}
    {\min_{j\ne i}\abs{\lambdah_i-\lambdah_j}}
    \norm{\Sv-\Sv^{(h)}}.
\end{equation*}
An argument similar to that used for the eigenvalue limits
yields the convergence
$|\langle\uh_i,\vv\rangle^2
-\rho_i^2\langle\uv_i,\vv\rangle^2|\convp 0$. Taking $\vv=\uv_i$ gives $\inner{\uh_i,\uv_i}\convp\rho_i$, while taking $\vv=\uv_j$ gives
$\inner{\uh_i,\uv_j}\convp0$ for every fixed $j\ne i$. This proves Lemma~\ref{lem:eigenpair asymptotics RMT}.

For Lemma~\ref{lem:projC uhi RMT}, write $\uh_i=\inner{\uh_i,\uv_i}\uv_i+(\Iv_p-\uv_i\uv_i^\top)\uh_i$. Then, we have
\begin{equation*}
    \norm{\projC\uh_i-\rho_i\projC\uv_i}\le \abs{\inner{\uh_i,\uv_i}-\rho_i}\norm{\projC\uv_i}+\norm{\projC(\Iv_p-\uv_i\uv_i^\top)\uh_i}.
\end{equation*}
Since $\inner{\uh_i,\uv_i}\convp\rho_i$, the first term on the right-hand side converges in probability to zero. Let $\cv_1,\ldots,\cv_r$ be an orthonormal basis of $\Cc$. Since $r$ is fixed, \eqref{eq:pf:eigenpair stronger limits} gives
\begin{equation*}
    \norm{\projC(\Iv_p-\uv_i\uv_i^\top)\uh_i}^2
    =
    \sum_{\ell=1}^r
    \inner{(\Iv_p-\uv_i\uv_i^\top)\cv_\ell,\uh_i}^2
    \convp0.
\end{equation*}
Therefore, $\norm{\projC\uh_i-\rho_i\projC\uv_i}\convp0$. Together with Assumption~\ref{ass:target}, this gives the remaining
assertions of Lemma~\ref{lem:projC uhi RMT}.
\end{proof}

We next verify the consistency of $\tilde\lambda_i$ and $\rhoh_i$ under our assumptions.

\begin{lemma}\label{lem:lambda rho estimation}
Suppose Assumption~\ref{ass:generalized spiked population model} holds. Then $\tilde\lambda_i\convp\lambda_i$ and $\rhoh_i\convp\rho_i$ for each $i\in[m]$.
\end{lemma}

\begin{proof}
Fix $i\in[m]$. By \eqref{eq:pf:eigenpair stronger limits} and \eqref{eq:pf:sample bulk separation}, every subsequence $(n_k)$ has a further subsequence $(n_{k_\ell})$ along which
\begin{equation*}
    \lambdah_i\convas\psi(\lambda_i),
    \qquad
    \lambdah_{m+1}\convas\sup\Gamma_{F_{\gamma,H}}.
\end{equation*}
All subsequent almost-sure limits are taken along this further subsequence. The argument in the proof of Theorem~11.16 of \cite{yao2015sample} then gives $\tilde\lambda_i\convas\lambda_i$.

To establish the consistency of $\rhoh_i$, we use the identity from the proof of Lemma~B of \cite{dey2019asymptotic},
\begin{equation*}
    \rho_i^2
    =
    \paren{
        1+\lambda_i\gamma
        \int\frac{t}{(\psi(\lambda_i)-t)^2}\,dF_{\gamma,H}(t)
    }^{-1}.
\end{equation*}
Choose $\eta>0$ sufficiently small that $\sup\Gamma_{F_{\gamma,H}}<\psi(\lambda_i)-\eta$. Almost surely, all non-spiked sample eigenvalues eventually lie in $[0,\psi(\lambda_i)-\eta]$, on which $t/(\lambdah_i-t)^2$ converges uniformly to the bounded continuous function $t/(\psi(\lambda_i)-t)^2$. Uniform convergence and the Portmanteau theorem applied to \eqref{eq:pf: sample ESD weak convergence} yield
\begin{align*}
    \frac{1}{p-m}\sum_{j=m+1}^p
    \frac{\lambdah_j}{(\lambdah_i-\lambdah_j)^2}
    &=
    \frac{1}{p-m}\sum_{j=m+1}^p
    \curly{
        \frac{\lambdah_j}{(\lambdah_i-\lambdah_j)^2}
        -
        \frac{\lambdah_j}{(\psi(\lambda_i)-\lambdah_j)^2}
    }\\
    &\quad+
    \frac{1}{p-m}\sum_{j=m+1}^p
    \frac{\lambdah_j}{(\psi(\lambda_i)-\lambdah_j)^2}\\
    &\convas
    \int\frac{t}{(\psi(\lambda_i)-t)^2}\,dF_{\gamma,H}(t).
\end{align*}
Combining this convergence with $\tilde\lambda_i\convas\lambda_i$ and $\gamma_n\to\gamma$ gives $\rhoh_i\convas\rho_i$. Since the original subsequence $(n_k)$ was arbitrary, we conclude that $\tilde\lambda_i\convp\lambda_i$ and $\rhoh_i\convp\rho_i$.
\end{proof}

\subsection{Proofs for Sections~\ref{subsec:single spike case}--\ref{subsec:multi spike case}}
This section proves the eigenvector estimation results developed in Sections~\ref{subsec:single spike case}--\ref{subsec:multi spike case}. We first consider the single-spiked case, and subsequently treat the multi-spiked case.

\begin{proof}[Proof of Proposition~\ref{prop:oracle single spike}]
Suppose first that $a_{11}>0$. By Lemmas~\ref{lem:lambda rho estimation} and \ref{lem:projC uhi RMT},
$\rhoh_1^2\convp\rho_1^2$ and $\norm{\projC\uh_1}^2\convp\rho_1^2a_{11}$, so $\hat{t}\convp t^*$. Moreover,
\begin{equation*}
\norm{\utilJS_1(\hat t)-\utilJS_1(t^*)}
=
\abs{\hat t-t^*}\norm{\uh_1-\projC\uh_1}
\convp0.
\end{equation*}
Since $a_{11}>0$, 
\begin{align*}
    \norm{\utilJS_1(t^*)}^2
    &=\norm{\projC\uh_1}^2+(1-t^*)^2\bigl(1-\norm{\projC\uh_1}^2\bigr)\\
    &\convp \rho_1^2a_{11}+(1-t^*)^2(1-\rho_1^2a_{11})>0.
\end{align*}
Hence normalization gives
$\norm{\uhJS_1(\hat t)-\uhJS_1(t^*)}\convp0$. Using \eqref{eq:single spike limiting alignment} at $t=t^*$, $\inner{\uhJS_1(\hat t),\uv_1}^2
\convp
G_{\rho_1}(a_{11})
>
\rho_1^2$, which proves part~(i).

Part~(ii) is deferred to Theorem~\ref{thm:uhJS limit}-(iii), whose proof includes the present setting as the special case $m=1$. Indeed, when $m=1$, $\Cc_1^{\rm aug}=\Cc$ and $\uh_1^{\rm JS}=\uhoneJS(\hat t)$.
\end{proof}

We next turn to the multi-spiked setting. The next lemma characterizes the limiting geometry of the residual directions $\rv_i=(\Iv_p-\Pv_\Cc)\uh_i$, the population spiked eigenvectors $\uv_i$ and the signal subspace $\Sc$. These limits will be used repeatedly in the analysis of the augmented target and signal subspaces.

\begin{lemma}\label{lem:pf:RtR RtUm WtW lim}
Suppose Assumptions~\ref{ass:generalized spiked population model}--\ref{ass:target} hold. Then,
\begin{align}
    \Rv^\top\Rv
    &\convp
    \Iv_m-\Dv\Av\Dv
    \succ\0v,
    \label{eq:pf:RtR lim new}\\
    \Rv^\top\Uv_m
    &\convp
    \Dv(\Iv_m-\Av),
    \label{eq:pf:RtU lim new}\\
    \Uv_m^\top\Pv_\Sc\Uv_m
    &\convp
    \Bv
    \succeq
    \Dv^2
    \succ\0v.
    \label{eq:pf:UtPSU lim new}
\end{align}
\end{lemma}

\begin{proof}
For $j,k\in[m]$, the orthonormality of the sample eigenvectors and Lemmas~\ref{lem:eigenpair asymptotics RMT} and \ref{lem:projC uhi RMT} give
\begin{align*}
    \inner{\rv_j,\rv_k}
    &=
    \inner{\uh_j,\uh_k}
    -
    \inner{\projC\uh_j,\projC\uh_k}
    \convp
    \delta_{jk}-\rho_j\rho_k a_{jk},\\
    \inner{\rv_j,\uv_k}
    &=
    \inner{\uh_j,\uv_k}
    -
    \inner{\projC\uh_j,\uv_k}
    \convp
    \rho_j(\delta_{jk}-a_{jk}),
\end{align*}
where $\delta_{jk}$ denotes the Kronecker delta. Moreover, since $\0v\preceq\Av\preceq\Iv_m$,
$\Iv_m-\Dv\Av\Dv\succeq\Iv_m-\Dv^2\succ\0v$. This proves \eqref{eq:pf:RtR lim new} and \eqref{eq:pf:RtU lim new}. 

Since $\Sc=\Cc+\vspan\{\uh_1,\ldots,\uh_m\}=\Cc\oplus\operatorname{col}(\Rv)$, we have $\Pv_\Sc=\Pv_\Cc+\Rv(\Rv^\top\Rv)^{-1}\Rv^\top$.
Hence, by Assumption~\ref{ass:target}, \eqref{eq:pf:RtR lim new}, and \eqref{eq:pf:RtU lim new},
\begin{align*}
    \Uv_m^\top\Pv_\Sc\Uv_m
    &=
    \Uv_m^\top\Pv_\Cc\Uv_m
    +
    (\Rv^\top\Uv_m)^\top
    (\Rv^\top\Rv)^{-1}
    \Rv^\top\Uv_m\\
    &\convp
    \Av
    +(\Iv_m-\Av)\Dv
    (\Iv_m-\Dv\Av\Dv)^{-1}
    \Dv(\Iv_m-\Av)
    =\Bv.
\end{align*}

Finally, since $\operatorname{col}(\Uvh_m)\subseteq\Sc$, we have
$\Pv_\Sc-\Uvh_m\Uvh_m^\top\succeq\0v$, and hence
$\Uv_m^\top(\Pv_\Sc-\Uvh_m\Uvh_m^\top)\Uv_m\succeq\0v$.
Using $\Uv_m^\top\Pv_\Sc\Uv_m\convp\Bv$ and
$\Uv_m^\top\Uvh_m\convp\Dv$ from Lemma~\ref{lem:eigenpair asymptotics RMT} gives
$\Bv-\Dv^2\succeq\0v$.
Since $\rho_i>0$ for every $i\in[m]$, we obtain
$\Bv\succeq\Dv^2\succ\0v$, proving \eqref{eq:pf:UtPSU lim new}.
\end{proof}

Using the limits established in the previous lemma, we now prove Lemma~\ref{lem:augmented target geometry}.
\begin{proof}[Proof of Lemma~\ref{lem:augmented target geometry}]
Fix $i\in[m]$. If $m=1$, then $\Ciaug=\Cc$ and $\aiiaug=a_{11}$, so the result follows directly from Assumption~\ref{ass:target} and Lemma~\ref{lem:projC uhi RMT}. Suppose henceforth that $m\ge2$.

Recall that $\Rv_{-i}=[\rv_j:j\in \mmi]$ and $\Ciaug=\Cc\oplus\operatorname{col}(\Rv_{-i})$. By Lemma~\ref{lem:pf:RtR RtUm WtW lim},
\begin{equation}
\Rv_{-i}^\top\Rv_{-i}
\convp
\Iv_{m-1}-\Dv_{-i}\Av_{-i}\Dv_{-i}
\succ\0v,
\qquad
\Rv_{-i}^\top\uv_i
\convp
-\Dv_{-i}\av_{-i,i}.
\label{eq:pf:Rvmi top Rvmi lim}
\end{equation}
Hence,
\begin{align*}
\norm{\proj_{\Ciaug}\uv_i}^2
&=
\norm{\projC\uv_i}^2
+
(\Rv_{-i}^\top\uv_i)^\top
(\Rv_{-i}^\top\Rv_{-i})^{-1}
\Rv_{-i}^\top\uv_i\\
&\convp
\aiiaug=a_{ii}
+
\av_{-i,i}^\top\Dv_{-i}
(\Iv_{m-1}-\Dv_{-i}\Av_{-i}\Dv_{-i})^{-1}
\Dv_{-i}\av_{-i,i}.
\end{align*}
Since $\Iv_{m-1}-\Dv_{-i}\Av_{-i}\Dv_{-i}$ and $\Dv_{-i}$ are positive definite,
\begin{equation*}
    \av_{-i,i}^\top\Dv_{-i}
(\Iv_{m-1}-\Dv_{-i}\Av_{-i}\Dv_{-i})^{-1}
\Dv_{-i}\av_{-i,i}\ge 0
\end{equation*}
and equality holds if and only if $\av_{-i,i}=\0v$. Therefore, $a_{ii}\le\aiiaug$, with strict inequality if and only if $a_{ij}\ne0$ for some $j\ne i$. Moreover, $\|\proj_{\Ciaug}\uv_i\|^2\le1$, so $\aiiaug\le1$.

We next prove 
\begin{equation}
\norm{\proj_{\Ciaug}\uh_i-\rho_i\proj_{\Ciaug}\uv_i}\convp0.
\label{eq:pf:augmented target parallelism}
\end{equation}
By the orthogonal direct sum $\Ciaug=\Cc\oplus\operatorname{col}(\Rv_{-i})$,
\begin{equation*}
\norm{\proj_{\Ciaug}(\uh_i-\rho_i\uv_i)}^2
=
\norm{\projC(\uh_i-\rho_i\uv_i)}^2
+
\norm{\proj_{\operatorname{col}(\Rv_{-i})}(\uh_i-\rho_i\uv_i)}^2.
\end{equation*}
The first term on the right-hand side converges to zero in probability by Lemma~\ref{lem:projC uhi RMT}. For the second term, since $\operatorname{col}(\Rv_{-i})\subseteq\Cc^\perp$, Lemma~\ref{lem:pf:RtR RtUm WtW lim} gives $\Rv_{-i}^\top\uh_i=\Rv_{-i}^\top\rv_i\convp-\rho_i\Dv_{-i}\av_{-i,i}$ and $\Rv_{-i}^\top\uv_i\convp-\Dv_{-i}\av_{-i,i}$, so $\Rv_{-i}^\top(\uh_i-\rho_i\uv_i)\convp\0v$. Since $(\Rv_{-i}^\top\Rv_{-i})^{-1}=O_p(1)$ by \eqref{eq:pf:Rvmi top Rvmi lim},
\begin{align*}
\norm{\proj_{\operatorname{col}(\Rv_{-i})}(\uh_i-\rho_i\uv_i)}^2
&=
\curly{\Rv_{-i}^\top(\uh_i-\rho_i\uv_i)}^\top
(\Rv_{-i}^\top\Rv_{-i})^{-1}
\Rv_{-i}^\top(\uh_i-\rho_i\uv_i)
\convp0.
\end{align*}
Consequently, we obtain \eqref{eq:pf:augmented target parallelism}. Together with $\|\proj_{\Ciaug}\uv_i\|^2\convp\aiiaug$, \eqref{eq:pf:augmented target parallelism} yields 
\begin{equation}
\norm{\proj_{\Ciaug}\uh_i}^2\convp\rho_i^2\aiiaug,
\qquad
\inner{\proj_{\Ciaug}\uh_i,\uv_i}\convp\rho_i\aiiaug.
\label{eq:augmented target projection limit}
\end{equation}
\end{proof}

\begin{proof}[Proof of Theorem~\ref{thm:uhJS limit}]
Fix $i\in[m]$. Recall $\tilde{\uv}^{\rm JS}_i = \uh_i  - \hat{t}_i (\Iv_p-\Pv_{\Ciaug})\uh_i$. To establish part~(i), we first show
\begin{equation}
    \norm{\utilJS_i-\rho_i\projS\uv_i}\convp0.
    \label{eq:unnormalized oracle approximation}
\end{equation}
Define $\rv_i^{\rm aug}=\uh_i-\projCiaug\uh_i$. Since $
\Sc
=
\Cc+\vspan\{\uh_1,\ldots,\uh_m\}
=
\Ciaug\oplus\vspan\{\rv_i^{\rm aug}\}$, we have
\begin{align*}
    \utilJS_i
    &=\projCiaug\uh_i+(1-\hat t_i)\rv_i^{\rm aug},\\
    \rho_i\projS\uv_i
    &=\rho_i\projCiaug\uv_i
    +\frac{\rho_i\inner{\rv_i^{\rm aug},\uv_i}}{\norm{\rv_i^{\rm aug}}^2}\rv_i^{\rm aug}.
\end{align*}
The components in $\Ciaug$ satisfy $\|\projCiaug\uh_i-\rho_i\projCiaug\uv_i\|\convp0$ by Lemma~\ref{lem:augmented target geometry}. To compare the coefficients of $\rv_i^{\rm aug}$, Lemma~\ref{lem:eigenpair asymptotics RMT} and \eqref{eq:augmented target projection limit} give
\begin{align}
    \norm{\rv_i^{\rm aug}}^2
    &=1-\norm{\projCiaug\uh_i}^2
    \convp1-\rho_i^2\aiiaug>0,
    \label{eq:pf: augmented residual norm lim}\\
    \inner{\rv_i^{\rm aug},\uv_i}
    &=\inner{\uh_i,\uv_i}-\inner{\projCiaug\uh_i,\uv_i}
    \convp\rho_i(1-\aiiaug).
    \label{eq:pf: augmented residual inner lim}
\end{align}
Thus, $\rhoh_i\convp \rho_i$ and the definition of $\hat{t}_i$ in \eqref{eq:ti hat def} yield $\hat t_i\convp(1-\rho_i^2)/(1-\rho_i^2\aiiaug)=t_i^*$. Both $1-\hat t_i$ and $\rho_i\inner{\rv_i^{\rm aug},\uv_i}/\|\rv_i^{\rm aug}\|^2$ therefore converge in probability to $\rho_i^2(1-\aiiaug)/(1-\rho_i^2\aiiaug)$. Since $\|\rv_i^{\rm aug}\|\le1$, the componentwise comparisons in $\Ciaug$ and along $\rv_i^{\rm aug}$ imply \eqref{eq:unnormalized oracle approximation}.

We also determine the norm of $\projS\uv_i$. The orthogonal direct sum $\Sc=\Ciaug\oplus\vspan\{\rv_i^{\rm aug}\}$, the limit $\|\projCiaug\uv_i\|^2\convp\aiiaug$ from Lemma~\ref{lem:augmented target geometry}, and \eqref{eq:pf: augmented residual norm lim}--\eqref{eq:pf: augmented residual inner lim} give
\begin{equation}
    \norm{\projS\uv_i}^2
    =\norm{\projCiaug\uv_i}^2
    +\frac{\inner{\rv_i^{\rm aug},\uv_i}^2}{\norm{\rv_i^{\rm aug}}^2}
    \convp\aiiaug+\frac{\rho_i^2(1-\aiiaug)^2}{1-\rho_i^2\aiiaug}
    =G_{\rho_i}(\aiiaug).
    \label{eq:signal projection norm limit}
\end{equation}
Since $G_{\rho_i}(\aiiaug)\ge\rho_i^2>0$, normalizing both vectors in \eqref{eq:unnormalized oracle approximation} yields $\|\uhiJS-\uhiorc\|\convp0$, proving part~(i). Moreover, $\langle\uhiorc,\uv_i\rangle^2=\|\projS\uv_i\|^2$, so \eqref{eq:signal projection norm limit} gives
\begin{equation}
    \inner{\uhiJS,\uv_i}^2\convp G_{\rho_i}(\aiiaug).
    \label{eq:pf:augmented JS limiting alignment}
\end{equation}

For part~(iii), suppose $a_{ii}=0$. Since $\Av\succeq\0v$, we have $a_{ij}=0$ for every $j\ne i$, and hence $\aiiaug=0$ by the definition of $\aiiaug$. Equation~\eqref{eq:signal projection norm limit} then gives $\|\projS\uv_i\|\convp\rho_i$. Since $\uh_i\in\Sc$, we have $\langle\uhiorc,\uh_i\rangle=\langle\uv_i,\uh_i\rangle/\|\projS\uv_i\|\convp1$. Together with $\|\uhiJS-\uhiorc\|\convp0$, this proves part~(iii).

It remains to establish part~(ii). Repeating the calculation leading to \eqref{eq:pf:augmented JS limiting alignment} with $\Ciaug$ replaced by $\Cc$ and $\Sc$ by $\Cc+\vspan\{\uh_i\}$ gives $\langle\uhJS_{i,\Cc},\uv_i\rangle^2\convp G_{\rho_i}(a_{ii})$. Since $0\le a_{ii}\le\aiiaug$, the strict monotonicity of $G_{\rho_i}$ gives $\rho_i^2=G_{\rho_i}(0)\le G_{\rho_i}(a_{ii})\le G_{\rho_i}(\aiiaug)$. The first inequality is strict if and only if $a_{ii}>0$, and Lemma~\ref{lem:augmented target geometry} shows that the second is strict if and only if $a_{ij}\ne0$ for some $j\ne i$.
\end{proof}

\subsection{Proofs for Section~\ref{subsec:spiked eigenspace estimation}}

This section proves the eigenspace and orthonormal-frame results in Section~\ref{subsec:spiked eigenspace estimation}. We begin by characterizing the maximal similarity with $\Uc_k$ among $k$-dimensional subspaces contained in $\Sc$.

\begin{lemma}\label{lem:similarity tilde Uc}
For $k\in[m]$, every $k$-dimensional subspace $\tilde\Uc_k\subset\Sc$ satisfies
\begin{equation*}
    S(\tilde\Uc_k,\Uc_k)^2
    =
    \frac{1}{k}\sum_{i=1}^k\norm{\proj_{\tilde\Uc_k}\uv_i}^2
    \le
    \frac{1}{k}\sum_{i=1}^k\norm{\projS\uv_i}^2.
\end{equation*}
When $\uhorc_1,\ldots,\uhorc_k$ are linearly independent, which occurs with probability tending to one under Assumptions~\ref{ass:generalized spiked population model}--\ref{ass:target}, equality holds if and only if $\tilde\Uc_k=\Uchorc_k$.
\end{lemma}

\begin{proof}
Let $\tilde\Uc_k\subset\Sc$ be a $k$-dimensional subspace. Since $\uv_1,\ldots,\uv_k$ form an orthonormal basis for $\Uc_k$,
$
    \Pv_{\Uc_k}=\sum_{i=1}^k\uv_i\uv_i^\top.
$
Thus, the definition of the similarity $S$ in \eqref{eq:def subspace similarity} gives
\begin{equation*}
    S(\tilde\Uc_k,\Uc_k)^2
    =
    \frac{1}{k}\tr\paren{\Pv_{\tilde\Uc_k}\Pv_{\Uc_k}}
    =
    \frac{1}{k}\sum_{i=1}^k\norm{\proj_{\tilde\Uc_k}\uv_i}^2.
\end{equation*}
Since $\tilde\Uc_k\subset\Sc$,
$
    \|\proj_{\tilde\Uc_k}\uv_i\|
    \le
    \|\projS\uv_i\|$ for every $i\in[k]$. Therefore,
\begin{equation}
    S(\tilde\Uc_k,\Uc_k)^2
    \le
    \frac{1}{k}\sum_{i=1}^k\norm{\projS\uv_i}^2.
    \label{eq:pf:Similarity upper bound}
\end{equation}
Since $\proj_{\tilde\Uc_k}\uv_i=\proj_{\tilde\Uc_k}\projS\uv_i$, the Pythagorean theorem gives
\begin{equation*}
    \norm{\projS\uv_i}^2-\norm{\proj_{\tilde\Uc_k}\uv_i}^2
    =
    \norm{\projS\uv_i-\proj_{\tilde\Uc_k}\uv_i}^2,
    \qquad i\in[k].
\end{equation*}
Thus equality in \eqref{eq:pf:Similarity upper bound} holds if and only if $\projS\uv_i\in\tilde\Uc_k$ for every $i\in[k]$. By Lemma~\ref{lem:pf:RtR RtUm WtW lim}, $\Uv_k^\top\Pv_\Sc\Uv_k\convp\Bv_{1:k,1:k}\succ\0v$, so $\projS\uv_1,\ldots,\projS\uv_k$ are linearly independent with probability tending to one under Assumptions~\ref{ass:generalized spiked population model}--\ref{ass:target}. On this event, $\uhorc_i=\projS\uv_i/\norm{\projS\uv_i}$ implies that $\uhorc_1,\ldots,\uhorc_k$ are also linearly independent and $\Uchorc_k=\vspan\{\projS\uv_1,\ldots,\projS\uv_k\}$. The equality condition is therefore equivalent to $\Uchorc_k\subseteq\tilde\Uc_k$, which, since both subspaces have dimension $k$, is equivalent to $\tilde\Uc_k=\Uchorc_k$.
\end{proof}

We now combine the preceding characterization with the individual-eigenvector results in Theorem~\ref{thm:uhJS limit} to prove Theorem~\ref{thm:main thm-1}.

\begin{proof}[Proof of Theorem~\ref{thm:main thm-1}]
Fix $k\in[m]$. Theorem~\ref{thm:uhJS limit}-(i) gives
\begin{equation*}
    \norm{\UvhJS_k-\Uvhorc_k}_F^2
    =\sum_{i=1}^k\norm{\uhJS_i-\uhorc_i}^2
    \convp0.
\end{equation*}
Since $\uhorc_i=\projS\uv_i/\norm{\projS\uv_i}$, Lemma~\ref{lem:pf:RtR RtUm WtW lim} yields
\begin{equation*}
    (\Uvhorc_k)^\top\Uvhorc_k
    \convp
    \diag(B_{11},\ldots,B_{kk})^{-1/2}
    \Bv_{1:k,1:k}
    \diag(B_{11},\ldots,B_{kk})^{-1/2}
    \succ\0v,
\end{equation*}
where positive definiteness follows from $\Bv_{1:k,1:k}\succ\0v$. Then, Lemma~\ref{lem:subspace perturbation}-(i) proves part~(i). 

For the comparison in part~(ii), we first record the limiting similarity of the proposed eigenspace estimator. Since $\uhorc_i=\projS\uv_i/\|\projS\uv_i\|$, we have $\|\projS\uv_i\|^2=\inner{\uhorc_i,\uv_i}^2$. Lemma~\ref{lem:similarity tilde Uc} therefore gives the equality below, and Theorem~\ref{thm:uhJS limit} gives the convergence:
\begin{equation}
    S(\Uchorc_k,\Uc_k)^2
    =\frac1k\sum_{i=1}^k\inner{\uhorc_i,\uv_i}^2
    \convp
\frac{1}{k}\sum_{i=1}^kG_{\rho_i}(\aiiaug).
    \label{eq:pf:oracle eigenspace similarity}
\end{equation}
Lemma~\ref{lem:subspace perturbation}-(ii) then gives
\begin{equation}
    S(\UchJS_k,\Uc_k)^2
    \convp
    \frac{1}{k}\sum_{i=1}^kG_{\rho_i}(\aiiaug).
    \label{eq:pf:augmented eigenspace similarity}
\end{equation}

We now consider the unaugmented eigenspace estimator $\UchJS_{k,\Cc}$. Write 
$$\tilde\Uv_{k,\Cc}=[\tilde{\uv}_{1,\Cc}^{\rm JS},\ldots,\tilde{\uv}_{k,\Cc}^{\rm JS}],$$
where $\tilde{\uv}_{i,\Cc}^{\rm JS}:=\projC\uh_i+(1-\hat t_{i,\Cc})\rv_i$. Lemma~\ref{lem:projC uhi RMT} gives $1-\hat t_{i,\Cc}\convp\rho_i^2(1-a_{ii})/(1-\rho_i^2a_{ii})$ for each $i\in[k]$. Together with Lemmas~\ref{lem:projC uhi RMT} and~\ref{lem:pf:RtR RtUm WtW lim}, this limit implies that $(\tilde\Uv_{k,\Cc})^\top\tilde\Uv_{k,\Cc}$ and $(\tilde\Uv_{k,\Cc})^\top\Uv_k$ have deterministic probability limits.

The probability limit of $(\tilde\Uv_{k,\Cc})^\top\tilde\Uv_{k,\Cc}$ is positive definite. Indeed, suppose $\norm{\tilde\Uv_{k,\Cc}\yv}^2\convp0$ for a fixed $\yv=(y_1,\ldots,y_k)^\top\in\Rb^k$. Since $\projC\uh_i\in\Cc$ and $\rv_i\in\Cc^\perp$ for every $i\in[k]$,
\begin{equation*}
    \norm{\tilde\Uv_{k,\Cc}\yv}^2
    =
    \norm{\sum_{i=1}^k y_i\projC\uh_i}^2
    +
    \norm{\sum_{i=1}^k(1-\hat t_{i,\Cc})y_i\rv_i}^2
    \convp0.
\end{equation*}
Both terms on the right-hand side are nonnegative and therefore converge to zero in probability. By \eqref{eq:pf:RtR lim new}, $[\inner{\rv_i,\rv_j}]_{i,j\in[k]}$ has a positive-definite probability limit, so convergence of the second term to zero implies $y_i\rho_i^2(1-a_{ii})/(1-\rho_i^2a_{ii})=0$ for every $i\in[k]$. Hence $y_i=0$ whenever $a_{ii}<1$. For each $i$ with $a_{ii}=1$, $\Iv_m-\Av\succeq\0v$ implies $a_{ij}=0$ for every $j\ne i$. Since Lemma~\ref{lem:projC uhi RMT} gives $\inner{\projC\uh_i,\projC\uh_j}\convp\rho_i\rho_j a_{ij}$, we obtain
\begin{equation*}
    \norm{\sum_{i=1}^k y_i\projC\uh_i}^2
    =
    \norm{\sum_{i\in[k],a_{ii}=1}y_i\projC\uh_i}^2
    \convp
    \sum_{i\in[k],a_{ii}=1}\rho_i^2y_i^2
    =0.
\end{equation*}
Because each $\rho_i>0$, we also have $y_i=0$ for every $i$ with $a_{ii}=1$. Thus $\yv=\0v$, proving that the probability limit of $(\tilde\Uv_{k,\Cc})^\top\tilde\Uv_{k,\Cc}$ is positive definite.

Since normalization does not change the column space, $\UchJS_{k,\Cc}=\operatorname{col}(\tilde\Uv_{k,\Cc})$. Hence the quantity
\begin{equation*}
    S(\UchJS_{k,\Cc},\Uc_k)^2
    =
    \frac{1}{k}\tr\bracket{
        \curly{(\tilde\Uv_{k,\Cc})^\top\tilde\Uv_{k,\Cc}}^{-1}
        (\tilde\Uv_{k,\Cc})^\top\Uv_k\Uv_k^\top\tilde\Uv_{k,\Cc}
    }
\end{equation*}
converges in probability to a deterministic limit. Since each $\uh_{i,\Cc}^{\rm JS}$ with $i\in[k]$ belongs to $\UchJS_{k,\Cc}$,
\begin{equation*}
    S(\UchJS_{k,\Cc},\Uc_k)^2
    =\frac{1}{k}\sum_{i=1}^k\norm{\proj_{\UchJS_{k,\Cc}}\uv_i}^2
    \ge\frac{1}{k}\sum_{i=1}^k\inner{\uh_{i,\Cc}^{\rm JS},\uv_i}^2.
\end{equation*}
Taking probability limits and using Theorem~\ref{thm:uhJS limit}-(ii) gives
\begin{equation*}
    \plim_{n,p\to\infty}S(\UchJS_{k,\Cc},\Uc_k)^2
    \ge\frac{1}{k}\sum_{i=1}^kG_{\rho_i}(a_{ii})
    \ge\frac{1}{k}\sum_{i=1}^k\rho_i^2
    =\plim_{n,p\to\infty}S(\Uch_k,\Uc_k)^2.
\end{equation*}
If $a_{ii}>0$ for some $i\in[k]$, then the second inequality holds strictly. Moreover, since $\UchJS_{k,\Cc}\subset\Sc$, Lemma~\ref{lem:similarity tilde Uc} and \eqref{eq:pf:oracle eigenspace similarity}--\eqref{eq:pf:augmented eigenspace similarity} give
\begin{equation}
    \plim_{n,p\to\infty}S(\UchJS_{k,\Cc},\Uc_k)^2
    \le\frac{1}{k}\sum_{i=1}^kG_{\rho_i}(\aiiaug)
    =\plim_{n,p\to\infty}S(\UchJS_k,\Uc_k)^2.
    \label{eq:pf:unaugmented augmented eigenspace comparison}
\end{equation}

Suppose $a_{ij}\ne0$ for some distinct $i,j\in[k]$. To prove that the inequality in \eqref{eq:pf:unaugmented augmented eigenspace comparison} is strict, assume, to the contrary, that it holds with equality. Then
\begin{equation*}
    \tr\bracket{\Uv_k^\top(\Pv_\Sc-\Pv_{\UchJS_{k,\Cc}})\Uv_k}
    =k\{S(\Uchorc_k,\Uc_k)^2-S(\UchJS_{k,\Cc},\Uc_k)^2\}
    \convp0.
\end{equation*}
Since $\UchJS_{k,\Cc}\subset\Sc$, the matrix inside the trace is positive semidefinite. Convergence of its trace to zero in probability therefore implies $\Uv_k^\top(\Pv_\Sc-\Pv_{\UchJS_{k,\Cc}})\Uv_k\convp\0v$. Using $\UchJS_{k,\Cc}\subset\Sc$ and
\begin{equation}
    \Pv_{\Uchorc_k}
    =\Pv_\Sc\Uv_k(\Uv_k^\top\Pv_\Sc\Uv_k)^{-1}\Uv_k^\top\Pv_\Sc,
    \label{eq:pf: Uchorc_k proj formula}
\end{equation}
we have
\begin{align}
    \norm{\Pv_{\UchJS_{k,\Cc}}-\Pv_{\Uchorc_k}}_F^2
    &=2k-2\tr\bracket{\Pv_{\Uchorc_k}\Pv_{\UchJS_{k,\Cc}}}\nonumber\\
    &=2k-2\tr\bracket{
        (\Uv_k^\top\Pv_\Sc\Uv_k)^{-1}
        \Uv_k^\top\Pv_{\UchJS_{k,\Cc}}\Uv_k
    }\nonumber\\
    &=2\tr\bracket{
        (\Uv_k^\top\Pv_\Sc\Uv_k)^{-1}
        \Uv_k^\top(\Pv_\Sc-\Pv_{\UchJS_{k,\Cc}})\Uv_k
    }
    \convp0.
    \label{eq:pf:unaugmented oracle equality}
\end{align}

We obtain a contradiction by finding a direction asymptotically orthogonal to $\Uchorc_k$ that has a nonzero limiting inner product with $\tilde{\uv}_{i,\Cc}^{\rm JS}\in\UchJS_{k,\Cc}$. Define
\begin{equation*}
    \wv_i=\projC\uv_i-\sum_{\ell=1}^k\frac{a_{\ell i}}{\rho_\ell}\uh_\ell.
\end{equation*}
Then $\wv_i\in\Sc$ and $\norm{\wv_i}$ is bounded. For every $s\in[k]$,
\begin{align*}
    \inner{\uv_s,\wv_i}
    &=\inner{\projC\uv_s,\projC\uv_i}
    -\sum_{\ell=1}^k\frac{a_{\ell i}}{\rho_\ell}\inner{\uv_s,\uh_\ell}
    \convp0.
\end{align*}
Thus $\Uv_k^\top\wv_i\convp\0v$. By \eqref{eq:pf: Uchorc_k proj formula}, $\|\proj_{\Uchorc_k}\wv_i\|\convp0$, and hence \eqref{eq:pf:unaugmented oracle equality} gives $\|\proj_{\UchJS_{k,\Cc}}\wv_i\|\convp0$. Since $\tilde{\uv}_{i,\Cc}^{\rm JS}\in\UchJS_{k,\Cc}$ and $\|\tilde{\uv}_{i,\Cc}^{\rm JS}\|\le1$, we obtain
\begin{equation}
    \abs{\inner{\wv_i,\tilde{\uv}_{i,\Cc}^{\rm JS}}}
    \le\norm{\proj_{\UchJS_{k,\Cc}}\wv_i}
    \convp0.
    \label{eq:pf:unaugmented witness inner product}
\end{equation}

We now evaluate this inner product directly. Recall that
\begin{equation*}
    \tilde{\uv}_{i,\Cc}^{\rm JS}
    =\hat t_{i,\Cc}\projC\uh_i+(1-\hat t_{i,\Cc})\uh_i,
    \qquad
    \hat t_{i,\Cc}\convp\frac{1-\rho_i^2}{1-\rho_i^2a_{ii}}.
\end{equation*}
By the orthonormality of the sample eigenvectors and Lemma~\ref{lem:projC uhi RMT},
\begin{align*}
    \inner{\wv_i,\tilde{\uv}_{i,\Cc}^{\rm JS}}
    &=\inner{\projC\uv_i,\projC\uh_i}
    -\frac{(1-\hat t_{i,\Cc})a_{ii}}{\rho_i}
    -\hat t_{i,\Cc}\sum_{\ell=1}^k\frac{a_{\ell i}}{\rho_\ell}
    \inner{\projC\uh_\ell,\projC\uh_i}\\
    &\convp
    \rho_i a_{ii}
    -\frac{\rho_i a_{ii}(1-a_{ii})}{1-\rho_i^2a_{ii}}
    -\frac{\rho_i(1-\rho_i^2)}{1-\rho_i^2a_{ii}}
    \sum_{\ell=1}^k a_{\ell i}^2\\
    &=-\frac{\rho_i(1-\rho_i^2)}{1-\rho_i^2a_{ii}}
    \sum_{\substack{\ell\in[k]\\\ell\ne i}}a_{\ell i}^2
    <0.
\end{align*}
This contradicts \eqref{eq:pf:unaugmented witness inner product}, proving that the inequality in \eqref{eq:pf:unaugmented augmented eigenspace comparison} is strict. This completes the proof of part~(ii).

Finally, if $a_{ii}=0$ for every $i\in[k]$, Theorem~\ref{thm:uhJS limit}-(iii) gives
\begin{equation*}
    \norm{\UvhJS_k-\Uvh_k}_F^2
    =\sum_{i=1}^k\norm{\uhJS_i-\uh_i}^2
    \convp0.
\end{equation*}
Since $\Uvh_k^\top\Uvh_k=\Iv_k$, Lemma~\ref{lem:subspace perturbation} yields $S(\UchJS_k,\Uch_k)\convp1$, proving part~(iii).
\end{proof}

We next study the effect of Gram--Schmidt orthogonalization on the proposed eigenvector estimators $\uhJS_1,\ldots,\uhJS_m$ in order. The following lemma gives a useful characterization of the alignment between a vector and its orthogonalized counterpart in terms of the corresponding Gram matrix. We use this characterization to derive the limiting alignment of the orthogonalized estimators $\qhJS_1,\ldots,\qhJS_m$.

\begin{lemma}\label{lem:QR-Cholesky}
Let $\Yv=[\yv_1,\ldots,\yv_k]\in\Rb^{p\times k}$ have full column rank, and
let $\Yv=\Qv\Tv$ be its unique thin QR decomposition such that $\Tv$ has
positive diagonal entries. If $\Mv=\Yv^\top\Yv$ and
$\Qv=[\qv_1,\ldots,\qv_k]$, then, for every $j\in[k]$,
\begin{equation}
    \inner{\qv_j,\yv_j}^2
    =
    M_{jj}-\Mv_{j,1:(j-1)}\Mv^{-1}_{1:(j-1),1:(j-1)}\Mv_{1:(j-1),j}
    =
    \min_{\zv\in\Rb^{j-1}}
    \begin{pmatrix}\zv\\1\end{pmatrix}^{\!\top}
    \Mv_{1:j,1:j}
    \begin{pmatrix}\zv\\1\end{pmatrix}.
    \label{eq:pf:diagonal cholesky factor}
\end{equation}
For $j=1$, \eqref{eq:pf:diagonal cholesky factor} is interpreted as $\inner{\qv_1,\yv_1}^2=M_{11}$.
\end{lemma}
\begin{proof}
The existence and uniqueness of the thin QR decomposition follow from Theorem~5.2.3 of \cite{golub2013matrix}. Moreover,
\begin{equation*}
    \Mv=\Yv^\top\Yv=\Tv^\top\Tv,
    \qquad
    \Tv=\Qv^\top\Yv.
\end{equation*}
Thus, $\Tv^\top$ is the Cholesky factor of $\Mv$, and its $j$th diagonal entry is $T_{jj}=\inner{\qv_j,\yv_j}$. Theorem~1.3 of \cite{zhang2006schur} yields the first equality in \eqref{eq:pf:diagonal cholesky factor}, and Theorem~2.3 of \cite{zhang2006schur} yields the second equality.
\end{proof}

\begin{proof}[Proof of Proposition~\ref{prop:orthogonalized frame}]
For $i=1$, we have $\qhJS_1=\uhJS_1$ and $\beta_1^2=B_{11}$, so the result follows from Theorem~\ref{thm:uhJS limit}. Fix $i\in\{2,\ldots,m\}$. Define the orthogonalized oracle vector by
\begin{equation*}
\qhorc_i
    =\frac{\qv_i^{\rm oracle}}{\norm{\qv_i^{\rm oracle}}},
    \qquad
    \qv_i^{\rm oracle}
    =(\Iv_p-\Pv_{\Uchorc_{i-1}})\uhorc_i.
\end{equation*}

We first show that $\norm{\qv_i^{\rm oracle}}^2\convp\beta_i^2/B_{ii}>0$. Using
$\uhorc_i=\projS\uv_i/\norm{\projS\uv_i}$, we have
\begin{equation}
    \norm{\qv_i^{\rm oracle}}^2
    =\inner{\qv_i^{\rm oracle},\uhorc_i}
    =
    \frac{
        \inner{
            (\Iv_p-\Pv_{\Uchorc_{i-1}})
            \projS\uv_i,\projS\uv_i
        }
    }{
        \norm{\projS\uv_i}^2
    }.
    \label{eq:pf:norm qiorc sq}
\end{equation}
We first investigate the numerator of the last term. By Lemma~\ref{lem:pf:RtR RtUm WtW lim}, we have $\Uv_i^\top\Pv_\Sc\Uv_i
    \convp
    \Bv_{1:i,1:i}
    \succ\0v$. Applying Lemma~\ref{lem:QR-Cholesky} to $\Pv_\Sc\Uv_i$ therefore yields
\begin{equation*}
    \inner{
        (\Iv_p-\Pv_{\Uchorc_{i-1}})
        \projS\uv_i
    ,\projS\uv_i}
    \convp
    \beta_i^2=B_{ii} - \Bv_{i,1:(i-1)}
\Bv_{1:(i-1),1:(i-1)}^{-1}
\Bv_{1:(i-1),i}>0.
\end{equation*}
Since $\norm{\projS\uv_i}^2\convp B_{ii}>0$,
\eqref{eq:pf:norm qiorc sq} gives
$\norm{\qv_i^{\rm oracle}}^2\convp\beta_i^2/B_{ii}>0$.

We now show that $\norm{\qhJS_i-\qhorc_i}\convp0$. By Theorem~\ref{thm:uhJS limit}-(i),
$\norm{\uhJS_i-\uhorc_i}\convp0$.
Moreover, Theorem~\ref{thm:main thm-1}-(i) gives
\begin{equation*}
    \norm{
        \Pv_{\UchJS_{i-1}}
        -
        \Pv_{\Uchorc_{i-1}}
    }_F^2
    =
    2(i-1)\{1-S(\UchJS_{i-1},\Uchorc_{i-1})^2\}
    \convp0.
\end{equation*}
By the triangle inequality,
\begin{align*}
    \norm{\qv_i^{\rm JS}-\qv_i^{\rm oracle}}
    &\le
    \norm{
        (\Iv_p-\Pv_{\UchJS_{i-1}})
        (\uhJS_i-\uhorc_i)
    }
    +
    \norm{
        (\Pv_{\Uchorc_{i-1}}-\Pv_{\UchJS_{i-1}})
        \uhorc_i
    }\\
    &\le
    \norm{\uhJS_i-\uhorc_i}
    +
    \norm{
        \Pv_{\UchJS_{i-1}}
        -
        \Pv_{\Uchorc_{i-1}}
    }_F
    \convp0.
\end{align*}
Together with $\norm{\qv_i^{\rm oracle}}^2\convp\beta_i^2/B_{ii}>0$, normalization gives
\begin{equation}
    \norm{\qhJS_i-\qhorc_i}=\norm{\frac{\qv^{\rm JS}_i}{\norm{\qv^{\rm JS}_i}}-\frac{\qv^{\rm oracle}_i}{\norm{\qv^{\rm oracle}_i}}}
    \convp0.
    \label{eq:pf:orthogonalized oracle approximation}
\end{equation}

We next determine the limiting alignment of $\qhorc_i$ with $\uv_i$. Since 
$\langle\qv_i^{\rm oracle},\uhorc_i\rangle
    =
    \norm{\qv_i^{\rm oracle}}^2$ and $\qv_i^{\rm oracle}\in\Sc$,
    \begin{align*}
    \inner{\qhorc_i,\uv_i}^2
    &=
    \frac{
        \inner{\qv_i^{\rm oracle},\projS\uv_i}^2
    }{
        \norm{\qv_i^{\rm oracle}}^2
    }=
    \norm{\projS\uv_i}^2
    \frac{
        \inner{\qv_i^{\rm oracle},\uhorc_i}^2
    }{
        \norm{\qv_i^{\rm oracle}}^2
    }\\
    &=
    \norm{\projS\uv_i}^2
    \norm{\qv_i^{\rm oracle}}^2
    \convp
    B_{ii}\frac{\beta_i^2}{B_{ii}}
    =
    \beta_i^2.
\end{align*}
Together with
\eqref{eq:pf:orthogonalized oracle approximation}, this gives
$
    \inner{\qhJS_i,\uv_i}^2
    \convp
    \beta_i^2.
$

It remains to compare $\beta_i^2$ with $\rho_i^2$.
By the variational characterization in
Lemma~\ref{lem:QR-Cholesky} and
$\Bv\succeq\Dv^2$ from
Lemma~\ref{lem:pf:RtR RtUm WtW lim},
\begin{equation}
    \beta_i^2
    =
    \min_{\yv\in\Rb^{i-1}}
    \begin{pmatrix}
        \yv\\
        1
    \end{pmatrix}^{\!\top}
    \Bv_{1:i,1:i}
    \begin{pmatrix}
        \yv\\
        1
    \end{pmatrix}
    \ge
    \min_{\yv\in\Rb^{i-1}}
    \begin{pmatrix}
        \yv\\
        1
    \end{pmatrix}^{\!\top}
    (\Dv_{1:i,1:i})^2
    \begin{pmatrix}
        \yv\\
        1
    \end{pmatrix}
    =
    \rho_i^2.
    \label{eq:pf:beta_i ge rho_i}
\end{equation}
The minimum on the right-hand side is uniquely attained at $\yv=\0v$.
If $\beta_i^2=\rho_i^2$, the minimum on the left-hand side is attained at $\yv=\0v$, and hence $B_{ii}
    =
    \beta_i^2
    =
    \rho_i^2$. Conversely, if $B_{ii}=\rho_i^2$, choosing $\yv=\0v$ in the
left-hand side of \eqref{eq:pf:beta_i ge rho_i} gives
$\beta_i^2\le B_{ii}=\rho_i^2$.
Together with $\beta_i^2\ge\rho_i^2$, this yields
$\beta_i^2=\rho_i^2$.
Thus, $\beta_i^2=\rho_i^2$ if and only if $B_{ii}=\rho_i^2$. By $B_{ii}=G_{\rho_i}(\aiiaug)$ and the equality condition in Theorem~\ref{thm:uhJS limit}, $B_{ii}=\rho_i^2$ if and only if $a_{ii}=0$. Therefore, $\beta_i^2>\rho_i^2$ if and only if $a_{ii}>0$.

Finally, suppose $a_{ii}=0$. Then $B_{ii}=\beta_i^2=\rho_i^2$, so $\norm{\qv_i^{\rm oracle}}\convp1$. Since $\norm{\qv_i^{\rm JS}-\qv_i^{\rm oracle}}\convp0$, we also have $\norm{\qv_i^{\rm JS}}\convp1$. The identity $\inner{\qhJS_i,\uhJS_i}=\norm{\qv_i^{\rm JS}}$, $\|\uhJS_i-\uh_i\|\convp 0$ from Theorem~\ref{thm:uhJS limit}-(iii), and the Cauchy--Schwarz inequality give
\begin{equation*}
    \inner{\qhJS_i,\uh_i}
    =
    \norm{\qv_i^{\rm JS}}
    +
    \inner{\qhJS_i,\uh_i-\uhJS_i}
    \convp1.
\end{equation*}

\end{proof}

\subsection{Proofs and technical details for Section~\ref{subsec:model misspecification}}

This section provides proofs and eigenspace counterparts of the results in Section~\ref{subsec:model misspecification}. Let $m_0$ denote the true number of spikes and $m$ the working value used to construct the proposed estimator. For each $k$, write $\UchJS_k(m)$ for the proposed estimator of $\Uc_k$ constructed using $m$, so that $\UchJS_k(m_0)$ is its correctly specified counterpart. The following proposition shows that underspecification may reduce the gain from augmentation, whereas overspecification leaves the estimators of the leading eigenspaces of dimension at most $m_0$ asymptotically unchanged. Here, $\aiiaug$ denotes the quantity defined in \eqref{eq:aii aug def}, evaluated under the $m_0$-spiked model.

\begin{proposition}
\label{prop:eigenspace spike misspecification}
Suppose Assumptions~\ref{ass:generalized spiked population model}--\ref{ass:target} hold with $m_0$ spikes.
\begin{enumerate}[label=(\roman*)]
\item
Suppose $m<m_0$. Then, for each $k\in[m]$,
\begin{equation}
\plim_{n,p\to\infty}S(\Uch_k,\Uc_k)^2
\le
\plim_{n,p\to\infty}S(\UchJS_k(m),\Uc_k)^2
\le
\plim_{n,p\to\infty}S(\UchJS_k(m_0),\Uc_k)^2.
\label{eq:apdx:underspecification PCA comparison}
\end{equation}
The first inequality is strict if and only if $a_{ii}>0$ for some $i\in[k]$. If $a_{ii}=0$ for every $i\in[k]$, then
$
S(\UchJS_k(m),\Uch_k)\convp1.
$

\item
Suppose $m>m_0$. Additionally assume Gaussianity and that $\Sigmav$ follows Johnstone's spiked population model in \eqref{eq:Johnstone spike model}. Then, for each $k\in[m_0]$,
$
S(\UchJS_k(m),\UchJS_k(m_0))\convp1.
$
Moreover, for each $m_0<k\le m$,
\begin{equation*}
S(\UchJS_k(m),\Uc_k)^2-S(\Uch_k,\Uc_k)^2
\ge
\frac{1}{k}\sum_{i=1}^{m_0}
\left[G_{\rho_i}(\aiiaug)-\rho_i^2\right]
+o_p(1).
\end{equation*}
\end{enumerate}
\end{proposition}

These results show that the proposed eigenspace estimators remain asymptotically no worse than PCA despite spike-number misspecification, paralleling the eigenvector results in Proposition~\ref{prop:number spike misspecification}. For $m_0<k\le m$, the individual estimator $\qhJS_k(m)$ asymptotically reduces to $\uh_k$. Nevertheless, the eigenspace estimator $\UchJS_k(m)$ retains the gains already attained for the first $m_0$ spiked components.

\begin{proof}[Proof of Propositions~\ref{prop:number spike misspecification} and \ref{prop:eigenspace spike misspecification}]
Write $\tilde\lambda_i(m)$, $\rhoh_i(m)$, $\hat t_i(m)$, and $\utilJS_i(m)$ for the quantities constructed using $m$, so that replacing $m$ by $m_0$ denotes the correctly specified quantities.

We first verify that $\tilde\lambda_j(m)$ and $\rhoh_j(m)$ remain consistent for every $j\le\min\{m,m_0\}$. From the definitions of $\tilde\lambda_j(m)$ and $\tilde\lambda_j(m_0)$, we have
\begin{equation*}
\abs{
\tilde\lambda_j(m)^{-1}
-\frac{n-m_0}{n-m}\tilde\lambda_j(m_0)^{-1}
}
\le
\frac{1}{n-m}
\sum_{\ell=\min\{m,m_0\}+1}^{\max\{m,m_0\}}
\frac{1}{\abs{\lambdah_j-\lambdah_\ell}}
\convp0.
\end{equation*}
The convergence follows from Lemma~\ref{lem:eigenpair asymptotics RMT} and \eqref{eq:pf:sample bulk separation}. Since $\tilde\lambda_j(m_0)\convp\lambda_j$ and $(n-m_0)/(n-m)\to1$, we obtain $\tilde\lambda_j(m)\convp\lambda_j$. A similar argument applied to \eqref{eq:rhoh_i def} gives $\rhoh_j(m)\convp\rho_j$.

\medskip
\noindent\emph{Underspecification.}
Suppose $m<m_0$. Throughout this case, $i$ and $k$ range over $[m]$. Define the signal and augmented target subspaces for the working value $m$ by
\begin{equation*}
\Sc(m)=\Cc+\vspan\{\uh_1,\ldots,\uh_m\},
\qquad
\Ciaug(m)=\Cc+\vspan\{\uh_j:j\in[m]\setminus\{i\}\}.
\end{equation*}
Define the corresponding oracle estimators by
\begin{equation*}
\uhorc_i(m)
=
\frac{\proj_{\Sc(m)}\uv_i}{\norm{\proj_{\Sc(m)}\uv_i}},
\qquad
\Uchorc_k(m)
=
\vspan\{\uhorc_1(m),\ldots,\uhorc_k(m)\},
\end{equation*}
and let $\qhorc_i(m)$ be the $i$th vector obtained by applying Gram--Schmidt orthogonalization to $\uhorc_1(m),\ldots,\uhorc_m(m)$ in order. Replacing $m$ by $m_0$ denotes the correctly specified counterparts.

The proof of Lemma~\ref{lem:augmented target geometry}
carries over with $\Ciaug$ and $\aiiaug$ replaced by
$\Ciaug(m)$ and $\aiiaug(m)$, respectively, and all spike
indices restricted to $[m]$, so the corresponding
conclusions hold. Here, for $m\ge2$, $\aiiaug(m)$ is obtained from \eqref{eq:aii aug def} by replacing $\Av_{-i}$ and $\Dv_{-i}$ with their leading $(m-1)\times(m-1)$ principal submatrices and $\av_{-i,i}$ with the vector of its first $m-1$ entries. For $m=1$, set $\aiiaug(m)=a_{ii}$. Together with $\rhoh_i(m)\convp\rho_i$, the arguments in the proofs of Theorem~\ref{thm:uhJS limit}-(i), Theorem~\ref{thm:main thm-1}-(i), and Proposition~\ref{prop:orthogonalized frame} give, for every $i,k\in[m]$,
\begin{equation*}
\norm{\uhJS_i(m)-\uhorc_i(m)}\convp0,
\qquad
\norm{\qhJS_i(m)-\qhorc_i(m)}\convp0,
\end{equation*}
and
\begin{equation}
S(\UchJS_k(m),\Uchorc_k(m))\convp1.
\label{eq:pf:underspecified oracle equivalence}
\end{equation}
The corresponding calculations for the oracle estimators yield
\begin{equation*}
S(\UchJS_k(m),\Uc_k)^2
\convp
\frac{1}{k}\sum_{i=1}^kG_{\rho_i}\{\aiiaug(m)\},
\qquad
\inner{\qhJS_i(m),\uv_i}^2\convp\beta_i^2(m).
\end{equation*}
For $i\ge2$, $\beta_i^2(m)$ is the Schur complement of the leading $(i-1)\times(i-1)$ principal submatrix of $\Bv(m)$ in the leading $i\times i$ principal submatrix of $\Bv(m)$, with $\beta_1^2(m)=[\Bv(m)]_{11}$, where
\begin{equation*}
\Bv(m)
=
\Av_m+(\Iv_m-\Av_m)\Dv_m
(\Iv_m-\Dv_m\Av_m\Dv_m)^{-1}
\Dv_m(\Iv_m-\Av_m).
\end{equation*}
Here, $\Av_m$ and $\Dv_m$ are the leading $m\times m$ principal submatrices of $\Av$ and $\Dv$, respectively.

The arguments in the proofs of Proposition~\ref{prop:orthogonalized frame} and Theorem~\ref{thm:main thm-1}-(ii) now give the first inequalities in \eqref{eq:underspecification PCA comparison} and \eqref{eq:apdx:underspecification PCA comparison}, respectively, with the stated conditions for strictness. Moreover, the corresponding arguments for an uninformative target show that $\inner{\qhJS_i(m),\uh_i}\convp1$ if $a_{ii}=0$, and that $S(\UchJS_k(m),\Uch_k)\convp1$ if $a_{ii}=0$ for every $i\in[k]$.

It remains to compare the underspecified estimators with their correctly specified counterparts. Since $\Sc(m)\subset\Sc(m_0)$, Lemma~\ref{lem:similarity tilde Uc} gives
\begin{align*}
S(\Uchorc_k(m),\Uc_k)^2
&=
\frac{1}{k}\sum_{i=1}^k\norm{\proj_{\Sc(m)}\uv_i}^2\\
&\le
\frac{1}{k}\sum_{i=1}^k\norm{\proj_{\Sc(m_0)}\uv_i}^2
=
S(\Uchorc_k(m_0),\Uc_k)^2.
\end{align*}
Taking probability limits and using \eqref{eq:pf:underspecified oracle equivalence} and Theorem~\ref{thm:main thm-1}-(i) proves the second inequality in \eqref{eq:apdx:underspecification PCA comparison}.

For the orthogonalized estimators, Lemma~\ref{lem:QR-Cholesky} applied to $\Pv_{\Sc(m)}\Uv_i$ and $\Pv_{\Sc(m_0)}\Uv_i$, together with $\Sc(m)\subset\Sc(m_0)$, gives
\begin{align*}
\inner{\qhorc_i(m),\uv_i}^2
&=
\min_{\yv\in\Rb^{i-1}}
\norm{\proj_{\Sc(m)}(\uv_i+\Uv_{i-1}\yv)}^2\\
&\le
\min_{\yv\in\Rb^{i-1}}
\norm{\proj_{\Sc(m_0)}(\uv_i+\Uv_{i-1}\yv)}^2
=
\inner{\qhorc_i(m_0),\uv_i}^2.
\end{align*}
Taking probability limits, using $\norm{\qhJS_i(m)-\qhorc_i(m)}\convp0$ and its correctly specified counterpart in \eqref{eq:pf:orthogonalized oracle approximation}, proves the second inequality in \eqref{eq:underspecification PCA comparison}.

\medskip
\noindent\emph{Overspecification.}
Assume additionally Gaussianity and that $\Sigmav$ follows Johnstone's spiked population model in \eqref{eq:Johnstone spike model}. Suppose $m>m_0$ and fix $i\in[m_0]$. To compare $\uhJS_i(m)$ with $\uhJS_i(m_0)$, we first show that
\begin{equation}
\norm{
\proj_{\Ciaug(m)}\uh_i-\proj_{\Ciaug(m_0)}\uh_i
}
\convp0.
\label{eq:pf:overspecified augmented projection equivalence}
\end{equation}

Recall that $\rv_j=(\Iv_p-\Pv_\Cc)\uh_j$, and write
\begin{equation*}
\Rv(m)=[\rv_1,\ldots,\rv_m],
\qquad
\Rv(m_0)=[\rv_1,\ldots,\rv_{m_0}].
\end{equation*}
Let $\Rv_{-i}(m)$ denote $\Rv(m)$ with its $i$th column removed, and use the same convention for $\Rv_{-i}(m_0)$. The orthogonal direct sum $\Ciaug(m)=\Cc\oplus\operatorname{col}(\Rv_{-i}(m))$ gives
\begin{equation}
\proj_{\Ciaug(m)}\uh_i
=
\projC\uh_i
+
\Rv_{-i}(m)
\{\Rv_{-i}(m)^\top\Rv_{-i}(m)\}^{-1}
\Rv_{-i}(m)^\top\uh_i.
\label{eq:pf:working augmented projection formula}
\end{equation}
The same formula holds with $m_0$ in place of $m$. When $m_0=1$, the terms involving $\Rv_{-i}(m_0)$ are omitted.

By \eqref{eq:pf:noise projection rate}, established at the end of the proof, $\norm{\projC\uh_j}\convp0$ for every $m_0<j\le m$. Combining this convergence with the orthonormality of the sample eigenvectors gives
\begin{align*}
\Rv_{-i}(m)^\top\Rv_{-i}(m)
&=
\begin{bmatrix}
\Rv_{-i}(m_0)^\top\Rv_{-i}(m_0) & \0v\\
\0v & \Iv_{m-m_0}
\end{bmatrix}
+o_p(1),
\nonumber\\
\Rv_{-i}(m)^\top\uh_i
&=
\begin{bmatrix}
\Rv_{-i}(m_0)^\top\uh_i\\
\0v
\end{bmatrix}
+o_p(1).
\end{align*}
Since $\Rv_{-i}(m_0)^\top\Rv_{-i}(m_0)$ converges in probability to a positive-definite matrix by \eqref{eq:pf:RtR lim new},
\begin{equation*}
\{\Rv_{-i}(m)^\top\Rv_{-i}(m)\}^{-1}
\Rv_{-i}(m)^\top\uh_i
=
\begin{bmatrix}
\{\Rv_{-i}(m_0)^\top\Rv_{-i}(m_0)\}^{-1}
\Rv_{-i}(m_0)^\top\uh_i\\
\0v
\end{bmatrix}
+o_p(1).
\end{equation*}
Substituting this expansion into
\eqref{eq:pf:working augmented projection formula}
and comparing with its counterpart for $m_0$, we obtain
\eqref{eq:pf:overspecified augmented projection equivalence},
since both $\Rv_{-i}(m)$ and $\Rv_{-i}(m_0)$ have spectral
norm at most one.

We now compare the estimators $\uhJS_i(m)$ and $\uhJS_i(m_0)$. By Lemma~\ref{lem:augmented target geometry}, the denominator in \eqref{eq:ti hat def} under correct specification converges in probability to $1-\rho_i^2\aiiaug>0$. Thus, \eqref{eq:pf:overspecified augmented projection equivalence} and the consistency of $\rhoh_i(m)$ and $\rhoh_i(m_0)$ give $\hat t_i(m)-\hat t_i(m_0)\convp0$. Recalling
\begin{equation*}
\utilJS_i(m)
=
\hat t_i(m)\proj_{\Ciaug(m)}\uh_i
+
\{1-\hat t_i(m)\}\uh_i
\end{equation*}
and the corresponding expression with $m_0$, we obtain
\begin{align*}
&\norm{\utilJS_i(m)-\utilJS_i(m_0)}\\
&\quad=
\norm{
\hat t_i(m)
\paren{\proj_{\Ciaug(m)}\uh_i-\proj_{\Ciaug(m_0)}\uh_i}
+
\{\hat t_i(m)-\hat t_i(m_0)\}
\paren{\proj_{\Ciaug(m_0)}\uh_i-\uh_i}
}\\
&\quad\le
\norm{\proj_{\Ciaug(m)}\uh_i-\proj_{\Ciaug(m_0)}\uh_i}
+
\abs{\hat t_i(m)-\hat t_i(m_0)}
\convp0.
\end{align*}
Moreover, \eqref{eq:unnormalized oracle approximation} and \eqref{eq:signal projection norm limit} imply that $\norm{\utilJS_i(m_0)}^2$ converges in probability to a positive constant. Normalization therefore yields
\begin{equation*}
\norm{\uhJS_i(m)-\uhJS_i(m_0)}\convp0,
\qquad i\in[m_0].
\end{equation*}

The Gram matrix of $\uhJS_1(m_0),\ldots,\uhJS_{m_0}(m_0)$ has a positive-definite probability limit, as established in the proof of Theorem~\ref{thm:main thm-1}. Hence Lemma~\ref{lem:subspace perturbation} and the Gram--Schmidt argument in the proof of Proposition~\ref{prop:orthogonalized frame} give
\begin{align}
S(\UchJS_k(m),\UchJS_k(m_0))
&\convp1,
\qquad k\in[m_0],
\label{eq:pf:overspecified leading eigenspace equivalence}\\
\norm{\qhJS_i(m)-\qhJS_i(m_0)}
&\convp0,
\qquad i\in[m_0].
\nonumber
\end{align}

\medskip
\noindent\emph{Non-spiked components under overspecification.}
Let $m_0<i\le m$. We first establish the convergence
\begin{align}
&\norm{\rhoh_i(m)^{-2}\utilJS_i(m)-\uh_i}
\nonumber\\
&\quad=
\norm{
\hat t_i(m)\rhoh_i(m)^{-2}\proj_{\Ciaug(m)}\uh_i
+
\bracket{\rhoh_i(m)^{-2}\{1-\hat t_i(m)\}-1}\uh_i
}
\nonumber\\
&\quad\le
\rhoh_i(m)^{-2}\norm{\proj_{\Ciaug(m)}\uh_i}
+
\abs{\rhoh_i(m)^{-2}\{1-\hat t_i(m)\}-1}
\convp0.
\label{eq:pf:overspecified additional unnormalized approximation}
\end{align}
To prove the convergence in \eqref{eq:pf:overspecified additional unnormalized approximation}, we use the following limit, whose proof is deferred to the end:
\begin{equation}
\rhoh_i(m)^{-2}\norm{\projC\uh_i}\convp0.
\label{eq:pf:projC uhi norm rhoh sq ratio}
\end{equation}
By the orthonormality of the sample eigenvectors,
\begin{equation*}
\norm{\Rv_{-i}(m)^\top\uh_i}^2
=
\sum_{j\in[m]\setminus\{i\}}
\inner{\uh_j,\projC\uh_i}^2
\le
\norm{\projC\uh_i}^2.
\end{equation*}
Using the orthogonal direct sum $\Ciaug(m)=\Cc\oplus\operatorname{col}(\Rv_{-i}(m))$, we have
\begin{align*}
\frac{\norm{\proj_{\Ciaug(m)}\uh_i}^2}{\rhoh_i(m)^4}
&=
\frac{\norm{\projC\uh_i}^2}{\rhoh_i(m)^4}
+
\frac{
\{\Rv_{-i}(m)^\top\uh_i\}^\top
\{\Rv_{-i}(m)^\top\Rv_{-i}(m)\}^{-1}
\Rv_{-i}(m)^\top\uh_i
}{
\rhoh_i(m)^4
}\\
&\le
\left[
1+\norm{\{\Rv_{-i}(m)^\top\Rv_{-i}(m)\}^{-1}}
\right]
\frac{\norm{\projC\uh_i}^2}{\rhoh_i(m)^4}
\convp0.
\end{align*}
The convergence follows from \eqref{eq:pf:projC uhi norm rhoh sq ratio} and
$\norm{\{\Rv_{-i}(m)^\top\Rv_{-i}(m)\}^{-1}}=O_p(1)$. Thus, the first term in the upper bound in \eqref{eq:pf:overspecified additional unnormalized approximation} converges to zero in probability.

For the second term, using $0<\rhoh_i(m)\le1$ and $\rhoh_i(m)^{-2}\|\proj_{\Ciaug(m)}\uh_i\|\convp0$ gives
\begin{equation*}
\rhoh_i(m)^{-2}\{1-\hat t_i(m)\}
=
\rhoh_i(m)^{-2}\Pi_{[0,1]}
\frac{
\rhoh_i(m)^2-\norm{\proj_{\Ciaug(m)}\uh_i}^2
}{
1-\norm{\proj_{\Ciaug(m)}\uh_i}^2
}
\convp1.
\end{equation*}
The upper bound in \eqref{eq:pf:overspecified additional unnormalized approximation} therefore converges to zero in probability.

We next consider $\qhJS_i(m)$. For every $j<i$, we have $\uh_i\in\Cc_j^{\rm aug}(m)$, and therefore
\begin{equation*}
\inner{\proj_{\Cc_j^{\rm aug}(m)}\uh_j,\uh_i}
=
\inner{\uh_j,\proj_{\Cc_j^{\rm aug}(m)}\uh_i}
=
\inner{\uh_j,\uh_i}
=
0.
\end{equation*}
Hence $\inner{\utilJS_j(m),\uh_i}=0$ for every $j<i$, which gives $\proj_{\UchJS_{i-1}(m)}\uh_i=\0v$. By \eqref{eq:pf:overspecified additional unnormalized approximation},
\begin{align*}
&\norm{
\rhoh_i(m)^{-2}
(\Iv_p-\Pv_{\UchJS_{i-1}(m)})\utilJS_i(m)
-\uh_i
}\\
&\qquad=
\norm{
(\Iv_p-\Pv_{\UchJS_{i-1}(m)})
\{\rhoh_i(m)^{-2}\utilJS_i(m)-\uh_i\}
}
\convp0.
\end{align*}
Normalization therefore gives
\begin{equation}
\norm{\qhJS_i(m)-\uh_i}
=
\norm{
\frac{
\rhoh_i(m)^{-2}(\Iv_p-\Pv_{\UchJS_{i-1}(m)})\utilJS_i(m)
}{
\norm{\rhoh_i(m)^{-2}(\Iv_p-\Pv_{\UchJS_{i-1}(m)})\utilJS_i(m)}
}
-\uh_i
}
\convp0,
\qquad m_0<i\le m.
\label{eq:pf:overspecified additional frame equivalence}
\end{equation}
This proves the assertion of Proposition~\ref{prop:number spike misspecification}-(ii) for $m_0<i\le m$.

We finally prove the eigenspace comparison. Let $m_0<k\le m$. Since $\qhJS_1(m),\ldots,\qhJS_k(m)$ and $\uh_1,\ldots,\uh_k$ are orthonormal bases of $\UchJS_k(m)$ and $\Uch_k$, respectively, \eqref{eq:def subspace similarity} gives
\begin{equation*}
kS(\UchJS_k(m),\Uc_k)^2
=
\sum_{i=1}^k\norm{\proj_{\Uc_k}\qhJS_i(m)}^2,\quad
kS(\Uch_k,\Uc_k)^2
=
\sum_{i=1}^k\norm{\proj_{\Uc_k}\uh_i}^2.
\end{equation*}
By \eqref{eq:pf:overspecified additional frame equivalence} and the reverse triangle inequality,
\begin{align*}
&k\left\{S(\UchJS_k(m),\Uc_k)^2-S(\Uch_k,\Uc_k)^2\right\}\\
&\quad=
\sum_{i=1}^{m_0}
\left\{
\norm{\proj_{\Uc_k}\qhJS_i(m)}^2
-\norm{\proj_{\Uc_k}\uh_i}^2
\right\}
+
\sum_{i=m_0+1}^{k}
\left\{
\norm{\proj_{\Uc_k}\qhJS_i(m)}^2
-\norm{\proj_{\Uc_k}\uh_i}^2
\right\}\\
&\quad=
\sum_{i=1}^{m_0}
\left\{
\norm{\proj_{\Uc_k}\qhJS_i(m)}^2
-\norm{\proj_{\Uc_k}\uh_i}^2
\right\}
+o_p(1).
\end{align*}
Since $\Uc_{m_0}\subset\Uc_k$, we have
\begin{align*}
\sum_{i=1}^{m_0}\norm{\proj_{\Uc_k}\qhJS_i(m)}^2
&\ge
\sum_{i=1}^{m_0}\norm{\proj_{\Uc_{m_0}}\qhJS_i(m)}^2\\
&=
m_0S(\UchJS_{m_0}(m),\Uc_{m_0})^2
\convp
\sum_{i=1}^{m_0}G_{\rho_i}(\aiiaug),
\end{align*}
where the convergence follows from \eqref{eq:pf:overspecified leading eigenspace equivalence} and Theorem~\ref{thm:main thm-1}-(ii). For the PCA term, Lemma~\ref{lem:eigenpair asymptotics RMT} gives
\begin{equation*}
\sum_{i=1}^{m_0}\norm{\proj_{\Uc_k}\uh_i}^2
=
\sum_{i=1}^{m_0}\sum_{j=1}^k\inner{\uh_i,\uv_j}^2
\convp
\sum_{i=1}^{m_0}\rho_i^2.
\end{equation*}
Therefore,
\begin{equation*}
S(\UchJS_k(m),\Uc_k)^2-S(\Uch_k,\Uc_k)^2
\ge
\frac{1}{k}\sum_{i=1}^{m_0}
\left[G_{\rho_i}\{\aiiaug\}-\rho_i^2\right]
+o_p(1),
\end{equation*}
which proves the assertion of Proposition~\ref{prop:eigenspace spike misspecification}-(ii) for $m_0<k\le m$.

\medskip
\noindent\emph{Verification of \eqref{eq:pf:projC uhi norm rhoh sq ratio}.}
Fix $0<\eta<1/36$ and $i\in\{m_0+1,\ldots,m\}$. By Theorem~2.17 of \cite{bloemendal2016principal}, for every deterministic unit vector $\cv\in\Rb^p$,
\begin{equation*}
\inner{\cv,\uh_i}^2=O_p(n^{-1+\eta}).
\end{equation*}
Let $\cv_1,\ldots,\cv_r$ be an orthonormal basis of $\Cc$. Since $r$ is fixed,
\begin{equation}
\norm{\projC\uh_i}^2
=
\sum_{j=1}^r\inner{\cv_j,\uh_i}^2
=
O_p(n^{-1+\eta}).
\label{eq:pf:noise projection rate}
\end{equation}
We next bound $\rhoh_i(m)^{-2}$. By definition,
\begin{equation}
\rhoh_i(m)^{-2}
=
1+
\frac{\gamma_n(m)\tilde\lambda_i(m)}{p-m}
\sum_{j=m+1}^p
\frac{\lambdah_j}{(\lambdah_i-\lambdah_j)^2},
\label{eq:pf:noise rho inverse representation}
\end{equation}
where $\gamma_n(m)=(p-m)/(n-m)$ and
\begin{equation*}
0<
\tilde\lambda_i(m)
=
\lambdah_i
\paren{
1+
\frac{\gamma_n(m)}{p-m}
\sum_{j=m+1}^p
\frac{\lambdah_j}{\lambdah_i-\lambdah_j}
}^{-1}
\le
\lambdah_i
=
O_p(1).
\end{equation*}
To bound the sum in \eqref{eq:pf:noise rho inverse representation}, we obtain lower bounds for the gaps $\lambdah_i-\lambdah_{m+k}$, $1\le k\le p-m$. Since $\lambdah_i\ge\lambdah_m$, it suffices to bound $\lambdah_m-\lambdah_{m+k}$ from below.

Let $q_{k,n}$ denote the upper $k/p$ quantile of the Marchenko--Pastur distribution with aspect ratio $p/n$ and scale $\tau^2$, and set $\gamma'=\min(\gamma,1)/4$. By Theorems~2.7 and~3.5 of \cite{bloemendal2016principal}, the bounds
\begin{align}
\abs{\lambdah_m-q_{m-m_0,n}}
&\le n^{-2/3+\eta},
\nonumber\\
\abs{\lambdah_{m+k}-q_{m-m_0+k,n}}
&\le n^{-2/3+\eta}(m-m_0+k)^{-1/3}
\label{eq:pf:noise eigenvalue rigidity}
\end{align}
hold simultaneously for all $1\le k\le\gamma'n$ with probability tending to one. Moreover, equation~(6.17) of \cite{bloemendal2016principal} gives a constant $c>0$ such that
\begin{equation*}
q_{j,n}-q_{j+1,n}
\ge
c(m-m_0)^{1/3}n^{-2/3}j^{-1/3},
\qquad 1\le j\le2\gamma'n.
\end{equation*}
Thus, for all sufficiently large $n$ and every $1\le k\le\gamma'n$,
\begin{align*}
q_{m-m_0,n}-q_{m-m_0+k,n}
&=
\sum_{j=m-m_0}^{m-m_0+k-1}(q_{j,n}-q_{j+1,n})\ge
\sum_{j=m-m_0}^{m-m_0+k-1}
c(m-m_0)^{1/3}n^{-2/3}j^{-1/3}\\
&\ge
c(m-m_0)^{1/3}n^{-2/3}k(m-m_0+k-1)^{-1/3}
\ge
cn^{-2/3}k^{2/3}.
\end{align*}
On the event where \eqref{eq:pf:noise eigenvalue rigidity} holds, we therefore have, for every $1\le k\le\gamma'n$,
\begin{align}
\lambdah_m-\lambdah_{m+k}
&\ge
q_{m-m_0,n}-q_{m-m_0+k,n}
-2n^{-2/3+\eta}
\nonumber\\
&\ge
cn^{-2/3}k^{2/3}-2n^{-2/3+\eta}.
\label{eq:pf:noise eigenvalue gap}
\end{align}
We next obtain lower bounds over three ranges of $k$: $1\le k<n^{3\eta}$, $n^{3\eta}\le k\le\gamma'n$, and $\gamma'n<k\le p-m$.

For $n^{3\eta}\le k\le \gamma' n $, $(2n^{-2/3+\eta})/(cn^{-2/3}k^{2/3})\le 2n^{-\eta}/c\to 0$ as $n\to\infty$.
Consequently, \eqref{eq:pf:noise eigenvalue gap} gives, with probability tending to one,
\begin{equation}
\lambdah_m-\lambdah_{m+k}
\ge
(c/2)n^{-2/3}k^{2/3},
\qquad
n^{3\eta}\le k\le\gamma'n.
\label{eq:pf:noise medium k eigenvalue gap}
\end{equation}

For $1\le k<n^{3\eta}$, Lemma~6.4 and Theorem~2.7 of \cite{bloemendal2016principal} give $\lambdah_m-\lambdah_{m+1}\ge n^{-2/3-\eta}$ with probability tending to one. On the same event,
\begin{equation}
\lambdah_m-\lambdah_{m+k}
\ge
\lambdah_m-\lambdah_{m+1}
\ge
n^{-2/3-\eta},
\qquad
1\le k<n^{3\eta}.
\label{eq:pf:noise small k eigenvalue gap}
\end{equation}

For the remaining indices $\gamma'n<k\le p-m$, \eqref{eq:pf:noise medium k eigenvalue gap} gives, for all sufficiently large $n$ and with probability tending to one,
\begin{equation}
\lambdah_m-\lambdah_{m+k}
\ge
\lambdah_m-\lambdah_{m+\lfloor\gamma'n\rfloor}
\ge
(c/2)n^{-2/3}\lfloor\gamma'n\rfloor^{2/3}
\ge
(c/4)(\gamma')^{2/3}.
\label{eq:pf:noise remaining k eigenvalue gap}
\end{equation}

Since $\lambdah_i\ge\lambdah_m$, combining \eqref{eq:pf:noise small k eigenvalue gap}, \eqref{eq:pf:noise medium k eigenvalue gap}, and \eqref{eq:pf:noise remaining k eigenvalue gap} shows that the inequalities
\begin{equation*}
\lambdah_i-\lambdah_{m+k}
\ge
\begin{cases}
n^{-2/3-\eta},
&1\le k<n^{3\eta},\\
(c/2)n^{-2/3}k^{2/3},
&n^{3\eta}\le k\le\gamma'n,\\
(c/4)(\gamma')^{2/3},
&\gamma'n<k\le p-m
\end{cases}
\end{equation*}
hold simultaneously for all $1\le k\le p-m$ with probability tending to one.
Together with $\lambdah_{m+k}\le\lambdah_1=O_p(1)$, these bounds yield
\begin{align*}
\frac{1}{p-m}\sum_{k=1}^{p-m}
\frac{\lambdah_{m+k}}{(\lambdah_i-\lambdah_{m+k})^2}
&=
O_p\left[
\frac{1}{n}
\left\{
n^{3\eta}n^{4/3+2\eta}
+
n^{4/3}
\sum_{k=\lceil n^{3\eta}\rceil}^{\lfloor\gamma'n\rfloor}
k^{-4/3}
+
n
\right\}
\right]\\
&=
O_p(n^{1/3+5\eta}).
\end{align*}
The last equality uses $\sum_{k=\lceil n^{3\eta}\rceil}^{\infty}k^{-4/3}=O(n^{-\eta})$. Since $\gamma_n(m)\tilde\lambda_i(m)=O_p(1)$, it follows from \eqref{eq:pf:noise rho inverse representation} that $\rhoh_i(m)^{-2}=O_p(n^{1/3+5\eta})$. Finally, combining this bound with \eqref{eq:pf:noise projection rate} gives
\begin{equation*}
\rhoh_i(m)^{-2}\norm{\projC\uh_i}
=
O_p(n^{-1/6+6\eta})
=
o_p(1),
\end{equation*}
where the last equality follows from $\eta<1/36$. This proves \eqref{eq:pf:projC uhi norm rhoh sq ratio} and completes the proof.
\end{proof}

\subsection{Proofs and technical details for Section~\ref{subsec: target dimension}}
\label{subsec:proof increasing target dimension}

This section proves the results in Section~\ref{subsec: target dimension} and extends the augmented shrinkage construction to increasing-dimensional target subspaces.

For $m\ge2$ and $i\in[m]$, define
\begin{equation}
\aiiaug(\delta)
\coloneq
a_{ii}
+\avmii^\top\Dvmi
\left[
\Iv_{m-1}-\Dvmi\Avmi\Dvmi
-\delta\curly{\Iv_{m-1}-(\Dvmi)^2}
\right]^{-1}
\Dvmi\avmii.
\label{eq:aii aug increasing target dimension}
\end{equation}
When $m=1$, set $\aiiaug(\delta)=a_{11}$. The following lemma extends Lemma~\ref{lem:augmented target geometry} to the increasing-dimensional setting of Lemma~\ref{lem target projected geometry}.

\begin{lemma}
\label{lem:augmented target geometry increasing dimension}
Under the conditions of Lemma~\ref{lem target projected geometry}, for each $i\in[m]$,
\begin{equation*}
\norm{\projCiaug\uv_i}^2\convp\aiiaug(\delta),
\end{equation*}
where $\aiiaug(\delta)$ satisfies $a_{ii}\le\aiiaug(\delta)\le1,$
and the first inequality is strict if and only if $a_{ij}\ne0$ for some $j\ne i$. Furthermore,
\begin{equation}
\inner{\projCiaug\uh_i,\uv_i}
\convp
\rho_i\aiiaug(\delta),
\qquad
\norm{\projCiaug\uh_i}^2
\convp
\rho_i^2\aiiaug(\delta)+\delta(1-\rho_i^2).
\label{eq:proj Caug geometry increasing target dimension}
\end{equation}
\end{lemma}

By Lemma~\ref{lem:augmented target geometry increasing dimension}, the limiting squared alignment can be expressed using $F_{\rho,\delta}$ from \eqref{eq rho delta def} as
\begin{equation*}
\inner{\uhiJS(t),\uv_i}^2
\convp
F_{\rho_i,\delta}\bigl(t;\aiiaug(\delta)\bigr).
\end{equation*}
Consequently, the limiting squared alignment is maximized over $t\in[0,1]$ at $t_{\rho_i,\delta}^*\bigl(\aiiaug(\delta)\bigr)$, where $t_{\rho,\delta}^*(a)$ is defined in Section~\ref{subsec: target dimension}. This implies that shrinkage is beneficial if and only if $\aiiaug(\delta)>\delta$.

As in the unaugmented case, set
\begin{equation*}
\hat a_{ii}^{\rm aug}
=
\frac{\norm{\projCiaug\uh_i}^2-\hat\delta(1-\rhoh_i^2)}{\rhoh_i^2},
\end{equation*}
where $\hat\delta=r_p/p$. Define the corresponding shrinkage parameter by
\begin{equation*}
\hat t_{i,\delta}^{\rm aug}
=
\begin{cases}
\displaystyle
\Pi_{[0,1]}\frac{
(\hat a_{ii}^{\rm aug}-\hat\delta)(1-\rhoh_i^2)
}{
\hat a_{ii}^{\rm aug}\curly{1-\norm{\projCiaug\uh_i}^2}
},
&\hat a_{ii}^{\rm aug}>\hat\delta,\\[3mm]
0,
&\hat a_{ii}^{\rm aug}\le\hat\delta.
\end{cases}
\end{equation*}
The following proposition is the augmented counterpart of Proposition~\ref{prop adjusted JS}.

\begin{proposition}
\label{prop adjusted JS augmented}
Under the conditions of Lemma~\ref{lem target projected geometry}, for each $i\in[m]$ such that $(\aiiaug(\delta),\delta)\ne(0,0)$,
\begin{equation*}
\hat t_{i,\delta}^{\rm aug}
\convp
t_{\rho_i,\delta}^*\bigl(\aiiaug(\delta)\bigr).
\end{equation*}
For every $i\in[m]$,
\begin{equation*}
\plim_{n,p\to\infty}
\inner{\uhiJS(\hat t_{i,\delta}^{\rm aug}),\uv_i}^2
\ge
\plim_{n,p\to\infty}
\inner{\uh_i,\uv_i}^2.
\end{equation*}
The inequality is strict if and only if $\aiiaug(\delta)>\delta$.
\end{proposition}

When $a_{ii}\le\delta<\aiiaug(\delta)$, the augmented estimator therefore strictly improves upon PCA, whereas the unaugmented estimator has the same limiting squared alignment as PCA by Proposition~\ref{prop adjusted JS}.

\begin{proof}[Proof of Lemmas~\ref{lem target projected geometry} and \ref{lem:augmented target geometry increasing dimension}]
We first determine, for $j,k\in[m]$, the limits of the two inner products $\langle\projC\uh_j,\projC\uh_k\rangle$ and $\langle\projC\uh_j,\uv_k\rangle$. Define
\begin{equation*}
\Cc_1=\vspan\{\projC\uv_j:j\in[m]\},
\qquad
\Cc_0=\Cc\cap\Uc_m^\perp.
\end{equation*}
For every $\cv\in\Cc$ and $j\in[m]$, we have $\langle\cv,\projC\uv_j\rangle=\langle\cv,\uv_j\rangle$. Hence, if $\cv\in\Cc$ satisfies $\cv\perp\Cc_1$, then $\cv\in\Cc_0$. Thus, we have the orthogonal direct sum
\begin{equation}
\Cc=\Cc_1\oplus\Cc_0,
\qquad
\dim(\Cc_1)\le m,
\qquad
\dim(\Cc_0)=r_p-\dim(\Cc_1).
\label{eq:pf:increasing target orthogonal decomposition}
\end{equation}
Moreover, $\proj_{\Cc_1}\uv_j=\projC\uv_j$ for every $j\in[m]$, since $\projC\uv_j\in\Cc_1$ and $\uv_j-\projC\uv_j$ is orthogonal to $\Cc$. Since $\Cc_1$ is deterministic and has dimension at most $m$, the argument in the proof of Lemma~\ref{lem:projC uhi RMT} gives
\begin{equation}
\norm{\proj_{\Cc_1}\uh_j-\rho_j\projC\uv_j}\convp0,
\qquad j\in[m].
\label{eq:pf:increasing target finite projection}
\end{equation}

Now, we consider the component $\proj_{\Cc_0}\uh_j$. Let $\Gv\in\Rb^{p\times p}$ be any orthogonal matrix satisfying $\Gv\Uv_m=\Uv_m$. Equation~(47) in the proof of Theorem~6 of \cite{paul2007asymptotics} gives
\begin{equation}
(\Gv\Pv_{\Uc_m^\perp}\uh_1,\ldots,\Gv\Pv_{\Uc_m^\perp}\uh_m)
\stackrel{d}{=}
(\Pv_{\Uc_m^\perp}\uh_1,\ldots,\Pv_{\Uc_m^\perp}\uh_m).
\label{eq:pf:increasing target joint rotation}
\end{equation}
Fix a nonzero deterministic vector $\yv\in\Rb^m$. The orthonormality of the sample eigenvectors and Lemma~\ref{lem:eigenpair asymptotics RMT} give
\begin{equation}
\norm{\proj_{\Uc_m^\perp}\Uvh_m\yv}^2
=
\norm{\yv}^2-\norm{\Uv_m^\top\Uvh_m\yv}^2
\convp
\yv^\top(\Iv_m-\Dv^2)\yv>0.
\label{eq:pf:increasing target linear combination norm}
\end{equation}
By \eqref{eq:pf:increasing target joint rotation}, the normalized vector $\proj_{\Uc_m^\perp}\Uvh_m\yv/\|\proj_{\Uc_m^\perp}\Uvh_m\yv\|$ is uniformly distributed on the unit sphere in $\Uc_m^\perp$. Since $\Cc_0\subset\Uc_m^\perp$, we have, for independent standard normal random variables $g_1,\ldots,g_{p-m}$,
\begin{equation}
\norm{
\proj_{\Cc_0}
\frac{\proj_{\Uc_m^\perp}\Uvh_m\yv}{\norm{\proj_{\Uc_m^\perp}\Uvh_m\yv}}
}^2
\stackrel{d}{=}
\frac{\sum_{j=1}^{\dim(\Cc_0)}g_j^2}{\sum_{j=1}^{p-m}g_j^2}.
\label{eq:pf:increasing target Gaussian ratio}
\end{equation}
By \eqref{eq:pf:increasing target orthogonal decomposition} and \eqref{eq target assumptions}, $\{p-m-\dim(\Cc_0)\}/(p-m)\to1-\delta>0$. The law of large numbers therefore gives
\begin{align*}
\frac{\sum_{j=1}^{\dim(\Cc_0)}g_j^2}{\sum_{j=1}^{p-m}g_j^2}
=
1-
\frac{p-m-\dim(\Cc_0)}{p-m}
\frac{
(p-m-\dim(\Cc_0))^{-1}
\sum_{j=\dim(\Cc_0)+1}^{p-m}g_j^2
}{
(p-m)^{-1}\sum_{j=1}^{p-m}g_j^2
}
\convp\delta.
\end{align*}
Since $\Cc_0\subset\Uc_m^\perp$, combining the ratio convergence with \eqref{eq:pf:increasing target linear combination norm} and \eqref{eq:pf:increasing target Gaussian ratio} yields
\begin{equation*}
\yv^\top\Uvh_m^\top\Pv_{\Cc_0}\Uvh_m\yv
=
\norm{\proj_{\Cc_0}\Uvh_m\yv}^2
\convp
\delta\yv^\top(\Iv_m-\Dv^2)\yv.
\end{equation*}
Since $m$ is fixed and $\yv$ is arbitrary, 
\begin{equation}
\Uvh_m^\top\Pv_{\Cc_0}\Uvh_m
\convp
\delta(\Iv_m-\Dv^2).
\label{eq:pf:increasing target bulk Gram}
\end{equation}
Using $\Pv_{\Cc_0}\Uv_m=\0v$, we combine \eqref{eq:pf:increasing target finite projection} and \eqref{eq:pf:increasing target bulk Gram} with \eqref{eq target assumptions} and $\Cc=\Cc_1\oplus\Cc_0$ to obtain
\begin{equation}
\begin{aligned}
\Uvh_m^\top\Pv_\Cc\Uvh_m
&\convp\Dv\Av\Dv+\delta(\Iv_m-\Dv^2),\\
\Uvh_m^\top\Pv_\Cc\Uv_m
&\convp\Dv\Av.
\end{aligned}
\label{eq:pf:increasing target Gram limits}
\end{equation}
Taking diagonal entries proves Lemma~\ref{lem target projected geometry}.

We next prove Lemma~\ref{lem:augmented target geometry increasing dimension}. Fix $i\in[m]$. If $m=1$, then $\Ciaug=\Cc$ and $\aiiaug(\delta)=a_{11}$, so the assertions follow from \eqref{eq target assumptions} and Lemma~\ref{lem target projected geometry}. Suppose henceforth that $m\ge2$. The orthonormality of the sample eigenvectors, Lemma~\ref{lem:eigenpair asymptotics RMT}, and \eqref{eq:pf:increasing target Gram limits} give
\begin{equation}
\begin{aligned}
\Rvmi^\top\Rvmi
&\convp
\Iv_{m-1}-\Dvmi\Avmi\Dvmi
-\delta\curly{\Iv_{m-1}-(\Dvmi)^2},\\
\Rvmi^\top\uv_i
&\convp-\Dvmi\avmii,
\qquad
\Rvmi^\top\uh_i\convp-\rho_i\Dvmi\avmii.
\end{aligned}
\label{eq:pf:increasing augmented residual limits}
\end{equation}
The probability limit of $\Rvmi^\top\Rvmi$ is positive definite. Indeed, $\Avmi\preceq\Iv_{m-1}$, $\delta<1$, and $0<\rho_j<1$ for every $j\in\mmi$ imply
\begin{align*}
&\Iv_{m-1}-\Dvmi\Avmi\Dvmi
-\delta\curly{\Iv_{m-1}-(\Dvmi)^2}\\
&\qquad=
\Dvmi(\Iv_{m-1}-\Avmi)\Dvmi
+(1-\delta)\curly{\Iv_{m-1}-(\Dvmi)^2}
\succ\0v.
\end{align*}
Using the orthogonal direct sum $\Ciaug=\Cc\oplus\operatorname{col}(\Rvmi)$, we have
\begin{equation}
\Pv_{\Ciaug}
=
\Pv_\Cc+\Rvmi(\Rvmi^\top\Rvmi)^{-1}\Rvmi^\top.
\label{eq:pf:increasing augmented projection formula}
\end{equation}
Using \eqref{eq:pf:increasing augmented residual limits}, \eqref{eq:pf:increasing augmented projection formula}, and \eqref{eq target assumptions}, we obtain
\begin{align}
\norm{\projCiaug\uv_i}^2
&=
\norm{\projC\uv_i}^2
+(\Rvmi^\top\uv_i)^\top(\Rvmi^\top\Rvmi)^{-1}\Rvmi^\top\uv_i
\nonumber\\
&\convp
a_{ii}+\avmii^\top\Dvmi
\left[\Iv_{m-1}-\Dvmi\Avmi\Dvmi-\delta\curly{\Iv_{m-1}-(\Dvmi)^2}\right]^{-1}
\Dvmi\avmii
\nonumber\\
&=\aiiaug(\delta).
\label{eq:pf:increasing augmented population projection}
\end{align}
The matrix in brackets in \eqref{eq:pf:increasing augmented population projection} is positive definite, and $\Dvmi$ is invertible. Hence the quadratic term in the probability limit is nonnegative and vanishes if and only if $\avmii=\0v$. Therefore, $a_{ii}\le\aiiaug(\delta)$, with strict inequality if and only if $a_{ij}\ne0$ for some $j\ne i$. Moreover, $\|\projCiaug\uv_i\|^2\le1$ implies $\aiiaug(\delta)\le1$.

It remains to prove the two limits in \eqref{eq:proj Caug geometry increasing target dimension}. By \eqref{eq:pf:increasing augmented projection formula}, Lemma~\ref{lem target projected geometry}, and \eqref{eq:pf:increasing augmented residual limits}, together with the definition \eqref{eq:aii aug increasing target dimension},
\begin{align*}
\inner{\projCiaug\uh_i,\uv_i}
&=
\inner{\projC\uh_i,\uv_i}
+(\Rvmi^\top\uh_i)^\top(\Rvmi^\top\Rvmi)^{-1}\Rvmi^\top\uv_i\\
&\convp
\rho_i a_{ii}+\rho_i\curly{\aiiaug(\delta)-a_{ii}}
=
\rho_i\aiiaug(\delta).
\end{align*}
Similarly,
\begin{align*}
\norm{\projCiaug\uh_i}^2
&=
\norm{\projC\uh_i}^2
+(\Rvmi^\top\uh_i)^\top(\Rvmi^\top\Rvmi)^{-1}\Rvmi^\top\uh_i\\
&\convp
\rho_i^2a_{ii}+\delta(1-\rho_i^2)
+\rho_i^2\curly{\aiiaug(\delta)-a_{ii}}\\
&=
\rho_i^2\aiiaug(\delta)+\delta(1-\rho_i^2).
\end{align*}
This proves \eqref{eq:proj Caug geometry increasing target dimension}.
\end{proof}

We now turn to Propositions~\ref{prop adjusted JS} and \ref{prop adjusted JS augmented}.

\begin{proof}[Proof of Propositions~\ref{prop adjusted JS} and \ref{prop adjusted JS augmented}]
We first determine the maximizer of $F_{\rho,\delta}(t;a)$ over $t\in[0,1]$, which will be used for both propositions. Suppose first that $(a,\delta)\ne(0,0)$. Recall from \eqref{eq rho delta def} that
\begin{equation*}
F_{\rho,\delta}(t;a)
=
\frac{
\rho^2\{1-t(1-a)\}^2
}{
\rho^2a+\delta(1-\rho^2)
+(1-t)^2\{1-\rho^2a-\delta(1-\rho^2)\}
}.
\end{equation*}
The denominator is positive for every $t\in[0,1]$. Differentiating gives
\begin{equation}
\frac{\partial F_{\rho,\delta}(t;a)}{\partial t}
=
\frac{
2\rho^2\{1-t(1-a)\}
\left[(1-\rho^2)(a-\delta)-ta\{1-\rho^2a-\delta(1-\rho^2)\}\right]
}{
\left[\rho^2a+\delta(1-\rho^2)
+(1-t)^2\{1-\rho^2a-\delta(1-\rho^2)\}\right]^2
}.
\label{eq:pf:F derivative}
\end{equation}
Suppose $a>\delta$. Since
\begin{equation*}
1-\rho^2a-\delta(1-\rho^2)
=
\rho^2(1-a)+(1-\delta)(1-\rho^2)>0,
\end{equation*}
the bracketed factor in the numerator of \eqref{eq:pf:F derivative} is strictly decreasing in $t$ and vanishes at
\begin{equation*}
t_{\rho,\delta}^*(a)
=
\frac{(a-\delta)(1-\rho^2)}{a\{1-\rho^2a-\delta(1-\rho^2)\}}.
\end{equation*}
Moreover,
\begin{equation*}
1-t_{\rho,\delta}^*(a)
=
\frac{(1-a)\{\rho^2a+\delta(1-\rho^2)\}}{a\{1-\rho^2a-\delta(1-\rho^2)\}}
\ge0,
\end{equation*}
so $t_{\rho,\delta}^*(a)\in(0,1]$. Since the other factors in \eqref{eq:pf:F derivative} are positive, $F_{\rho,\delta}(t;a)$ is strictly increasing for $t<t_{\rho,\delta}^*(a)$ and strictly decreasing for $t>t_{\rho,\delta}^*(a)$. Hence $t_{\rho,\delta}^*(a)$ is the unique maximizer over $t\in[0,1]$.

If $a\le\delta$, the bracketed factor in the numerator of \eqref{eq:pf:F derivative} is negative for $0<t<1$. Thus $F_{\rho,\delta}(t;a)$ is strictly decreasing on $[0,1]$, and $t_{\rho,\delta}^*(a)=0$ is the unique maximizer. Finally, when $a=\delta=0$, the definition of $F_{\rho,\delta}$ gives $F_{\rho,0}(t;0)=\rho^2$ for every $t\in[0,1)$, so $t_{\rho,0}^*(0)=0$ is also a maximizer.

Thus $t_{\rho,\delta}^*(a)$ is a maximizer, unique unless $a=\delta=0$, and
\begin{equation}
F_{\rho,\delta}\bigl(t_{\rho,\delta}^*(a);a\bigr)
\ge
F_{\rho,\delta}(0;a)
=
\rho^2,
\label{eq:pf:dimension adjusted optimal comparison}
\end{equation}
with strict inequality if and only if $a>\delta$.

We now prove Proposition~\ref{prop adjusted JS}. Fix $i\in[m]$ and suppose first that $(a_{ii},\delta)\ne(0,0)$. By Lemmas~\ref{lem:lambda rho estimation} and \ref{lem target projected geometry}, together with $\hat\delta\to\delta$,
\begin{equation*}
\hat a_{ii}
=
\frac{\norm{\projC\uh_i}^2-\hat\delta(1-\rhoh_i^2)}{\rhoh_i^2}
\convp a_{ii}.
\end{equation*}
The definition of $\hat t_{i,\delta}$ then yields $\hat t_{i,\delta}\convp t_{\rho_i,\delta}^*(a_{ii})$. Consequently, Lemmas~\ref{lem:eigenpair asymptotics RMT} and \ref{lem target projected geometry} give
\begin{align*}
\inner{\uh_{i,\Cc}^{\rm JS}(\hat t_{i,\delta}),\uv_i}^2
&=
\frac{
\left\{\hat t_{i,\delta}\inner{\projC\uh_i,\uv_i}
+(1-\hat t_{i,\delta})\inner{\uh_i,\uv_i}\right\}^2
}{
\norm{\projC\uh_i}^2
+(1-\hat t_{i,\delta})^2\{1-\norm{\projC\uh_i}^2\}
}\\
&\convp
F_{\rho_i,\delta}\bigl(t_{\rho_i,\delta}^*(a_{ii});a_{ii}\bigr)
\ge\rho_i^2,
\end{align*}
where the inequality is strict if and only if $a_{ii}>\delta$ by \eqref{eq:pf:dimension adjusted optimal comparison}.

It remains to consider $a_{ii}=\delta=0$. By Lemma~\ref{lem target projected geometry}, $\|\projC\uh_i\|\convp0$. When $\hat a_{ii}>\hat\delta$, the definition of $\hat t_{i,\delta}$ gives
\begin{equation*}
\hat t_{i,\delta}
\le
\frac{1-\rhoh_i^2}{1-\norm{\projC\uh_i}^2}
\convp
1-\rho_i^2<1,
\end{equation*}
whereas $\hat t_{i,\delta}=0$ otherwise. Hence $(1-\hat t_{i,\delta})^{-1}=O_p(1)$, and
\begin{equation*}
\norm{
\frac{\hat t_{i,\delta}\projC\uh_i+(1-\hat t_{i,\delta})\uh_i}{1-\hat t_{i,\delta}}
-\uh_i
}
=
\frac{\hat t_{i,\delta}}{1-\hat t_{i,\delta}}
\norm{\projC\uh_i}
\convp0.
\end{equation*}
Normalization gives $\|\uh_{i,\Cc}^{\rm JS}(\hat t_{i,\delta})-\uh_i\|\convp0$. Lemma~\ref{lem:eigenpair asymptotics RMT} therefore yields $\langle\uh_{i,\Cc}^{\rm JS}(\hat t_{i,\delta}),\uv_i\rangle^2\convp\rho_i^2$, completing the proof of Proposition~\ref{prop adjusted JS}.

For Proposition~\ref{prop adjusted JS augmented}, we apply the preceding argument with $\projC$ and $a_{ii}$ replaced by $\projCiaug$ and $\aiiaug(\delta)$, respectively. Lemma~\ref{lem:lambda rho estimation} and \eqref{eq:proj Caug geometry increasing target dimension} give $\hat a_{ii}^{\rm aug}\convp\aiiaug(\delta)$. If $(\aiiaug(\delta),\delta)\ne(0,0)$, we therefore obtain
\begin{align*}
\hat t_{i,\delta}^{\rm aug}
&\convp t_{\rho_i,\delta}^*\bigl(\aiiaug(\delta)\bigr),\\
\inner{\uhiJS(\hat t_{i,\delta}^{\rm aug}),\uv_i}^2
&\convp
F_{\rho_i,\delta}\bigl(t_{\rho_i,\delta}^*(\aiiaug(\delta));\aiiaug(\delta)\bigr)
\ge\rho_i^2,
\end{align*}
with strict inequality if and only if $\aiiaug(\delta)>\delta$. If $\aiiaug(\delta)=\delta=0$, then \eqref{eq:proj Caug geometry increasing target dimension} gives $\|\projCiaug\uh_i\|\convp0$, while the definition of $\hat t_{i,\delta}^{\rm aug}$ gives $(1-\hat t_{i,\delta}^{\rm aug})^{-1}=O_p(1)$. Repeating the argument for the unaugmented case $a_{ii}=\delta=0$ yields
\begin{equation*}
\norm{\uhiJS(\hat t_{i,\delta}^{\rm aug})-\uh_i}\convp0,
\end{equation*}
so the limiting squared alignment is $\rho_i^2$. This proves Proposition~\ref{prop adjusted JS augmented}.
\end{proof}

\subsection{Proofs and technical details for Section~\ref{sec:comparison}}

This section proves the comparison results across asymptotic regimes in Section~\ref{sec:comparison}. For the comparison with the HL estimator under the RMT regime in Section~\ref{subsec:comparison HL RMT}, we first establish the equivalent representations in \eqref{eq:estimator_equiv_reprn}, which express both estimated eigenspaces in terms of $\Pv_\Cc\Uvh_m$ and $\Rv=(\Iv_p-\Pv_\Cc)\Uvh_m$. 

\begin{lemma}\label{lem:estimator equivalent representation}
Suppose $\hat\Lambdav_m$ and $\Rv^\top\Rv$ are invertible. If the truncation of $\hat t_i$ in \eqref{eq:ti hat def} is omitted, the resulting James--Stein eigenspace and the HL estimator satisfy
\begin{equation*}
\begin{aligned}
\UchJS_m
&=
\operatorname{col}\bracket{
\Pv_\Cc\Uvh_m+
\Rv\curly{\Iv_m-(\Rv^\top\Rv)^{-1}(\Iv_m-\hat\Dv^2)}
}, \\
\UchHL_m
&=
\operatorname{col}\bracket{
\Pv_\Cc\Uvh_m+
\Rv
\paren{\Iv_m-(\Rv^\top\Rv)^{-1}\bar\lambda\hat\Lambdav_m^{-1}}
}.
\end{aligned}
\end{equation*}
\end{lemma}
\begin{proof}
We first consider $\UchJS_m$. Fix $i\in[m]$. Since
$
\Ciaug
=
\Cc\oplus\operatorname{col}(\Rvmi),
$
the residual from projecting $\uh_i$ onto $\Ciaug$ can be written as
$
\uh_i-\projCiaug\uh_i
=
\rv_i-\proj_{\operatorname{col}(\Rvmi)}\rv_i.
$
Hence $\uh_i-\projCiaug\uh_i\in\operatorname{col}(\Rv)$ and is orthogonal to $\rv_j$ for every $j\ne i$. Moreover,
\begin{equation*}
\inner{\rv_i,\uh_i-\projCiaug\uh_i}
=
\norm{\uh_i-\projCiaug\uh_i}^2.
\end{equation*}
Therefore,
\begin{equation*}
\Rv^\top(\uh_i-\projCiaug\uh_i)
=
\norm{\uh_i-\projCiaug\uh_i}^2\ev_i.
\end{equation*}
Since $\uh_i-\projCiaug\uh_i\in\operatorname{col}(\Rv)$, write
$
\uh_i-\projCiaug\uh_i=\Rv\yv
$
for some $\yv\in\Rb^m$. Then
\begin{equation*}
\Rv^\top\Rv\yv
=
\Rv^\top(\uh_i-\projCiaug\uh_i)
=
\norm{\uh_i-\projCiaug\uh_i}^2\ev_i,
\end{equation*}
and hence
\begin{equation}
\uh_i-\projCiaug\uh_i
=\paren{1-\norm{\projCiaug\uh_i}^2}\Rv(\Rv^\top\Rv)^{-1}\ev_i.
\label{eq:pf:augmented residual representation}
\end{equation}
Thus, using the definition of $\hat t_i$ in \eqref{eq:ti hat def},
\begin{align*}
\utilJS_i
=
\uh_i-\hat t_i(\uh_i-\projCiaug\uh_i)
=
\uh_i
-
(1-\rhoh_i^2)
\Rv(\Rv^\top\Rv)^{-1}\ev_i.
\end{align*}
Stacking these identities over $i\in[m]$ gives
\begin{align*}
[\utilJS_1,\ldots,\utilJS_m]
&=
\Uvh_m
-
\Rv(\Rv^\top\Rv)^{-1}(\Iv_m-\hat\Dv^2)\\
&=
\Pv_\Cc\Uvh_m
+
\Rv\curly{
\Iv_m-(\Rv^\top\Rv)^{-1}(\Iv_m-\hat\Dv^2)
},
\end{align*}
where we used $\Uvh_m=\Pv_\Cc\Uvh_m+\Rv$. Since normalization of the individual columns does not change their span,
\begin{equation*}
\UchJS_m
=
\operatorname{col}\bracket{
\Pv_\Cc\Uvh_m
+
\Rv\curly{
\Iv_m-(\Rv^\top\Rv)^{-1}(\Iv_m-\hat\Dv^2)
}
}.
\end{equation*}

For the HL estimator, recall from \eqref{eq:UchHL} that
\begin{equation*}
\UchHL_m
=
\operatorname{col}\curly{
(\Uvh_m\hat\Lambdav_m\Uvh_m^\top-\bar\lambda\Iv_p)\Rv
}.
\end{equation*}
Since $
\Uvh_m^\top\Rv
=
\Uvh_m^\top(\Iv_p-\Pv_\Cc)\Uvh_m
=
\Rv^\top\Rv,
$
we have
\begin{equation*}
(\Uvh_m\hat\Lambdav_m\Uvh_m^\top-\bar\lambda\Iv_p)\Rv
=
\Uvh_m\hat\Lambdav_m\Rv^\top\Rv-\bar\lambda\Rv.
\end{equation*}
On the other hand, using $\Uvh_m=\Pv_\Cc\Uvh_m+\Rv$ again,
\begin{align*}
&\left[
\Pv_\Cc\Uvh_m
+
\Rv\curly{
\Iv_m-(\Rv^\top\Rv)^{-1}
\bar\lambda\hat\Lambdav_m^{-1}
}
\right]
\hat\Lambdav_m\Rv^\top\Rv\\
&\qquad=
\Uvh_m\hat\Lambdav_m\Rv^\top\Rv-\bar\lambda\Rv.
\end{align*}
Since $\hat\Lambdav_m\Rv^\top\Rv$ is invertible, right multiplication by this matrix does not change the column space. Therefore,
\begin{equation*}
\UchHL_m
=
\operatorname{col}\bracket{
\Pv_\Cc\Uvh_m+
\Rv
\paren{
\Iv_m-(\Rv^\top\Rv)^{-1}
\bar\lambda\hat\Lambdav_m^{-1}
}
}.
\end{equation*}
\end{proof}

\begin{proof}[Proof of Proposition~\ref{prop:compare UchHL UchJS RMT}]
By Lemma~\ref{lem:estimator equivalent representation}, $\UchHL_m\subset\Sc$. Hence Lemma~\ref{lem:similarity tilde Uc} gives
\begin{equation*}
S(\UchHL_m,\Uc_m)^2
\le
S(\Uchorc_m,\Uc_m)^2.
\end{equation*}
On the other hand, Theorem~\ref{thm:main thm-1}-(i) gives
$S(\UchJS_m,\Uchorc_m)\convp1$, so $\UchJS_m$ and $\Uchorc_m$ have the same limiting similarity with $\Uc_m$. Thus, to prove part~(i), it suffices to show that the preceding inequality is strict in the limit whenever $a_{ii}>0$ for some $i\in[m]$.

We first establish the limit of $\bar\lambda$. By \eqref{eq:pf: sample ESD weak convergence}, \eqref{eq:pf:sample bulk separation}, and the Helly--Bray theorem,
\begin{equation*}
\frac{1}{p-m}\sum_{j=m+1}^p\lambdah_j
\convp
\int t\,dF_{\gamma,H}(t)
=
\int t\,dH(t),
\end{equation*}
where the last equality follows from Lemma~2.16 of \cite{yao2015sample}. Therefore,
\begin{equation}
\bar\lambda
=
\frac{1}{n-m}\sum_{j=m+1}^p\lambdah_j
\convp
\beta\coloneq\gamma\int t\,dH(t)>0.
\label{eq:pf:HL bulk limit}
\end{equation}
Since $\rho_i^2=\lambda_i\psi'(\lambda_i)/\psi(\lambda_i)$,
\begin{align}
\psi(\lambda_i)(1-\rho_i^2)-\beta
&=
\psi(\lambda_i)-\lambda_i\psi'(\lambda_i)-\beta
\nonumber\\
&=
\gamma\int t\left\{
\frac{\lambda_i^2}{(\lambda_i-t)^2}-1
\right\}\,dH(t)
>0,
\qquad i\in[m].
\label{eq:pf:HL shrinkage gap}
\end{align}
Hence
\begin{equation*}
\frac{1-\rho_i^2}{1-\rho_i^2\aiiaug}=\plim_{n,p\to\infty}\hat{t}_i>\plim_{n,p\to\infty}\hat{t}_i^{\rm HL}=\frac{\beta/\psi(\lambda_i)}{1-\rho_i^2\aiiaug}.
\end{equation*}
Recall from \eqref{eq:ti star def} that when $\aiiaug>0$, $t_i^*=(1-\rho_i^2)/(1-\rho_i^2\aiiaug)$ is the unique shrinkage parameter which maximizes the limit of the squared alignment $\langle\uhJS_i(t),\uv_i\rangle^2$. Thus, $\hat{t}_i^{\rm HL}$ underestimates the optimal shrinkage parameter $t_i^*$, which in turn means that $\uhJS_i(\hat{t}_i^{\rm HL})$ has a lower squared alignment with $\uv_i$ than $\uhJS_i(\hat{t}_i)$.

We next establish strict improvement in eigenspace similarity. We first verify that the squared norm $\|\proj_{\UchHL_m}\uv_i\|^2$ has a deterministic probability limit for every $i\in[m]$.
By Lemma~\ref{lem:pf:RtR RtUm WtW lim}, $\Rv^\top\Rv\convp\Iv_m-\Dv\Av\Dv\succ\0v$, so the HDLSS representation in Lemma~\ref{lem:estimator equivalent representation} gives
\begin{equation}
\UchHL_m=\operatorname{col}(\Hv),
\qquad
\Hv
=
\Uvh_m-\bar\lambda\Rv(\Rv^\top\Rv)^{-1}\hat\Lambdav_m^{-1}.
\label{eq:pf:HL generating matrix}
\end{equation}
The identities $\Pv_\Cc\Rv=\0v$ and $\Uvh_m^\top\Rv=\Rv^\top\Rv$ imply
\begin{equation}
\Pv_\Cc\Hv=\Pv_\Cc\Uvh_m,
\qquad
\Uvh_m^\top\Hv=\Iv_m-\bar\lambda\hat\Lambdav_m^{-1}.
\label{eq:pf:HL projection identities}
\end{equation}
Expanding the products in \eqref{eq:pf:HL generating matrix} and using $\Uvh_m^\top\Rv=\Rv^\top\Rv$, we obtain
\begin{align*}
\Hv^\top\Hv
&=
\Iv_m-2\bar\lambda\hat\Lambdav_m^{-1}
+\bar\lambda^2\hat\Lambdav_m^{-1}
(\Rv^\top\Rv)^{-1}\hat\Lambdav_m^{-1},\\
\Hv^\top\Uv_m
&=
\Uvh_m^\top\Uv_m
-\bar\lambda\hat\Lambdav_m^{-1}
(\Rv^\top\Rv)^{-1}\Rv^\top\Uv_m.
\end{align*}
Writing $\Dv_\psi=\diag(\psi(\lambda_1),\ldots,\psi(\lambda_m))$, Lemmas~\ref{lem:eigenpair asymptotics RMT} and \ref{lem:pf:RtR RtUm WtW lim}, together with \eqref{eq:pf:HL bulk limit}, therefore give
\begin{align}
\Hv^\top\Hv
&\convp
\Iv_m-2\beta\Dv_\psi^{-1}
+\beta^2\Dv_\psi^{-1}
(\Iv_m-\Dv\Av\Dv)^{-1}\Dv_\psi^{-1},
\nonumber\\
\Hv^\top\Uv_m
&\convp
\Dv-\beta\Dv_\psi^{-1}
(\Iv_m-\Dv\Av\Dv)^{-1}\Dv(\Iv_m-\Av).
\label{eq:pf:HL matrix limits}
\end{align}

To verify that the limit of $\Hv^\top\Hv$ is positive definite, note that $\beta/\psi(\lambda_i)<1$ for every $i\in[m]$ by \eqref{eq:pf:HL shrinkage gap}. Since $\Uvh_m\Uvh_m^\top\preceq\Iv_p$, \eqref{eq:pf:HL projection identities} gives
\begin{equation*}
\Hv^\top\Hv
\succeq
(\Uvh_m^\top\Hv)^\top(\Uvh_m^\top\Hv)
=
(\Iv_m-\bar\lambda\hat\Lambdav_m^{-1})^2
\convp
(\Iv_m-\beta\Dv_\psi^{-1})^2
\succ\0v.
\end{equation*}
Thus, the limit of $\Hv^\top\Hv$ in \eqref{eq:pf:HL matrix limits} is positive definite. For each $i\in[m]$, the projection formula
\begin{equation*}
\norm{\proj_{\UchHL_m}\uv_i}^2
=
(\Hv^\top\uv_i)^\top
(\Hv^\top\Hv)^{-1}
\Hv^\top\uv_i
\end{equation*}
and \eqref{eq:pf:HL matrix limits} now show that $\|\proj_{\UchHL_m}\uv_i\|^2$ has a deterministic probability limit. By \eqref{eq:pf:oracle eigenspace similarity} and the definition of $S$,
\begin{equation}
m\left\{
S(\Uchorc_m,\Uc_m)^2-S(\UchHL_m,\Uc_m)^2
\right\}
=
\sum_{i=1}^m
\left\{
\norm{\projS\uv_i}^2
-
\norm{\proj_{\UchHL_m}\uv_i}^2
\right\}.
\label{eq:pf:HL oracle gap}
\end{equation}
Each summand is nonnegative because $\UchHL_m\subset\Sc$, and has a deterministic probability limit by \eqref{eq:pf:UtPSU lim new} and \eqref{eq:pf:HL matrix limits}.

Fix $i\in[m]$ with $a_{ii}>0$. We show that the $i$th summand in \eqref{eq:pf:HL oracle gap} has a strictly positive limit. Suppose, to the contrary, that this summand converges to zero, and let
$
\yv_i=(\Hv^\top\Hv)^{-1}\Hv^\top\uv_i.
$
Then $\Hv\yv_i=\proj_{\UchHL_m}\uv_i$. Since $\UchHL_m\subset\Sc$, the Pythagorean theorem gives
\begin{equation}
\norm{\projS\uv_i-\Hv\yv_i}^2
=
\norm{\projS\uv_i}^2
-
\norm{\proj_{\UchHL_m}\uv_i}^2
\convp0.
\label{eq:pf:HL oracle projection equality}
\end{equation}

We first identify the coefficients $\yv_i$ using the projection onto $\Uch_m$. Since $\Uch_m\subset\Sc$, we have $\Uvh_m^\top\Pv_\Sc=\Uvh_m^\top$. Hence \eqref{eq:pf:HL projection identities} and \eqref{eq:pf:HL oracle projection equality} imply
\begin{align*}
\norm{
\Uvh_m^\top\uv_i
-
(\Iv_m-\bar\lambda\hat\Lambdav_m^{-1})\yv_i
}
&=
\norm{
\Uvh_m^\top(\projS\uv_i-\Hv\yv_i)
}\\
&\le
\norm{\projS\uv_i-\Hv\yv_i}
\convp0.
\end{align*}
The matrix $\Iv_m-\beta\Dv_\psi^{-1}$ is invertible by \eqref{eq:pf:HL shrinkage gap}. Thus, Lemma~\ref{lem:eigenpair asymptotics RMT} and \eqref{eq:pf:HL bulk limit} give
\begin{equation}
\yv_i
\convp
(\Iv_m-\beta\Dv_\psi^{-1})^{-1}\Dv\ev_i
=
\frac{\rho_i}{1-\beta/\psi(\lambda_i)}\ev_i.
\label{eq:pf:HL projection coefficient limit}
\end{equation}

We now compare the projections onto $\Cc$. Since $\Cc\subset\Sc$ and $\Pv_\Cc\Hv=\Pv_\Cc\Uvh_m$ by \eqref{eq:pf:HL projection identities},
\begin{align}
\norm{\projC\uv_i-\Pv_\Cc\Uvh_m\yv_i}
&=
\norm{\Pv_\Cc(\projS\uv_i-\Hv\yv_i)}
\nonumber\\
&\le
\norm{\projS\uv_i-\Hv\yv_i}
\convp0.
\label{eq:pf:HL target approximation}
\end{align}
On the other hand, \eqref{eq:pf:HL projection coefficient limit} and Lemma~\ref{lem:projC uhi RMT} give
\begin{equation*}
\norm{
\Pv_\Cc\Uvh_m\yv_i
-
\frac{\rho_i^2}{1-\beta/\psi(\lambda_i)}\projC\uv_i
}
\convp0.
\end{equation*}
Consequently, Assumption~\ref{ass:target} yields
\begin{equation*}
\norm{\projC\uv_i-\Pv_\Cc\Uvh_m\yv_i}^2
\convp
\left\{
1-\frac{\rho_i^2}{1-\beta/\psi(\lambda_i)}
\right\}^2a_{ii}
>0,
\end{equation*}
where positivity follows from \eqref{eq:pf:HL shrinkage gap} and $a_{ii}>0$. This contradicts \eqref{eq:pf:HL target approximation}. Therefore, the $i$th summand in \eqref{eq:pf:HL oracle gap} has a strictly positive limit whenever $a_{ii}>0$.

If $a_{ii}>0$ for some $i\in[m]$, summing in \eqref{eq:pf:HL oracle gap}, using \eqref{eq:pf:oracle eigenspace similarity} and \eqref{eq:pf:augmented eigenspace similarity}, and taking square roots gives
\begin{equation*}
\plim_{n,p\to\infty}S(\UchJS_m,\Uc_m)
=
\plim_{n,p\to\infty}S(\Uchorc_m,\Uc_m)
>
\plim_{n,p\to\infty}S(\UchHL_m,\Uc_m).
\end{equation*}
This proves part~(i).

For part~(ii), suppose $a_{ii}=0$ for every $i\in[m]$. Since $\Av\succeq\0v$, we have $\Av=\0v$. Lemmas~\ref{lem:projC uhi RMT} and \ref{lem:pf:RtR RtUm WtW lim} then give
\begin{equation*}
\norm{\Rv-\Uvh_m}_F
=
\norm{\Pv_\Cc\Uvh_m}_F
\convp0,
\qquad
\Rv^\top\Rv\convp\Iv_m.
\end{equation*}
Together with \eqref{eq:pf:HL bulk limit} and $\hat\Lambdav_m\convp\Dv_\psi$, this implies
$
\bar\lambda(\Rv^\top\Rv)^{-1}\hat\Lambdav_m^{-1}
\convp
\beta\Dv_\psi^{-1}.
$
The representation \eqref{eq:pf:HL generating matrix} therefore yields
\begin{equation*}
\norm{
\Hv-\Uvh_m(\Iv_m-\beta\Dv_\psi^{-1})
}_F
\convp0.
\end{equation*}
The columns of $\Uvh_m(\Iv_m-\beta\Dv_\psi^{-1})$ span $\Uch_m$ and have Gram matrix $(\Iv_m-\beta\Dv_\psi^{-1})^2\succ\0v$ by \eqref{eq:pf:HL shrinkage gap}. Lemma~\ref{lem:subspace perturbation}-(ii) and Theorem~\ref{thm:main thm-1}-(iii) therefore give
\begin{equation*}
S(\UchHL_m,\UchJS_m)
=
S(\Uch_m,\UchJS_m)+o_p(1)
\convp1,
\end{equation*}
proving part~(ii).
\end{proof}

We next turn to the asymptotic regimes considered in Section~\ref{sec:other asymptotic regimes}, beginning with the UHD regime. Our analysis of the UHD regime builds on the eigenpair asymptotics established in \cite{Yata2012}. For $i\in[m]$, define
\begin{equation*}
c_i
=
\lim_{n,p\to\infty}\frac{\tau^2p}{n\lambda_i}
=
\frac{\tau^2c}{\sigma_i^2}
\in(0,\infty),
\qquad
\rho_i^{\mathrm{UHD}}=(1+c_i)^{-1/2}.
\end{equation*}
The following lemma provides the UHD counterparts of Lemmas~\ref{lem:eigenpair asymptotics RMT}, \ref{lem:projC uhi RMT}, and \ref{lem:lambda rho estimation}. Consequently, the eigenvector estimation results under the UHD regime take the same form as those under the RMT regime, with $\rho_i$ replaced by $\rho_i^{\mathrm{UHD}}$.

\begin{lemma}\label{lem:eigenpair asymptotics UHD}
Suppose Assumptions~\ref{ass:spiked population model HL,UHD} and \ref{ass:target} hold under the UHD regime. Then, for every $i\in[m]$,
\begin{equation*}
\lambdah_i/\lambda_i\convp1+c_i,
\qquad
\inner{\uh_i,\uv_i}\convp\rho_i^{\mathrm{UHD}},
\end{equation*}
and $\langle\uh_i,\uv_j\rangle\convp0$ for every fixed $j\ne i$. Moreover,
\begin{equation*}
\rhoh_i\convp\rho_i^{\mathrm{UHD}},
\qquad
\norm{\projC\uh_i-\rho_i^{\mathrm{UHD}}\projC\uv_i}\convp0.
\end{equation*}
\end{lemma}

\begin{proof}
Fix $i\in[m]$. Lemma~5 of \cite{Yata2012} gives
$\lambdah_i/\lambda_i\convp1+c_i$.
Together with Assumption~\ref{ass:spiked population model HL,UHD}-(a),
Lemma~2 of \cite{Yata2012} and Weyl's inequality yield
$n\lambdah_{m+1}/p\convp\tau^2$, while Lemma~7 of the paper gives
$n\bar\lambda/p\convp\tau^2$.
Since $\tau^2p/(n\lambda_i)\to c_i$, it follows that
\begin{equation}
\frac{\bar\lambda}{\lambdah_i}\convp\frac{c_i}{1+c_i},
\qquad
\frac{\lambdah_{m+1}}{\lambdah_i}\convp\frac{c_i}{1+c_i}.
\label{eq:pf:UHD bulk to spike ratios}
\end{equation}

We use these eigenvalue limits to prove the consistency of $\rhoh_i$. Since $p>n$ eventually and $\lambdah_j=0$ for $j>n$, the definition of $\tilde\lambda_i$ gives
\begin{equation*}
\tilde\lambda_i^{-1}
=
\frac{1}{n-m}\sum_{j=m+1}^n
\frac{1}{\lambdah_i-\lambdah_j}.
\end{equation*}
Substituting into the definition of $\rhoh_i$ in \eqref{eq:rhoh_i def}, we obtain
\begin{equation}
\rhoh_i^{-2}
=
1+
\frac{
\displaystyle\sum_{j=m+1}^n
\frac{\lambdah_j}{(\lambdah_i-\lambdah_j)^2}
}{
\displaystyle\sum_{j=m+1}^n
\frac{1}{\lambdah_i-\lambdah_j}
}
=
\lambdah_i
\frac{
\displaystyle\sum_{j=m+1}^n
\frac{1}{(\lambdah_i-\lambdah_j)^2}
}{
\displaystyle\sum_{j=m+1}^n
\frac{1}{\lambdah_i-\lambdah_j}
}.
\label{eq:pf:UHD rhoh representation}
\end{equation}
Since $\bar\lambda=(n-m)^{-1}\sum_{j=m+1}^n\lambdah_j$,
Jensen's and Cauchy--Schwarz inequalities give, respectively,
\begin{equation*}
\frac{\lambdah_i}{\lambdah_i-\bar\lambda}
\le
\frac{\lambdah_i}{n-m}
\sum_{j=m+1}^n\frac{1}{\lambdah_i-\lambdah_j}
\le
\rhoh_i^{-2}.
\end{equation*}
For an upper bound on $\rhoh_i^{-2}$, $\lambdah_j\le\lambdah_{m+1}$ for
$j=m+1,\ldots,n$ implies
\begin{equation*}
\sum_{j=m+1}^n\frac{1}{(\lambdah_i-\lambdah_j)^2}
\le
\frac{1}{\lambdah_i-\lambdah_{m+1}}
\sum_{j=m+1}^n\frac{1}{\lambdah_i-\lambdah_j}.
\end{equation*}
Thus, \eqref{eq:pf:UHD rhoh representation} yields
$\rhoh_i^{-2}\le
\lambdah_i/(\lambdah_i-\lambdah_{m+1})$. Both bounds for $\rhoh_i^{-2}$ converge in probability to $1+c_i$ by \eqref{eq:pf:UHD bulk to spike ratios}. Hence $\rhoh_i\convp(1+c_i)^{-1/2}=\rho_i^{\mathrm{UHD}}$.

We next establish the eigenvector limits. Let $\vh_i$ be the unit right singular vector of $\Xv$ corresponding to $\uh_i$, so that $\uh_i=(n\lambdah_i)^{-1/2}\Xv\vh_i$, and let $\zv_j'\in\Rb^n$ denote the transpose of the $j$th row of $\Zv$. Lemma~5 of \cite{Yata2012} gives $\|\vh_i-\zv_i'/\|\zv_i'\|\|\convp0$. Since $\Xv=\sum_{j=1}^p\sqrt{\lambda_j}\,\uv_j{\zv_j'}^\top$, we have
\begin{equation}
\inner{\uh_i,\uv_j}
=
\sqrt{\frac{\lambda_j}{\lambdah_i}}\,
\frac{{\zv_j'}^\top\vh_i}{\sqrt n},
\qquad j\in[p].
\label{eq:pf:UHD sample population inner product}
\end{equation}
For every fixed $j$, the law of large numbers gives ${\zv_j'}^\top\zv_i'/n\convp\delta_{ij}$ and $\|\zv_j'\|/\sqrt n\convp1$. Thus, by the Cauchy--Schwarz inequality,
\begin{align*}
\frac{{\zv_j'}^\top\vh_i}{\sqrt n}
&=
\frac{{\zv_j'}^\top\zv_i'}{n}
\left(\frac{\norm{\zv_i'}}{\sqrt n}\right)^{-1}
+
\frac{{\zv_j'}^\top}{\sqrt n}
\left(
\vh_i-\frac{\zv_i'}{\norm{\zv_i'}}
\right)
\convp\delta_{ij}.
\end{align*}
Since $\lambda_j/\lambdah_i=O_p(1)$ for every fixed $j$ and $\lambda_i/\lambdah_i\convp(1+c_i)^{-1}$, \eqref{eq:pf:UHD sample population inner product} yields $\langle\uh_i,\uv_i\rangle\convp\rho_i^{\mathrm{UHD}}$ and $\langle\uh_i,\uv_j\rangle\convp0$ for every fixed $j\ne i$.

Finally, to prove $\|\projC\uh_i-\rho_i^{\mathrm{UHD}}\projC\uv_i\|\convp0$, decompose
\begin{equation*}
\projC\uh_i-\rho_i^{\mathrm{UHD}}\projC\uv_i
=
\sum_{j=1}^m
\left\{
\inner{\uh_i,\uv_j}-\rho_i^{\mathrm{UHD}}\delta_{ij}
\right\}\projC\uv_j
+
\Pv_\Cc\Pv_{\Uc_m^\perp}\uh_i.
\end{equation*}
The norm of the sum converges to zero in probability by the eigenvector limits just established, since $m$ is fixed and $\|\projC\uv_j\|\le1$. To control the remaining term, the spectral decomposition of $\Sv$ gives $\lambdah_i\uh_i\uh_i^\top\preceq\Sv$, and hence
\begin{equation*}
\lambdah_i\norm{\Pv_\Cc\Pv_{\Uc_m^\perp}\uh_i}^2
\le
\tr\left(
\Pv_\Cc\Pv_{\Uc_m^\perp}\Sv\Pv_{\Uc_m^\perp}
\right).
\end{equation*}
Since $\E\Sv=\Sigmav$ and $\tr(\Pv_\Cc)=r$, we have
\begin{align*}
\E\tr\left(
\Pv_\Cc\Pv_{\Uc_m^\perp}\Sv\Pv_{\Uc_m^\perp}
\right)
&=
\tr\left(
\Pv_\Cc\sum_{j=m+1}^p\lambda_j\uv_j\uv_j^\top
\right)\\
&=
\sum_{j=m+1}^p\lambda_j\norm{\projC\uv_j}^2
\le
r\lambda_{m+1}
=
O(1).
\end{align*}
$\tr\left(
\Pv_\Cc\Pv_{\Uc_m^\perp}\Sv\Pv_{\Uc_m^\perp}
\right)$ is therefore $O_p(1)$ by Markov's inequality. Since $\lambdah_i/\lambda_i\convp1+c_i$ and $\lambda_i\to\infty$, we obtain $\|\Pv_\Cc\Pv_{\Uc_m^\perp}\uh_i\|\convp0$, which proves $\|\projC\uh_i-\rho_i^{\mathrm{UHD}}\projC\uv_i\|\convp0$.
\end{proof}

We now turn to the HDLSS regime. The next lemma summarizes the asymptotic behavior of the sample eigenvalues and eigenvectors established in Theorems~2--3 of \cite{Jung2012a}. Although these theorems are stated in terms of convergence in distribution, their proof actually establishes convergence in probability, as stated below.

\begin{lemma}\label{lem:eigenpair asymptotics HL}
Suppose Assumption~\ref{ass:spiked population model HL,UHD} holds under the HDLSS regime. There exist continuous random variables
$W_1\ge\cdots\ge W_m>0$ such that, for every $i\in[m]$,
\begin{equation*}
\frac{n}{p}\lambdah_i
\convp
W_i+\tau^2,
\qquad
\cos\angle(\uh_i,\Uc_m)
\convp
\paren{1+\frac{\tau^2}{W_i}}^{-1/2}.
\end{equation*}
Moreover, for every fixed $j\in\{m+1,\ldots,n\}$, $n\lambdah_j/p\convp\tau^2$.
\end{lemma}
\begin{proof}[Proof of Proposition~\ref{prop:multi spike other regimes}]
Under the UHD regime, Lemma~\ref{lem:eigenpair asymptotics UHD} allows the proofs of Theorems~\ref{thm:uhJS limit} and \ref{thm:main thm-1} to be applied with $\rho_i$ replaced by $\rho_i^{\mathrm{UHD}}$. This proves $S(\UchJS_m,\Uchorc_m)\convp1$ and parts~(ii)--(iii) under the UHD regime. Moreover, \eqref{eq:pf:UHD bulk to spike ratios} gives $\bar\lambda/\lambdah_i-(1-\rhoh_i^2)\convp0$. Together with Lemmas~\ref{lem:subspace perturbation} and \ref{lem:estimator equivalent representation}, we have $S(\UchJS_m,\UchHL_m)\convp1$. 

Under the HDLSS regime, Lemma~\ref{lem:eigenpair asymptotics HL} gives $n\bar\lambda/p\convp\tau^2$. Together with \eqref{eq:pf:UHD rhoh representation}, the eigenvalue limits in Lemma~\ref{lem:eigenpair asymptotics HL} imply
\begin{equation*}
1-\rhoh_i^2-\frac{\bar\lambda}{\lambdah_i}\convp0,
\qquad
\frac{\bar\lambda}{\lambdah_i}\convp\frac{\tau^2}{W_i+\tau^2}\in(0,1).
\end{equation*}

To account for truncation in $\hat t_i$, we first bound $\Rv^\top\Rv$ from below. Let $\zv_j'\in\Rb^n$ denote the transpose of the $j$th row of $\Zv$. Then, we have
\begin{align*}
\frac{1}{p}\Xv^\top\Pv_{\Uc_m^\perp}\Xv
&=\frac{1}{p}\sum_{j=m+1}^p
\lambda_j\zv_j'{\zv_j'}^\top
\convp\tau^2\Iv_n,\\
\E\|\Pv_\Cc\Pv_{\Uc_m^\perp}\Xv\|_F^2
&=n\sum_{j=m+1}^p\lambda_j\|\projC\uv_j\|^2
\le nr\lambda_{m+1}=O(1).
\end{align*}
Here, the probability limit uses the weighted weak law of large numbers. Let $\hat\Vv_m=[\vh_1,\ldots,\vh_m]$ contain the right singular vectors of $\Xv$, so that
$
\Uvh_m=\Xv\hat\Vv_m(n\hat\Lambdav_m)^{-1/2}. 
$
Using the preceding results, Markov's inequality, and
$n\lambdah_m/p\convp W_m+\tau^2>0$, we obtain
\begin{align*}
\Uvh_m^\top\Pv_{\Uc_m^\perp}\Uvh_m
&=\frac{1}{n}\hat\Lambdav_m^{-1/2}\hat\Vv_m^\top
(\Xv^\top\Pv_{\Uc_m^\perp}\Xv)
\hat\Vv_m\hat\Lambdav_m^{-1/2}\\
&=\frac{\tau^2p}{n}\hat\Lambdav_m^{-1}+o_p(1)
=\bar\lambda\hat\Lambdav_m^{-1}+o_p(1),\\
\|\Pv_\Cc\Pv_{\Uc_m^\perp}\Uvh_m\|_F
&\le
\frac{\|\Pv_\Cc\Pv_{\Uc_m^\perp}\Xv\|_F}
{\sqrt{n\lambdah_m}}
\convp0.
\end{align*}
Thus, $\|\Pv_\Cc\Uvh_m-\Pv_\Cc\Pv_{\Uc_m}\Uvh_m\|_F=o_p(1)$. Since $\Pv_\Cc\preceq\Iv_p$, it follows that
\begin{align*}
\Rv^\top\Rv
&=\Iv_m-(\Pv_\Cc\Uvh_m)^\top(\Pv_\Cc\Uvh_m)
=\Iv_m-\Uvh_m^\top\Pv_{\Uc_m}\Pv_\Cc\Pv_{\Uc_m}\Uvh_m+o_p(1)\\
&\succeq\Iv_m-\Uvh_m^\top\Pv_{\Uc_m}\Uvh_m+o_p(1)
=\Uvh_m^\top\Pv_{\Uc_m^\perp}\Uvh_m+o_p(1)
=\bar\lambda\hat\Lambdav_m^{-1}+o_p(1).
\end{align*}
Because each $\bar\lambda/\lambdah_i$ has a positive probability limit, we have $\|(\Rv^\top\Rv)^{-1}\|=O_p(1)$ and
$(\Rv^\top\Rv)^{-1}\preceq\bar\lambda^{-1}\hat\Lambdav_m+o_p(1)$.
Taking squared norms in \eqref{eq:pf:augmented residual representation}, we obtain
\begin{equation*}
1-\norm{\projCiaug\uh_i}^2
=\left\{\bigl[(\Rv^\top\Rv)^{-1}\bigr]_{ii}\right\}^{-1}
\ge\frac{\bar\lambda}{\lambdah_i}+o_p(1).
\end{equation*}
Combining this lower bound with $1-\rhoh_i^2=\bar\lambda/\lambdah_i+o_p(1)$ and the definition of $\hat t_i$, we obtain
\begin{equation*}
\hat t_i(1-\norm{\projCiaug\uh_i}^2)
=\min\{1-\rhoh_i^2,1-\norm{\projCiaug\uh_i}^2\}
=\frac{\bar\lambda}{\lambdah_i}+o_p(1).
\end{equation*}
By \eqref{eq:pf:augmented residual representation} and the definition of $\utilJS_i$,
\begin{equation*}
\utilJS_i
=\uh_i-\hat t_i(\uh_i-\projCiaug\uh_i)
=\uh_i-\hat t_i\paren{1-\norm{\projCiaug\uh_i}^2}
\Rv(\Rv^\top\Rv)^{-1}\ev_i.
\end{equation*}
Since $\|\Rv\|\le1$ and $\|(\Rv^\top\Rv)^{-1}\|=O_p(1)$, we have
\begin{equation*}
\bigl\|[\utilJS_1,\ldots,\utilJS_m]-\Hv\bigr\|_F\convp0,
\qquad
\Hv=\Uvh_m-\bar\lambda\Rv(\Rv^\top\Rv)^{-1}\hat\Lambdav_m^{-1},
\end{equation*}
where $\operatorname{col}(\Hv)=\UchHL_m$ by Lemma~\ref{lem:estimator equivalent representation}.

To pass to the column spaces, note that
$\Uvh_m^\top\Hv=\Iv_m-\bar\lambda\hat\Lambdav_m^{-1}$ and hence
\begin{equation*}
\Hv^\top\Hv
\succeq(\Uvh_m^\top\Hv)^\top(\Uvh_m^\top\Hv)
=(\Iv_m-\bar\lambda\hat\Lambdav_m^{-1})^2.
\end{equation*}
Since $1-\bar\lambda/\lambdah_i\convp W_i/(W_i+\tau^2)>0$
for every $i\in[m]$, the argument in the proof of
Lemma~\ref{lem:subspace perturbation}-(i) gives
$S(\UchJS_m,\UchHL_m)\convp1$. Theorem~7 of \cite{yoon2025adaptive} and its proof establish
$S(\UchHL_m,\Uchorc_m)\convp1$ and the conclusions of
parts~(ii)--(iii) with $\UchJS_m$ replaced by $\UchHL_m$.
These conclusions transfer to $\UchJS_m$ by the argument
in the proof of Lemma~\ref{lem:subspace perturbation}-(ii).
\end{proof}
\end{appendix}

\bibliographystyle{imsart-number} 
\bibliography{bibliography}       

\end{document}